\UseRawInputEncoding
\documentclass{amsart}
\usepackage{amsfonts,amssymb,amsmath,amsthm,amscd, url, enumerate, fullpage}
\newtheorem{theorem}{Theorem}[section]
\newtheorem{lemma}[theorem]{Lemma}
\newtheorem{proposition}[theorem]{Proposition}
\newtheorem{corollary}[theorem]{Corollary}
\theoremstyle{definition}

\newenvironment{remark}[1][Remark]{\begin{trivlist} \item[\hskip \labelsep {\bfseries #1}]}{\end{trivlist}}

\newcommand{\m}{\operatorname{mod}}
\newcommand{\Kl}{\operatorname{Kl}}
\newcommand{\GLn}{\operatorname{GL}_n}
\newcommand{\diag}{\operatorname{diag}}

\newcommand{\R}{ {\bf{R}}  }
\newcommand{\seven}{\sum_{\chi \m p^{\beta} \atop \operatorname{primitive}, \chi(-1)=1} }
\newcommand{\sevenp}{\sum_{\chi \m p \atop \operatorname{primitive}, \chi(-1) = 1} }
\newcommand{\seveny}{\sum_{\chi \m p^{\beta-y} \atop \operatorname{primitive}, \chi(-1) = 1} }
\newcommand{\K}{\mathfrak{K}}
\newcommand{\Res}{\operatorname{Res}}

\begin{document}

\title[Dirichlet twists of $\GLn$-automorphic $L$-functions]{Dirichlet twists of $\GLn$-automorphic $L$-functions and hyper-Kloosterman Dirichlet series}

\author{Jeanine Van Order}

\begin{abstract} 

We calculate mean values of $\GLn$-automorphic $L$-functions twisted by primitive even Dirichlet characters of prime-power conductor,
at arbitrary points within the critical strip, by derivation of special Voronoi summation formulae. Our calculation is novel in that the twisted sum 
can be expressed in terms of the average itself, and also that it sees the derivation of various new summation formulae in the setting of prime-power 
modulus. One consequence, as we explain, is to show the analytic continuation and additive summation formulae for hyper-Kloosterman Dirichlet series 
associated to $\GLn$-automorphic $L$-functions. \\

Nous calculons les valuers moyennes des fonctions $L$ automorphes sur $\GLn$ tordues par des caract\`eres de Dirichlet primitifs et pairs, 
du conducteur une puissance d'un nombre premier, \`a des points arbitraires dans la bande critique, en d\'erivant des formules de sommation 
sp\'eciales du type Voronoi. Notre calcul est nouveau car la somme est exprim\'e en termes de la moyenne elle-m\^eme, et aussi qu'il voit 
la d\'erivation de diverses nouvelles formules de sommation dans le regime des puissances d'un nombre premier. Une cons\'equence, 
comme nous l'expliquons, est de montrer les prolongations analytiques et des formules de sommation additive pour les s\'eries 
de Dirichlet hyper-Kloosterman associ\'ees aux fonctions $L$ automorphes sur $\GLn$.

\end{abstract}

\maketitle
\tableofcontents

\section{Introduction} 

Let $\pi = \otimes_v \pi_v$ be a cuspidal automorphic representation of $\GLn({\bf{A}}_{\bf{Q}})$ of conductor $N$ 
and unitary central character $\omega$ for $n \geq 2$. Suppose the achimedean component $\pi_{\infty}$ of $\pi$ is 
spherical and parametrized by a diagonal matrix $\diag(\mu_j)_{j=1}^n$. We consider the standard $L$-function
\begin{align*} \Lambda(s, \pi) &= L(s, \pi_{\infty})L(s, \pi)= \prod_v L(s, \pi_v) \end{align*}
of $\pi$, whose Euler factors $L(s, \pi_v)$ at an unramified places $v$ are given by the $n$-fold products
\begin{align*} L(s, \pi_v) &= \begin{cases} \prod_{j=1}^n \left( 1 - \alpha_j(\pi_v) v^{-s}\right)^{-1} &\text{ if $v$ is finite}\\
\prod_{j=1}^n \Gamma_{ {\R}}(s - \mu_j(\pi_v)) &\text{ if $v = \infty$ is the real place}, \end{cases} \end{align*} 
where the $(\alpha_j(\pi_v))_j$ and $(\mu_j(\pi_{\infty}))_j$ denote the corresponding Satake parameters of the local representations $\pi_v$.
More precisely, we shall consider twists $\Lambda(s, \pi \otimes \chi) = L(s, \pi) L(s, \pi \otimes \chi)$ of this standard $L$-function by primitive, 
even Dirichlet characters $\chi$ as follows. 

Fix a prime number $p$ which does not divide $N$, and let $\beta \geq 2$ be any integer. 
Let $\delta \in {\bf{C}}$ be any complex number inside the critical strip $0 < \Re(\delta) <1$. 
We derive various exact summation formulae in the style of Lavrik \cite{Lav} and Voronoi \cite{Vo} to describe the mean values
\begin{align*} X_{\beta}(\pi, \delta) &= \frac{2}{\varphi^{\star}(p^{\beta})} \seven L(\delta, \pi \otimes \chi), \end{align*}
where $\varphi^{\star}(p^{\beta}) = \varphi(p^{\beta}) - \varphi(p^{\beta-1})$ denotes the number of primitive Dirichlet characters $\chi \m p^{\beta}$, 
and the sum runs over all primitive even Dirichlet characters $\chi$ of conductor $p^{\beta}$. 
To be clear, we average over the finite parts of the completed $L$-functions $\Lambda(s, \pi \otimes \chi)$, 
whose archimedean components are each given by $L(s, \pi_{\infty})$ (independently of the choice of $\chi$), 
where the main difficulty and novelty is to compute the implicit polar term directly. 
We note that this average is of interest for several reasons, one being the applications to the generalized 
Ramanujan conjecture (at the real place) via the argument of Luo-Rudnick-Sarnak \cite[$\S 1$]{LRS}. 
To be more concrete, we derive the following formulae in terms of the $L$-function coefficients $a(m)$ of $\pi$. 
Let $W(\pi)$ denote the root number of $L(s, \pi)$, so that the functional equation for the standard $L$-function reads 
$\Lambda(s, \pi) = W(\pi)\Lambda(1-s, \widetilde{\pi})$. Fix a rational prime $p$ not diving $N$.
Given an integer $\beta \geq 1$ and a coprime class $c$ modulo $p^{\beta}$, consider the $n$-dimensional 
hyper-Kloosterman sum of modulus $p^{\beta}$ evaluated at $c$:
\begin{align*} \Kl_n(c, p^{\beta}) &= \sum_{x_1, \ldots, x_n \m p^{\beta} \atop x_1 \cdots x_n \equiv c \m p^{\beta}} 
e \left( \frac{x_1 + \ldots + x_n}{p^{\beta}} \right).\end{align*}
Here (as usual) $e(x) = \exp(2 \pi i x)$. We consider natural sums of these hyper-Kloosterman sums, 
\begin{align*} \Kl_n(\pm c, p^{\beta}) := \Kl_n(c, p^{\beta}) + \Kl_n(-c, p^{\beta}) 
&= \sum_{x_1, \ldots, x_n \m p^{\beta} \atop x_1 \cdots x_n \equiv \pm c \m p^{\beta}} e \left( \frac{x_1 + \ldots + x_n}{p^{\beta}} \right). \end{align*}
Given any choice of real number $Z >0$, we derive the following summation formula for the twisted sum in the approximate functional equation 
formula for $X_{\beta}(\pi, \delta)$ (see Lemma \ref{AFE} and Proposition \ref{MF}) in the course of showing of Theorems \ref{VSF2}, \ref{VSF3}, and Corollary \ref{VSF4} below. 
Writing to $\overline{c}$ denote the multiplicative inverse of a class $c \m p^{\beta}$, 
and taking $k(s)$ to be the Mellin transform of some smooth and compactly supported function (see Lemma \ref{test}), or in fact any such test function with 
$k(0) = 1$ if the generalized Ramanujan conjecture for $\pi$ at the real place is known, we derive the summation formula
\begin{align*} \frac{p}{\varphi(p)} \frac{W(\pi) \omega(p^{\beta}) (N p^{n \beta})^{\frac{1}{2}-\delta}}{p^{\frac{\beta n}{2}}} 
&\sum_{m \geq 1 \atop (m,p)=1} \frac{\overline{a(m)}}{m^{1-\delta}} \Kl_n(\pm m\overline{N}, p^{\beta}) \int_{\Re(s) = 2} \frac{k(-s)}{s} 
\frac{L(1-s+\delta, \widetilde{\pi}_{\infty})}{L(-s + \delta, \pi_{\infty})} \left(\frac{mZ}{N p^{n \beta}} \right)^{-s} \frac{ds}{2 \pi i} \\
&= X_{\beta}(\pi, \delta) + Z^{1 - \delta} \left( \sum_{m \geq 1 \atop m \equiv \pm 1 \m p^{\beta}} \frac{a(m)}{m} 
\int_{\Re(s) = - 2} \frac{k(-s + (1 - \delta))}{s - (1 - \delta)} \left(\frac{Z}{m} \right)^{s} \frac{ds}{2 \pi i} \right. \\ &\left. 
- \frac{1}{\varphi(p)} \sum_{ {m \geq 1 \atop m \equiv \pm 1 \m p^{\beta-1}} \atop m \not\equiv \pm 1 \m p^{\beta}} \frac{a(m)}{m}
\int_{\Re(s) = - 2} \frac{k(-s + (1 - \delta))}{s - (1 - \delta)} \left(\frac{Z}{m} \right)^{s} \frac{ds}{2 \pi i} \right). \end{align*}
In particular, we compute the average $X_{\beta}(\pi, \delta)$ as a residue term directly, which is a nontrivial calculation. 
The value in this calculation is to illustrate the derivation through successive Voronoi summation formulae, 
where the explicit nature of the prime-power modulus setting reveals the structure of passage clearly.  
Such summation formulae are not accessible via any of the existing works on Voronoi, among them those of Miller-Schmid \cite{MSn}, 
Goldfeld-Li \cite{GL08}, \cite{GL06} or Ichino-Templier \cite{IT}, or the more recent works of Miller-Zhou \cite{MZ} and Kiral-Zhou \cite{KZ}. 
This is a consequence of the delicate analysis required to deal with the implicit and non-admissible choice of archimedean weight function, 
which leads to the (indirect) derivation of the residual term $X_{\beta}(\pi, \delta)$.\footnote{The aforementioned works require smooth 
and compactly supported test functions, or else work directly on the level of Dirichlet series in the range of absolute convergence.} 
Unlike these other works, we also make use of the setting of prime-power modulus, where the hyper-Kloosterman sums which appear 
after unraveling the $n$-th power Gauss sums can be evaluated explicitly in the style of Sali\'e (see Proposition \ref{salie}). 
This calculation with its intermediate summation formulae suggests potential applications to the calculation of higher moments 
of $L$-functions, as well as to estimation in the style of Luo-Rudnick-Sarnak \cite{LRS}, although we do not pursue such applications here.
Note as well that we restrict to the setting of cuspidal representations for simplicity, and that a similar summation formula could be derived for
coefficients of Eisenstein series. In this way, our calculations should also imply the analytic continuation and corresponding functional 
equations for Eisenstein series on $\operatorname{GL}_n({\bf{A}}_{\bf{Q}})$ twisted by additive characters and hyper-Kloosterman sums. 
To spell out this latter point in a related special case, we explain in a final section $\S \ref{hKL}$ how to derive the analytic continuation and functional
equations of the following class of hyper-Kloosterman Dirichlet series: Given a coprime class $h \m p^{\beta}$ and $s \in {\bf{C}}$ 
(first with $\Re(s) >1$), we first consider the series defined by
\begin{align}\label{HKDS} \K_n(\pi, h, p^{\beta}, s) &= \sum_{m \geq 1 \atop (m, p)=1} \frac{a(m)}{m^s} \Kl_n(\pm mh, p^{\beta})
= \sum_{m \geq 1 \atop (m, p)=1} \frac{a(m)}{m^s} \left( \Kl_n(mh, p^{\beta}) + \Kl_n(-mh, p^{\beta}) \right). \end{align}
We prove the following theorems as a direct consequence of the calculations described above. 

\begin{theorem}\label{DAFI}

Let $\pi$ be a cuspidal $\GLn({\bf{A}}_{\bf{Q}})$-automorphic representation for $n \geq 2$
with level $N$, central character $\omega$, and $L$-function coefficients $a(m)$ as above. Let 
\begin{align*} F(s) &= \frac{L(1 -s, \widetilde{\pi}_{\infty})}{L(s, \pi_{\infty})} = \pi^{-\frac{n}{2} + ns} 
\frac{\prod_{j=1}^n \Gamma \left( \frac{1 - s - \overline{\mu}_j}{2} \right) }{\prod_{j=1}^n \Gamma \left( \frac{s - \mu_j}{2}\right)} \end{align*}
denote the quotient of archimedean factors appearing in the functional equation $(\ref{fFE})$ for $L(s, \pi \otimes \chi)$ below.
Fix a rational prime $p$ which does not divide $N$. Let $\beta \geq 1$ be any integer, and $h$ any coprime class modulo $p^{\beta}$. \\

\noindent (A) The Dirichlet series $\K_n(\pi, h, p^{\beta}, s)$ has an analytic continuation to all $s \in {\bf{C}}$, 
and satisfies the following additive functional identity: \\

\begin{itemize}

\item[(i)] If $\beta \geq 2$, then for $\Re(s) <0$ (after analytic continuation)
\begin{align*} \mathfrak{K}_n(\pi, h, p^{\beta}, s) &= W(\pi) \omega(p^{\beta}) N^{\frac{1}{2} - s} p^{n \beta(1-s)} F(s) 
\left( \frac{\varphi(p)}{p} \sum_{m \geq 1 \atop m \equiv \pm h N \m p^{\beta}} \frac{\overline{a(m)}}{m^{1- s}}
 - \frac{1}{p} \sum_{ {m \geq 1 \atop m \equiv \pm h N \m p^{\beta-1}} \atop m \not\equiv \pm h N \m p^{\beta}} \frac{\overline{a(m)}}{m^{1-s}} \right). \end{align*} 

\item[(ii)] If $\beta = 1$, then for $\Re(s) <0$ (after analytic continuation)
\begin{align*} \K_n(\pi, h, p, s) &= W(\pi) N^{\frac{1}{2} - s} F(s) \left( p^{n (1 - s)} \omega(p) 
\left[ \sum_{m \geq 1 \atop m \equiv \pm hN \m p} \frac{ \overline{a(m)}} {m^{1-s}} 
- \frac{2}{p-3} \sum_{m \geq 1 \atop m \not\equiv \pm hN \m p} \frac{ \overline{a(m)}}{m^{1-s}} \right]
+ \frac{2}{p-3} (-1)^n L(1- s, \widetilde{\pi}) \right). \\ \end{align*}

\end{itemize}

\noindent (B) Let $\phi$ be any smooth function on $y \in {\bf{R}}_{>0}$ which decays rapidly at $0$ and $\infty$, 
and let $\phi^*(s) = \int_0^{\infty} \phi(y) y^s \frac{dy}{y}$ denote its Mellin transform (when defined). Let us also write 
$\Phi = \Phi(\phi)$ to denote the function on $y \in {\bf{R}}_{>0}$ defined for a suitable choice of real number $\sigma \in {\bf{R}}_{>1}$ by the integral transform 
\begin{align*} \Phi(y) &= \int_{(-\sigma)} \phi^*(s) F(s) y^s \frac{ds}{2 \pi i } = \int_{(-\sigma)} \phi^*(s) \left( 
\pi^{-\frac{n}{2} + ns} \frac{\prod_{j=1}^n \Gamma \left( \frac{1 - s - \overline{\mu}_j}{2}\right) }{\prod_{j=1}^n \Gamma \left( \frac{s - \mu_j}{2} \right)}
\right) y^s \frac{ds}{2 \pi i }. \end{align*}

\begin{itemize}

\item[(i)] If $\beta \geq 2$, then we have for any coprime class $h \m p^{\beta}$ the summation formula 
\begin{align*} \sum_{m \geq 1 \atop (m,p)=1} &a(m) \Kl_n(\pm mh, p^{\beta}) \phi(m) \\
&= W(\pi) \omega(p^{\beta}) N^{\frac{1}{2}} p^{n \beta} \left( 
\frac{\varphi(p)}{p} \sum_{m \geq 1 \atop m \equiv \pm hN \m p^{\beta}} \frac{\overline{a(m)}}{m} \Phi \left( \frac{m}{N p^{n \beta}} \right)
- \frac{1}{p} \sum_{ {m \geq 1 \atop m \equiv \pm hN \m p^{\beta}} \atop m \not\equiv \pm hN \m p^{\beta} } 
\frac{\overline{a(m)}}{m} \Phi \left( \frac{m}{N p^{n \beta}} \right) \right). \\ \end{align*}

\item[(ii)] If $\beta = 1$, then we have for any coprime class $h \m p$ the summation formula 
\begin{align*} \sum_{m \geq 1 \atop (m,p)=1} &a(m) \Kl_n(\pm hN, p) \phi(m) 
= W(\pi) N^{\frac{1}{2}} F(s) \\ &\times \left( p^{n} \omega(p) \left[ \sum_{m \geq 1 \atop m \equiv \pm hN \m p}
\frac{\overline{a(m)}}{m} \Phi \left( \frac{m}{N p^n}\right) - \frac{2}{p-3} \sum_{m \geq 1 \atop m \not\equiv \pm hN \m p} 
\frac{\overline{a(m)}}{m} \Phi \left(\frac{m}{N p^n} \right) \right] 
+ (-1)^n \frac{2}{p-3} \sum_{m \geq 1} \frac{\overline{a(m)}}{m} \Phi \left( \frac{m}{N} \right)  \right). \\ \end{align*}

\end{itemize}

\end{theorem}

\begin{remark}

Let us note that although the main (residual) calculations in the body of this work cannot be recovered by existing 
Voronoi summation formulae, the simpler Voronoi formulae of Theorem \ref{DAFI} (A) and (B) above can be 
derived from those of Miller-Schmidt \cite{MSn} after taking a sum over additive characters to reduce to Ramanujan sums. 
To be more precise, one can consider a sum over coprime residue classes $a \bmod p^{\beta}$ of sums of the form
\begin{align*} \sum_{m \geq 1} \frac{a(m)}{m^s} e \left( \frac{aq}{p^{\beta}} \right), \end{align*}
to which the theorems of \cite{MSn} apply. Thus taking another coprime class $h \bmod p^{\beta}$, we have that 
\begin{align*} \sum_{a \bmod p^{\beta} \atop (a, p^{\beta}) = 1} e \left ( - \frac{ha}{p^{\beta}}  \right)
\sum_{m \geq 1} \frac{a(m)}{m^s} e \left( \frac{aq}{p^{\beta}} \right)
&= \sum_{m \geq 1} \frac{a(m)}{m^s} \sum_{a \bmod p^{\beta} \atop (a, p^{\beta}) = 1} e \left ( - \frac{ha}{p^{\beta}}  \right)
e \left( \frac{aq}{p^{\beta}} \right) = \sum_{m \geq 1} \frac{a(m)}{m^s} c_{p^{\beta}}(m-h), \end{align*}
where $c_{p^{\beta}}(r)$ denotes the Ramanujan sum of modulus $p^{\beta}$ at $r$. 
Since we have the well-known relation 
\begin{align*} c_{p^{\beta}}(r) &= \mu \left( \frac{p^{\beta}}{(p^{\beta}, r)} \right) 
\frac{\varphi(p^{\beta})}{\varphi \left(  p^{\beta}/(p^{\beta}, r) \right)}, \end{align*}
we deduce in the case of $\beta \geq 2$ (via the contribution of the M\"obius function to $c_{p^{\beta}}(m-h)$)
that the additional hyper-Kloosterman sums of moduli dividing $p^{\beta}$ in the formula of \cite{MSn} vanish. 
Thus the formulae of Theorem \ref{DAFI} (A) and (B) can be recovered from \cite{MSn}, 
although we give a different (streamlined) proof.  \end{remark}

We also consider the setting corresponding to twists by $\operatorname{GL}_1({\bf{A}}_{\bf{Q}})$ as follows.
Let us again fix $\xi$ a primitive Dirichlet character of conductor $q$ prime to $p$. 
Given $n \geq 1$ an integer, $\beta \geq 1$ an integer, $h$ a coprime class modulo $p^{\beta}$, 
and $s \in {\bf{C}}$ (first with $\Re(s) >1$), we consider the Dirichlet series defined by 
\begin{align*} \K_n^0(\xi, h, p^{\beta}, s) &= \sum_{m \geq 1 \atop (m, p)=1} \frac{\xi(m)}{m^s} \Kl_n(\pm mh, p^{\beta})
= \sum_{m \geq 1 \atop (m, p)=1} \frac{\xi(m)}{m^s} \left( \Kl_n(mh, p^{\beta}) + \Kl_n(-mh, p^{\beta}) \right), \end{align*}
as well as 
\begin{align*} \K_0^0(\xi, h, p^{\beta}, s) &= \begin{cases}
\sum\limits_{m \geq 1 \atop m \equiv h \m p^{\beta}} \frac{\xi(m)}{m^s} 
- \frac{1}{p} \sum\limits_{ {m \geq 1 \atop m \equiv \pm h \m p^{\beta-1}} \atop m \not\equiv \pm h \m p^{\beta}} \frac{\xi(m)}{m^s} &\text{ if $\beta \geq 2$} \\
\sum\limits_{m \geq 1 \atop m \equiv \pm h \m p} \frac{\xi(m)}{m^s} - \frac{2}{p-3} \sum\limits_{m \geq 1 \atop m \not\equiv \pm h \m p} \frac{\xi(m)}{m^s}
&\text{ if $\beta =1$}. \end{cases} \end{align*}

\begin{theorem}\label{D}

Fix an integer $n \geq 1$. Fix a prime number $p$. Let $\xi$ be any primitive Dirichlet character of conductor $q$ prime to $p$.
Let $\tau(\xi)$ denote the standard Gauss sum of $\xi$.
Fix an integer $\beta \geq 2$, and let $h$ be any coprime class modulus $p^{\beta}$. \\

\noindent(A) The Dirichlet series $\K_n^0(\xi, h, p^{\beta}, s)$
has an analytic continuation to all $s \in {\bf{C}}$, and satisfies the following additive functional identity. \\

\begin{itemize}

\item[(i)] If $\beta \geq 2$, then we have for $s \in {\bf{C}}$ with $\Re(s) <0$ (after analytic continuation) the functional identity 
\begin{align*} \K_n^0(\xi, h, p^{\beta}, s) &= \xi(p^{\beta}) \tau(\xi) q^{-s}  p^{\beta(1 - s)} \left( \pi^{s - \frac{1}{2}} 
\frac{\Gamma \left( \frac{1-s}{2} \right) }{ \Gamma \left( \frac{s}{2}\right)} \right) \K_{n-1}^0(\overline{\xi}, \overline{qh}, p^{\beta}, 1-s). \end{align*}

\item[(ii)] If $\beta = 1$, then we have for $s \in {\bf{C}}$ with $\Re(s) < 0$ (after analytic continuation) the functional identity 
\begin{align*} \K_n^0(\xi, h, p, s) &= \tau(\xi)  q^{-s}  \left( \pi^{s - \frac{1}{2}} 
\frac{\Gamma \left( \frac{1-s}{2} \right) }{ \Gamma \left( \frac{s}{2}\right)} \right) \left[ 
p^{1 - s} \xi(p) \K_{n-1}^0(\overline{\xi}, \overline{hq}, p, 1-s) + (-1)^n \left( 1 + \frac{2}{p-3} \epsilon_p(s, \xi) \right) L^{(p)}(1-s, \overline{\xi}) \right] \end{align*}
Here, $\epsilon_p(s, \xi)^{-1}$ denotes the Euler factor at $p$ of $L(s, \xi)$, 
so that $\epsilon_p(s, \xi) L(s, \xi) = L^{(p)}(s, \xi)$ denotes the incomplete $L$-function of $\xi$, with the Euler factor at $p$ removed. \\

\end{itemize}

\noindent (B) Suppose $n \geq 2$. Let $\phi$ be a smooth function on $y \in {\bf{R}}_{>0}$ which decays rapidly at $0$ and $\infty$, 
and let $\phi^*(s) = \int_0^{\infty} \phi(y) y^s \frac{dy}{y}$ denote its Mellin transform (when defined). Let us also write 
$\Phi = \Phi(\phi)$ to denote the function on $y \in {\bf{R}}_{>0}$ defined for a suitable choice of real number $\sigma \in {\bf{R}}_{>1}$ by the integral transform 
\begin{align*} \Phi(y) &= \int_{(-\sigma)} \phi^*(s) 
\left( \pi^{s - \frac{1}{2}} \frac{ \Gamma \left( \frac{1 - s}{2}\right) }{\Gamma \left( \frac{s}{2} \right)} \right) y^s \frac{ds}{2 \pi i }. \end{align*}

\begin{itemize}

\item[(i)] If $\beta \geq 2$, then we have for any coprime class $h \m p^{\beta}$ the summation formula 

\begin{align*} \sum_{m \geq 1 \atop (m,p)=1} \xi(m)\Kl_{n}(\pm mh, p^{\beta}) \phi(m) 
&= \tau(\xi) \xi(p^{\beta}) p^{\beta} \sum_{m \geq 1\atop (m,p)=1} \frac{\overline{\xi}(m)}{m} 
\Kl_{n-1}(\pm m \overline{hq}, p^{\beta}) \Phi \left( \frac{m}{q p^{\beta}} \right). \end{align*}

\item[(ii)] If $\beta = 1$, then we have for any coprime class $h \m p$ the summation formula 

\begin{align*} \sum_{m \geq 1 \atop (m,p)=1} &\xi(m)\Kl_{n}(\pm mh, p) \phi(m) \\
&= \tau(\xi) \left( \xi(p) p \sum_{m \geq 1 \atop (m,p)=1} \frac{\overline{\xi}(m)}{m} \Kl_{n-1}(\pm m \overline{hq}, p)
\Phi \left( \frac{m}{pq} \right) + (-1)^n \sum_{m \geq 1 \atop (m, p)=1} \frac{\overline{\xi}(m)}{m} 
\left( \Phi\left( \frac{m}{q} \right) + \frac{2}{p-3} \widetilde{\Phi}\left( \frac{m}{q} \right)  \right) \right). \end{align*}
Here, $\widetilde{\Phi}$ denotes the function on $y \in {\bf{R}}_{>0}$ defined by the modified integral transform 
\begin{align*} \widetilde{\Phi}(y) &= \int_{(-\sigma)} \phi^*(s) \epsilon_p(s, \xi)
\left( \pi^{s - \frac{1}{2}} \frac{ \Gamma \left( \frac{1 - s}{2}\right) }{\Gamma \left( \frac{s}{2} \right)} \right) y^s \frac{ds}{2 \pi i }. \end{align*}

\end{itemize} \end{theorem}

It is curious that while these latter results are derived almost entirely via the functional equations for $L(s, \pi \otimes \chi)$ or $L(s, \xi \otimes \chi)$, 
with a modest amount of harmonic analysis, the series $\K_n(\pi, h, p^{\beta}, s)$ and even $\K_n^0(\xi, h, p^{\beta}, s)$
do not seem to be well-understood or so far much developed. At the same time, it seems likely they have a crucial role 
to play in the estimation of the moments $X_{\beta}(\pi, \delta)$, and hence in subsequent progress towards to the generalized Ramanujan conjecture. 
As well, it seems likely this perspective could shed light on the open problem of calculating higher moments of $L$-functions, 
not only through natural links with Eisenstein series, but also through the scope it suggests for using $p$-adic Fourier theory (see e.g.~\cite{ST}) as a tool for estimation.
The work is therefore written with this perspective in mind, and with many of the lesser-known details for the case of prime-power modulus $\beta \geq 2$
described in full, so that other cases that we omit for simplicity such as Eisenstein series or $n=1$ could be derived mutatis mutandis in the same way. 

\subsubsection*{Acknowledgements} I am extremely grateful to Dorian Goldfeld, as well as to Friedrich G\"otze and 
Philippe Michel, for so many discussions and careful readings of this work, and also to Steve Miller for helpful suggestions 
for improvement of the writing. I am especially grateful to Dorian Goldfeld for his constant encouragement and commentary on this work, 
including various numerical calculations. I should also like to thank Valentin Blomer and anonymous referees for constructive comments, 
as well as Farrell Brumley, Gergely Harcos, and Djordje Milicevic, for helpful discussions on an earlier incarnation of this work.

\section{Some background} 

Fix $\chi$ a primitive even Dirichlet character of conductor $q$ prime to $N$. Recall that for $\Re(s) >1$ we consider
\begin{align*} L(s, \pi \otimes \chi) &= \sum_{m \geq 1 \atop (m,q)=1} a(m) \chi (m) m^{-s}. \end{align*} 
Recall too that this forms one component of the standard $L$-function $\Lambda(s, \pi) = L(s, \pi_{\infty}) L(s, \pi)$, where 
\begin{align*} L(s, \pi_{\infty}) &= \prod_{j=1}^{n} \Gamma_{\bf{R}}(s - \mu_j) 
= \prod_{j=1}^n \pi^{-\frac{(s - \mu_j) }{2}} \Gamma \left( \frac{s - \mu_j}{2}\right) \end{align*}
denotes the archimedean component, defined in terms of the Satake parameters $(\mu_j)_{j=1}^n$. 
Note that when $\pi_{\infty}$ is unitary, $\lbrace \overline{\mu_j}  \rbrace = \lbrace - \mu_j \rbrace$.
Let $\delta_0 = \max_j (\Re(\mu_j))_{j=1}^n$ denote the maximal real part of any of these parameters, 
so that $L(s, \pi_{\infty})$ is entire in the half plane $\Re(s) > \delta_0$. Note that the generalized Ramanujan/Selberg 
conjecture predicts $\delta_0 = 0$, and also that we have the following unconditional bounds towards this conjecture: 

\begin{theorem}[Luo-Rudnick-Sarnak, {\cite[Theorem 1.2]{LRS}}] 
Let $\pi = \otimes_v \pi_v$ be a cuspidal automorphic representation of $\GLn({\bf{A}}_{\bf{Q}})$ with unitary central character. 
If the component $\pi_{\infty}$ is spherical and parametrized by $\diag(\mu_j)_{j=1}^n$, then for each index $1 \leq j \leq n$, we have the bound 
$\left\vert \Re(\mu_j ) \right\vert  \leq  \frac{1}{2} - \frac{1}{n^2 + 1}. $\end{theorem}

\begin{remark} Better approximations towards the conjecture (e.g.~towards Selberg's eigenvalue conjecture \cite{Se}) exist for $n=2$,
where the current record is $7/64$ by Kim-Sarnak \cite{KSa}. \end{remark}

\section{Functional equations} 

Given a continuous or piecewise continuous function $f$ on $x \in {\bf{R}}$, let $f^*(s) = \int_0^{\infty}f(x) x^s \frac{dx}{x}$ denote its Mellin transform. 
We start with the following choice of test function $k(s)$ (cf.~\cite[$\S 3$]{LRS}). 

\begin{lemma}\label{test} 

Fix $g \in \mathcal{C}_c^{\infty}({\bf{R}}_{>0})$ a smooth test function. Let 
\begin{align*} G(x) &= \prod_{j=1}^n \left( x \frac{d}{dx} + \overline{\mu}_j \right) g(x) .\end{align*} 
Then, the Mellin transform $G^*(s) = \int_0^{\infty} G(x) x^s \frac{dx}{x}$ of $G(s)$ satisfies the relation 
\begin{align*} G^*(s) &= g^*(s) \prod_{j=1}^n (-s+ \overline{\mu}_j). \end{align*} 
In particular, $G^*(0) = \prod_{j=1}^n \overline{\mu}_j$ and $G^*(\overline{\mu}_1) = \cdots = G^*(\overline{\mu}_n) =0$. 
If we assume additionally that $\prod_{j=1}^n \overline{\mu}_j \neq 0$, then the (holomorphic) function $k(s)$ defined by 
\begin{align}\label{k} k(s) = \frac{G^*(s)}{\prod_{j=1}^n \overline{\mu}_j} \end{align}
satisfies the properties that $k(0) = 1$ and that $k(\overline{\mu}_1) = \cdots = k(\overline{\mu}_n) = 0$.
\end{lemma}

\begin{proof} 

The claim is easy to deduce using integration by parts, or even simply the known formula for the Mellin transform of $(x \frac{d}{dx})^n g(x)$ as $(-s)^n g^*(s)$. 
\end{proof}

Let us henceforth take $k(s) = G^*(s)$ to be the Mellin transform defined in $(\ref{k})$, 
imposing the additional condition\footnote{Note that \cite{LRS} take such a Mellin transform $g^*(s)$ (denoted $k(s) = f^*(s)$) as the test function in their approximate 
functional equation. However, there is typo in \cite{LRS} on the line before equation (3.6), i.e.~the condition should read $\int_0^{\infty} f(x) \frac{dx}{x} =1$.} 
that $\int_0^{\infty} G(x) \frac{dx}{x} =1$ so that $k(0) = 1$. Let $\chi$ be any primitive even Dirichlet chapter of conductor $q$ prime to the conductor $N$ of $\pi$.
Note that the completed $L$-functions $\Lambda(s, \pi \otimes \pi) = L(s, \pi_{\infty} \otimes \chi_{\infty}) L(s, \pi \otimes \chi)$ and 
$\Lambda(s, \pi) = L(s, \pi_{\infty})L(s, \pi)$ then have the same archimedean components $L(s, \pi_{\infty} \otimes \chi_{\infty}) = L(s, \pi_{\infty})$.
We can then write the functional equation of the finite part of the $L$-function $L(s, \pi \otimes \chi)$ in this setup as 
\begin{align*} L(s, \pi \otimes \chi) 
&= W(\pi) \omega(q) \chi(N) \left(  \frac{\tau(\chi)}{\sqrt{q}}\right)^n (N q^n)^{\frac{1}{2} - s} \left(  \frac{L(1-s, \widetilde{\pi}_{\infty})}{L(s, \pi_{\infty})} \right) 
L(1-s, \widetilde{\pi} \otimes \chi^{-1}) \\
&= W(\pi) \omega(q)\chi(N) \left(  \frac{\tau(\chi)}{\sqrt{q}}\right)^n (N q^n)^{\frac{1}{2} - s} 
\left( \pi^{-\frac{n}{2} + ns} \frac{\prod_{j=1}^n \Gamma \left( \frac{1 - s - \overline{\mu}_j}{2}\right)}{\prod_{j=1}^n \Gamma \left( \frac{s-\mu_j}{2} \right)} \right) 
L(1-s, \widetilde{\pi} \otimes \chi^{-1}) \end{align*}
Here (again), $W(\pi)$ denotes the root number of $\Lambda(s, \pi)$, and $\omega = \omega_{\pi}$ the central character of $\pi$.
Let us also write $F(s)$ to denote the quotient of archimedean factors in this functional equation:
\begin{align}\label{F} F(s) &= \frac{L(1-s, \widetilde{\pi}_{\infty})}{L(s, \pi_{\infty})} = \pi^{-\frac{n}{2} +ns} \cdot 
\frac{ \prod_{j=1}^n \Gamma \left( \frac{1 - s - \overline{\mu}_j}{2} \right)}{\prod_{j=1}^n \Gamma \left( \frac{s - \mu_j}{2}\right)}.\end{align}

Let us now consider the following smooth and rapidly decaying functions on $y \in {\bf{R}}_{>0}$:
\begin{align}\label{V1} V_{1}(y) &= \frac{1}{2 \pi i } \int_{\Re(s) = 2} k(s) y^{-s} \frac{ds}{s} \end{align} and 
\begin{align}\label{V2} V_2(y) = V_{\delta, 2}(y) &= \frac{1}{2 \pi i }\int_{\Re(s) = 2} k(-s) F(-s + \delta)
y^{-s} \frac{ds}{s}. \end{align}
We can apply a standard contour argument to the integral 
\begin{align}\label{contour} \frac{1}{2 \pi i} \int_{\Re(s)=2} k(s) L(s + \delta, \pi \otimes \chi) Z^s \frac{ds}{s}\end{align} 
to derive the following useful formula. 

\begin{lemma}\label{AFE}

Let $\chi$ be a primitive even Dirichlet character of conductor $q$ coprime to the level $N$ of $\pi$.
Let $Z > 0$ be any real number. Let $\delta$ be any complex number with $0 < \Re(\delta) < 1$. Then, we have
\begin{align}\label{formula} L(\delta, \pi \otimes \chi) &= \sum_{m \geq 1 \atop (m, q)=1} \frac{a(m) \chi(m)}{m^{\delta}} V_{1}\left( \frac{m}{ Z }\right)
+ W(\pi) \omega(q) \chi(N)\left( \frac{\tau(\chi)}{\sqrt{q}} \right)^n (N q^n)^{\frac{1}{2} - \delta} \sum_{m \geq 1} \frac{\overline{a(m)} \chi^{-1}(m)}{m^{1-\delta} }
V_{2}\left(\frac{mZ}{N q^n}\right). \end{align} \end{lemma} 

\begin{proof} The result is a standard; see \cite[Lemma 3.2]{LRS}. \end{proof}

The functions $V_1(x)$ and $V_2(x)$ decay rapidly as follows. Let us first review how to apply the Stirling approximation theorem 
to estimate the quotient of gamma factors appearing in the second function $V_2(x)$:
 
\begin{lemma}\label{stirling} 

Given $s \in {\bf{C}}$, write $s = \sigma + it$ for $t \neq 0$. Then, for $\sigma = \Re(s)$ fixed and $\vert \Im(s) \vert \rightarrow + \infty$, we have 
\begin{align*} \frac{ \prod_{j=1}^n \Gamma \left( \frac{1 - s - \overline{\mu}_j}{2}\right)}{ \prod_{j=1}^n \Gamma \left( \frac{s - \mu_j}{2} \right)}
&= \frac{ \prod_{j=1}^n \vert 1 - s - \overline{\mu}_j \vert^{1/2 - \sigma - \overline{\mu}_j} }{\prod_{j=1}^n \vert s - \mu_j \vert^{ \sigma - \mu_j - \frac{1}{2}} }.\end{align*} 

\end{lemma}

\begin{proof} See the discussion in \cite[Ch.~5, A4]{IK}. Stirling's asymptotic formula implies that 
\begin{align*} \frac{ \prod_{j=1}^n  \Gamma \left( \frac{ 1 - s - \overline{\mu}_j }{2} \right)}{\prod_{j=1}^n \Gamma \left( \frac{ s - \mu_j}{2}  \right)  }
&\approx \frac{ \prod_{j=1}^n \vert 1 - s - \overline{\mu}_j \vert^{1 - \sigma - \overline{\mu}_j - 1/2} e^{-\vert t \vert \frac{\pi}{2}}}
{\prod_{j=1}^n \vert s - \mu_j \vert^{\sigma - \mu_j - 1/2} e^{-\vert t \vert \frac{\pi}{2} }}
=  \frac{ \prod_{j=1}^n \vert 1 - s - \overline{\mu}_j \vert^{1/2 - \sigma - \overline{\mu}_j} }
{\prod_{j=1}^n \vert s - \mu_j \vert^{\sigma - \mu_j - \frac{1}{2}} }. \end{align*}  \end{proof}

\begin{lemma}\label{RD} Let $\delta_0 = \max_j (\Re(\overline{\mu}_j))$. The functions $V_1(x)$ and $V_2(x)$ are bounded as follows:

\begin{itemize}

\item[(i)] For each of $j =1,2$, $V_j(x) = O_{C, j}(x^{-C})$ for any choice of $C >0$ when $x \geq 1$, i.e.~as $x \rightarrow \infty$.
\item[(ii)] $V_1(x) = 1 + O_A(x^A)$ for any choice of $A \geq 1$ when $0 < x \leq 1$, i.e.~as $x \rightarrow 0$.
\item[(iii)] $V_2(x) \ll_{\varepsilon} 1 + O(x^{1 - \Re(\delta) - \delta_0 - \varepsilon})$ when $0 < x  \leq 1$, i.e.~as $x \rightarrow 0$.

\end{itemize}

\end{lemma}

\begin{proof} The result follows from the same standard contour argument given in \cite[Lemma 3.1]{LRS}. 
\end{proof}

Finally, let us record the following observation for future use. Recall that $\delta_0 = \max_{j= 1,2}(\Re(\overline{\mu}_j))$. 

\begin{proposition}\label{mellin}

Let $\phi_{\infty}$ denote the function defined on a real variable $x \in {\bf{R}}_{>0}$ by $\phi_{\infty}(x) = x^{-(1 - \delta)} V_2(f_{\beta}^{-1}x)$,
where $f_{\beta} >0$ is some arbitrary fixed real number. We have the following integral presentation of this function $\phi_{\infty}(x)$ for any $x \in {\bf{R}}_{>0}$:
For any choice of real number $\sigma$ in the interval $\max (\delta_0, 1 - \Re(\delta)) < \sigma < 3 - \Re(\delta)$,
\begin{align}\label{exactmellin} \phi_{\infty}(x) 
&= \int_{\Re(s) = \sigma } f_{\beta}^{s - (1 - \delta)} \frac{k(-s + (1 - \delta))}{ s - (1 - \delta)} \left( 
\pi^{-\frac{n}{2} + n(s-1)} \frac{ \prod_{j=1}^n \Gamma \left( \frac{s - \overline{\mu_j}}{2} \right)}
{ \prod_{j=1}^n \Gamma \left( \frac{1 - s - \mu_j}{2}\right)} \right) x^{-s} \frac{ds}{2 \pi i}.\end{align} 

\end{proposition}

\begin{proof} 

Recall that the cutoff function $V_2(x)$ is defined explicitly for any $x \in {\bf{R}}_{>0}$ as  
\begin{align}\label{V2def} V_2(x) &= \int_{\Re(s) = 2} \frac{k(-s) }{s} \left( 
\pi^{-\frac{n}{2} + n(-s +\delta)} \cdot \frac{ \prod_{j=1}^n \Gamma \left( \frac{1 + s - \delta - \overline{\mu}_j}{2}\right)  }
{\prod_{j=1}^n \Gamma\left( \frac{-s + \delta - \mu_j}{2} \right) }\right) x^{-s} \frac{ds}{2 \pi}.\end{align} 
Recall too that the function $k(s)$ is holomorphic and bounded for $\vert \Im(s) \vert \rightarrow \infty$, 
with the additional properties $k(0)=1$ and $k(\overline{\mu}_1) = \cdots = k(\overline{\mu}_n) = 0$. 
Now, it is easy to see that the quotient of gamma factors in the kernel has poles as $s = \overline{\mu}_1 - (1 - \delta), \ldots, s= \overline{\mu}_n - (1 - \delta)$. 
We may therefore move the line of integration in this definition $(\ref{V2def})$ to the left, avoiding these poles. That is, we may also define 
\begin{align*} V_2(x) &= \int_{(\sigma)} \frac{k(-s) }{s} \left( \pi^{-\frac{n}{2} + n(-s +\delta)} \cdot 
\frac{ \prod_{j=1}^n \Gamma \left( \frac{1 + s - \delta - \overline{\mu}_j}{2}\right)  }{\prod_{j=1}^n \Gamma\left( \frac{-s + \delta - \mu_j}{2} \right) }\right) 
x^{-s} \frac{ds}{2 \pi}\end{align*} so long as 
\begin{align*} \max_j \left(0, \Re(\overline{\mu}_j) - (1 - \Re(\delta) ) \right) ~~<~~ \sigma ~~\leq~~ 2. \end{align*}
Let us now return to the function $\phi_{\infty}(x) = x^{-(1 - \delta)} V_2(f_{\beta}^{-1} x)$. Observe (using the definition) that we have 
\begin{align*} \phi_{\infty}(x) &= \int_{(2)} \frac{k(-s) }{s} \left( \pi^{-\frac{n}{2} + n(-s +\delta)} \cdot 
\frac{ \prod_{j=1}^n \Gamma \left( \frac{1 + s - \delta - \overline{\mu}_j}{2}\right) }
{\prod_{j=1}^n \Gamma\left( \frac{-s + \delta - \mu_j}{2} \right) }\right)
x^{-( 1- \delta)} \left(\frac{x}{f_{\beta}} \right)^{-s} \frac{ds}{2 \pi i} \\
&= \int_{(2)} f_{\beta}^{s} \frac{k(-s) }{s} \left( \pi^{-\frac{n}{2} + n(-s +\delta)} \cdot 
\frac{ \prod_{j=1}^n \Gamma \left( \frac{1 + s - \delta - \overline{\mu}_j}{2}\right)}
{\prod_{j=1}^n \Gamma\left( \frac{-s + \delta - \mu_j}{2} \right) }\right)  x^{-s - (1 - \delta)}   \frac{ds}{2 \pi i} \\
&= \int_{(2 + (1 - \Re(\delta)))} f_{\beta}^{s - (1 - \delta)} \frac{k(-s + (1 - \delta)) }{s - (1 - \delta)} \left( \pi^{-\frac{n}{2} + n(-s +1)} \cdot 
\frac{ \prod_{j=1}^n \Gamma \left( \frac{ s - \overline{\mu}_j}{2}\right) }
{\prod_{j=1}^n \Gamma\left( \frac{ 1 -s - \mu_j}{2} \right) }\right) x^{-s} \frac{ds}{2 \pi i}, \end{align*} 
where in the last step we change variables $s \rightarrow s - (1 - \delta)$. 
Thus for $s \in {\bf{C}}$ with $\Re(s) = \sigma$ in the interval 
\begin{align*} \max_j \left( 1- \Re(\delta), \Re(\overline{\mu}_j) \right) ~~<~~\sigma ~~<~~2+ (1 - \Re(\delta)),\end{align*} 
we may write 
\begin{align*}\phi_{\infty}(x) &= \int_{(\sigma)} f_{\beta}^{s - (1 - \delta)} \frac{k(-s + (1 - \delta)) }{s - (1 - \delta)} \left( \pi^{-\frac{n}{2} + n(-s +1)} \cdot 
\frac{ \prod_{j=1}^n \Gamma \left( \frac{s - \overline{\mu}_j }{2} \right) }{\prod_{j=1}^n \Gamma \left( \frac{1 -s - \mu_j}{2} \right) } 
\right) x^{-s} \frac{ds}{2 \pi i}. \end{align*}
This shows the stated presentation of $\phi_{\infty}(x)$. \end{proof}

\section{Average values} 

Fix a prime $p$ which does not divide the dimension $n$ or the conductor $N$ of $\pi$. Fix $\beta \geq 1$ an integer.
Let $\varphi^{\star}(p^{\beta})$ denote the number of primitive Dirichlet characters $\chi \m p^{\beta}$. Hence, 
\begin{align*} \varphi^{\star}(p^{\beta})
&= p^{\beta} \prod_{p \mid\mid p^{\beta}} \left( 1 - \frac{2}{p}\right) \prod_{p^2 \mid p^{\beta}} \left( 1 - \frac{1}{p}\right)^2, \end{align*}
where the factor of $( 1 - 2/p)$ is omitted if $\beta \geq 2$ (as we shall usually assume). 
To derive our working expressing for the average $X_{\beta}(\pi, \delta)$, we begin with the following basic formulae, 
which although classical do not seem to be so well-known in the setting of prime-power modulus. 

\begin{proposition}\label{QO} 

Fix an integer $\beta \geq 2$. We have for any integer $m \geq 1$ that 
\begin{align*} \seven \chi(m) &= \begin{cases} 
\frac{1}{2}\varphi^{\star}(p^{\beta}) &\text{if $m \equiv \pm 1 \m p^{\beta}$} \\ 
-\frac{1}{2} \varphi(p^{\beta - 1}) &\text{if $m \equiv \pm 1 \m p^{\beta - 1}$and $m \not\equiv \pm 1 \m p^{\beta} $} \\ 
0 &\text{otherwise}. \end{cases} \end{align*} 

In the case that $\beta = 1$ corresponding to prime modulus, we also have the formula
\begin{align*} \seven \chi(m) &= \begin{cases} 
0 &\text{if $m \equiv 0 \m p$} \\ 
\frac{\varphi(p)}{2}- 1 &\text{if $m \equiv \pm 1 \m p$} \\ 
-1 &\text{otherwise}. \end{cases} \end{align*}

\end{proposition} 

\begin{proof} Fix integers $m \geq 1$ and $\beta \geq 1$. Let us first consider the sum over primitive characters $\chi \m p^{\beta}$, 
which via the M\"obius inversion formula (\cite[(3.8)]{IK}) is
\begin{align*} \sum_{\chi \m p^{\beta} \atop \chi \neq \chi_0} \chi(m) 
&= \sum_{0 \leq x \leq \beta \atop p^x \mid (m-1, p^{\beta})} \varphi(p^x) \mu\left( \frac{p^{\beta}}{p^x}\right).\end{align*}
Here, $\mu$ denotes the M\"obius function. It is easy to see from this formula that for $\beta \geq 2$ we have the relations 
\begin{align*}
\sum_{\chi \m p^{\beta} \atop \chi \neq \chi_0} \chi(m) 
&= \begin{cases} \varphi^{\star}(p^{\beta}) &\text{ if $m \equiv 1 \m p^{\beta}$} \\ 
-\varphi(p^{\beta-1})  &\text{ if $m \equiv 1 \m p^{\beta - 1}$ and $m \not\equiv 1 \m p^{\beta}$} \\ 0 &\text{ otherwise} \end{cases}, \end{align*}
using that $\varphi(p^{\beta}) - \varphi(p^{\beta - 1}) = \varphi^{\star}(p^{\beta})$ and that $\mu(p^{\beta}) = 0$. 
To detect relations for the subset of even characters $\chi(-1) = \chi(1)$, we compute 
\begin{align*} \sum_{\chi \m p^{\beta} \atop \chi \neq \chi_0} \chi(m) \left( \frac{\chi(1) + \chi(-1)}{2}\right) 
&= \frac{1}{2} \sum_{\chi \m p^{\beta} \atop \chi \neq \chi_0} \chi(m) + \frac{1}{2} \sum_{\chi \m p^{\beta} \atop \chi \neq \chi_0} \chi(-m). \end{align*}
The stated relations are then easy to derive. The well-known case of $\beta =1$ (cf.~\cite[(3.11)]{LRS}) can also be derived in this way, using the relations
\begin{align*}\sum_{\chi \m p \atop \chi \neq \chi_0} \chi(m) 
&= \begin{cases} \varphi^{\star}(p) &\text{ if $m \equiv \pm 1 \m p$} \\ 
0  &\text{ if $m \equiv 0 \m p$} \\ -1 &\text{ otherwise} \end{cases}. \end{align*} \end{proof}

Using this result, we now derive the following basic but crucial result for our calculations. Fix an integer $n \geq 1$. 
Given a residue class $r$ prime to the modulus $p^{\beta}$ (and hence $r$ prime to $p$), let us write $\Kl_n(r, p^{\beta})$ 
to denote the classical hyper-Kloosterman sum evaluated at $r$:
\begin{align*} \Kl_n(r, p^{\beta}) &:= \sum_{x_1, \cdots, x_n \m p^{\beta} \atop x_1 \cdots x_n  \equiv r \m p^{\beta}} 
e\left( \frac{x_1 + \cdots + x_n}{p^{\beta}}\right).\end{align*} 
Here, we write $e(x) = \exp(2 \pi i x)$. We also use the notation $\Kl_1$ to denote the corresponding Ramanujan sum. 
Given a coprime residue class $r \m p^{\beta}$, let us write $\overline{r}$ to denote the multiplicative inverse of $r \m p^{\beta}$. 

\begin{lemma}\label{SOGS} Let $n \geq 1$ be any integer. \\

(i) Given an integer $\beta \geq 2$, we have for any integer $r$ coprime to $p$ that
\begin{align*} \seven \overline{\chi}(r) \tau(\chi)^n 
&= \frac{\varphi(p^{\beta})}{2} \left( \Kl_n(r, p^{\beta}) +  \Kl_n(-r, p^{\beta}) \right), \end{align*}
where the sum ranges over primitive, even Dirichlet characters $\chi \m p^{\beta}$. \\

(ii) In the case of prime modulus corresponding to $\beta =1$, we also have for any integer $r$ coprime to $p$ that 
\begin{align*} \sevenp \overline{\chi}(r) \tau(\chi)^n
&= \left( \frac{\varphi(p)}{2} -1 \right) \left( \Kl_n(r, p^{\beta} ) + \Kl_n(-r, p^{\beta}) \right) - (-1)^{n} , \end{align*}
where the sum ranges over primitive, even Dirichlet characters $\chi \m p$. \end{lemma}

\begin{proof} Let us start with (i). Opening up the sum, we have the identification
\begin{align*} \seven \overline{\chi}(r) \tau(\chi)^n 
&=\seven \sum_{x_1, \ldots,  x_n \m p^{\beta}} \chi( \overline{r} x_1 \cdots x_n) e\left( \frac{x_1 + \cdots + x_n}{p^{\beta}}\right). \end{align*}
Switching the order of summation and using the relations of Proposition \ref{QO}, we then obtain 
\begin{align*} \frac{\varphi^{\star}(p^{\beta})}{2} \sum_{x_1, \ldots, x_n \m p^{\beta} \atop  x_1 \cdots x_n \equiv \pm r \m p^{\beta}} e \left( \frac{x_1 + \cdots + x_n}{p^{\beta}}\right)
-  \frac{\varphi(p^{\beta-1})}{2} \sum_{ {x_1, \ldots, x_n \m p^{\beta} \atop  x_1 \cdots x_n \equiv \pm r \m p^{\beta - 1}} \atop  x_1 \cdots x_n \not\equiv \pm r \m p^{\beta}} 
e \left( \frac{x_1 + \cdots + x_n}{p^{\beta}}\right). \end{align*}  
Now, consider the second sum in this expression, which after writing $y =  \overline{x_1 \cdots x_{n-1}}r \m p^{\beta}$ is the same as
\begin{align}\label{expd} \sum_{ {x_1, \ldots, x_n \m p^{\beta} \atop x_1 \cdots x_n \equiv \pm r \m p^{\beta-1}} \atop  x_1 \cdots x_n \not\equiv \pm r \m p^{\beta} } 
e \left( \frac{x_1 + \cdots + x_n}{p^{\beta}}\right) &= \sum_{x_1, \ldots, x_{n-1} \m p^{\beta}} e \left( \frac{x_1 + \cdots + x_{n-1}}{p^{\beta}}\right) 
\sum_{x_n \equiv \pm y \m p^{\beta-1} \atop x_n \not\equiv \pm y \m p^{\beta} } e \left( \frac{x_n}{p^{\beta}}\right). \end{align} 
Observe that each class $x_n$ in the inner sum can then be written as $x_n = \pm y + lp^{\beta-1}$ for some $1 \leq l \leq p-1$, 
\begin{align*} \sum_{x_n \equiv \pm y \m p^{\beta-1} \atop x_n \not\equiv \pm y \m p^{\beta} } e \left( \frac{x_n}{p^{\beta}}\right) 
&= \sum_{1 \leq l \leq p-1} e \left( \frac{y + l p^{\beta-1}}{p^{\beta}} \right) + e \left( \frac{-y + l p^{\beta-1} }{p^{\beta}} \right) 
= \left( e\left( \frac{y}{p^{\beta}}\right) + e \left(-\frac{y}{p^{\beta}} \right) \right) \sum_{1 \leq l \leq p-1} e \left(\frac{l}{p} \right).\end{align*}
Using the well-known identity $ \sum_{1 \leq l \leq p-1} e \left( \frac{l}{p} \right) = -1$, it is then easy to see that the sum $(\ref{expd})$ is equal to 
\begin{align*} - \sum_{x_1, \ldots, x_n \m p^{\beta} \atop  x_1 \cdots x_n \equiv \pm r \m p^{\beta}} e \left( \frac{x_1 + \cdots + x_n}{p^{\beta}}\right)
&= - \left( \Kl_n(r, p^{\beta}) + \Kl_n(-r, p^{\beta}) \right). \end{align*}
In this way, we obtain the formula
\begin{align*}\seven \tau(\chi)^n &= \left( \frac{\varphi^{\star}(p^{\beta})   + \varphi(p^{\beta-1})}{2} \right) 
\sum_{x_1, \ldots, x_n \m p^{\beta} \atop  x_1 \cdots x_n \equiv \pm r \m p^{\beta}} e \left( \frac{x_1 + \cdots + x_n}{p^{\beta}}\right). \end{align*} 
The stated formula then follows, using that $\varphi^{\star}(p^{\beta}) = \varphi(p^{\beta}) - \varphi(p^{\beta-1})$. 

To derive (ii) (cf.~\cite[(3.19)]{LRS}), we open up the sum and switch the order of summation to obtain
\begin{align*} \sevenp \overline{\chi}(r) \tau(\chi)^n 
&= \sevenp \overline{\chi}(r) \sum_{x_1, \ldots, x_n \m p^{\beta}} \chi(x_1 \cdots x_n) e \left(\frac{x_1 + \cdots + x_n}{p^{\beta}} \right) \\
&= \sum_{x_1, \ldots, x_n \m p^{\beta}} \sevenp \chi( x_1 \cdots x_n \overline{r})  e \left(\frac{x_1 + \cdots + x_n}{p^{\beta}} \right). \end{align*}
Using Proposition \ref{QO} to evaluate in the inner sum then gives us the expression 
\begin{align*} &\left( \frac{\varphi(p)}{2} -1 \right) \sum_{x_1, \ldots, x_n \m p \atop x_1 \cdots x_n \equiv \pm r \m p} e \left( \frac{x_1 + \cdots + x_n}{p} \right)
- \sum_{x_1, \ldots, x_n \m p \atop   \overline{x_1 \cdots x_n r} \not\equiv \pm 1 \m p } e \left( \frac{x_1 + \cdots + x_n}{p} \right)  \\ 
&= \left( \frac{\varphi(p)}{2} -1 \right) \sum_{x_1, \cdots, x_n \m p \atop x_1 \cdots x_n \equiv \pm r \m p} e \left( \frac{x_1 + \cdots + x_n}{p} \right)
- \left( \sum_{x_1 \m p \atop x_1 \not\equiv 1 \m p  } e \left( \frac{x_1 }{p} \right) \cdots \sum_{x_n \m p \atop x_n \not\equiv 1 \m p  } e \left( \frac{x_n }{p} \right) \right) \\
&= \left( \frac{\varphi(p)}{2} - 1 \right) \sum_{x_1, \ldots, x_n \m p \atop x_1 \cdots x_n \equiv \pm r \m p} e \left( \frac{x_1 + \cdots + x_n}{p} \right) - (-1)^n. \end{align*} 
\end{proof}

Using these relations, we can now derive the following moment formula (assuming $\beta \geq 2$ for simplicity): 

\begin{proposition}\label{MF} 

Fix a prime $p$ which does not divide the conductor $N$ of $\pi$, and let $\beta \geq 2$ be any integer.
We have for any choice of real parameter $Z >0$ the following average formula:

\begin{align}\label{decomposition} X_{\beta}(\pi, \delta) &=  X_{\beta, 1}(\pi, \delta, Z) + X_{\beta, 2}(\pi, \delta, Z), \end{align} 
where 
\begin{align}\label{X1} X_{\beta, 1}(\pi, \delta, Z) 
&= \sum_{m \geq 1 \atop m \equiv \pm 1 \m p^{\beta}} \frac{a(m)}{m^{\delta}} V_{1}\left( \frac{m}{ Z}\right) 
- \frac{1}{\varphi(p)} \sum_{ {m \geq 1 \atop  m \equiv \pm 1 \m p^{\beta - 1}} \atop m \not\equiv \pm 1 \m p^{\beta}} 
\frac{a(m)}{m^{\delta}} V_{1}\left( \frac{m}{Z}\right), \end{align} 
\begin{align}\label{X2} X_{\beta, 2}(\pi, \delta, Z) 
&=  \left( \frac{p}{\varphi(p)} \right) \frac{ W(\pi) \omega(p^{\beta}) (Np^{\beta n})^{\frac{1}{2} - \delta}}{ (p^{\beta})^{\frac{n}{2}}  } 
\sum_{m \geq 1 \atop (m, p)=1} \frac{\overline{a(m)}}{m^{1-\delta} } V_{2}\left(\frac{mZ}{N p^{\beta n}}\right) 
\left( \Kl_n( m \overline{N}, p^{\beta}) + \Kl_n(- m \overline{N}, p^{\beta}) \right).\end{align}

\end{proposition} 

\begin{proof}

Using formula Lemma \ref{AFE}, we can decompose the average $X_{\beta}(\pi, \delta)$ into sums 
\begin{align*} X_{\beta, 1}(\pi, \delta, Z) 
&:= \frac{2}{\varphi^{\star}(p^{\beta})} \seven \sum_{m \geq 1 \atop (m, p)=1} \frac{a(m) \chi(m)}{m^{\delta}} V_{1}\left( \frac{m}{Z}\right)\end{align*} and 
\begin{align*} X_{\beta, 2}(\pi, \delta, Z) 
&:= \frac{2}{\varphi^{\star}(p^{\beta})} 
\seven W(\pi) \omega(p^{\beta}) \left( \frac{\tau(\chi)}{p^{\beta}}\right)^n (Np^{n\beta})^{\frac{1}{2} - \delta} 
\sum_{m \geq 1 \atop (m, p)=1} \frac{\overline{a(m)} \chi^{-1}(m)}{m^{1-\delta} } V_{2}\left(\frac{mZ}{N p^{\beta n}}\right). \end{align*}

To evaluate $X_{\beta,1}(\pi, \delta, Z)$, we switch the order of summation, then use $(\ref{QO})$ to evaluate the inner sum:
\begin{align*} X_{\beta, 1}(\pi, \delta, Z) 
&= \left(\frac{\varphi^{\star}(p^{\beta})}{2} \right)^{-1}\sum_{m \geq 1 \atop (m, p)=1} \frac{a(m)}{m^{\delta}} V_1 \left( \frac{m}{Z}\right) \seven \chi(m)  \\ 
&= \sum_{ m \geq 1 \atop m \equiv \pm 1 \m p^{\beta}} \frac{a(m)}{m^{\delta}} V_{1}\left( \frac{m}{ Z}\right) - \frac{\varphi(p^{\beta-1})}{\varphi^{\star}(p^{\beta})} 
\sum_{ {m \geq 1 \atop m \equiv \pm 1 \m p^{\beta -1}} \atop m \not\equiv \pm 1 \m p^{\beta}} 
\frac{a(m)}{m^{\delta}} V_{1}\left( \frac{m}{Z}\right). \end{align*}
The stated formula is then easy to derive from the fact that $\varphi^{\star}(p^{\beta}) = (p-1)^2 p^{\beta-2}$ for $\beta \geq 2$.  

To evaluate the twisted sum $X_{\beta, 2}(\pi, \delta, Z)$, let us first open up the sum and switch the order of summation:
\begin{align*} \seven &W(\pi) \omega(p^{\beta}) \chi(N) \left( \frac{\tau(\chi)}{p^{\frac{\beta}{2}}} \right)^n (Np^{n \beta})^{\frac{1}{2} - \delta}  \sum_{m \geq 1 \atop (m,p)=1} 
\frac{\overline{a(m)} \chi^{-1}(m)}{m^{1-\delta} } V_{2}\left(\frac{mZ}{N p^{\beta n}}\right) \\ 
&= W(\pi) \omega(p^{\beta}) \cdot \frac{ (N p^{\beta n})^{\frac{1}{2} - \delta}}{ p^{\frac{ \beta n}{2}}  }  
\sum_{m \geq 1 \atop (m, p)=1}  \frac{\overline{a(m)}}{m^{1-\delta} } V_{2}\left(\frac{mZ}{N p^{\beta n}}\right) \seven \overline{\chi}( \overline{N} m) \tau(\chi)^n  \end{align*} 
Now, we can use Lemma \ref{SOGS} to evaluate the inner sum in this latter expression as 
\begin{align*} \seven \overline{\chi}(\overline{N} m) \tau(\chi)^n 
&= \frac{\varphi(p^{\beta})}{2} \left( \Kl_n( m \overline{N}, p^{\beta}) + \Kl_n(- m \overline{N}, p^{\beta}) \right). \end{align*} 
Substituting this back into the previous expression then gives 
\begin{align*} \frac{\varphi(p^{\beta})}{2} \cdot W(\pi) \omega(p^{\beta}) \cdot \frac{ (N p^{\beta n})^{\frac{1}{2} - \delta}}{ p^{\frac{ \beta n}{2}}  }  
\sum_{m \geq 1 \atop (m, p)=1} \frac{\overline{a(m)}}{m^{1-\delta} } V_{2}\left(\frac{mZ}{N p^{\beta n}}\right) 
\left( \Kl_n( m \overline{N}, p^{\beta}) + \Kl_n(- m \overline{N}, p^{\beta}) \right), \end{align*}
from which we derive the identity 
\begin{align*} X_{\beta, 2}(\pi, \delta, p^u) &= \frac{2}{\varphi^{\star}(p^{\beta})} 
\frac{\varphi(p^{\beta})}{2} \cdot W(\pi) \omega(p^{\beta}) \cdot \frac{ (N p^{\beta n})^{\frac{1}{2} - \delta}}{ p^{\frac{ \beta n}{2}}  }  
\sum_{m \geq 1 \atop (m, p)=1} \frac{\overline{a(m)}}{m^{1-\delta} } V_{2}\left(\frac{mZ}{N p^{\beta n}}\right) 
\left( \Kl_n( m \overline{N}, p^{\beta}) + \Kl_n(- m \overline{N}, p^{\beta}) \right). \end{align*}
The  stated formula for $X_2(\pi, \delta, Z)$ then follows after taking into account that for $\beta \geq 2$,
\begin{align}\label{factors} \frac{2}{\varphi^{\star}(p^{\beta})} \frac{\varphi(p^{\beta})}{2} 
&= \frac{(p-1) p^{\beta-1}}{(p-1)^2 p^{\beta-2}} = \frac{p}{\varphi(p)}. \end{align} \end{proof}

\section{Preliminary estimates} 

Let us now consider the following preliminary estimates for $X_{\beta}(\pi, \delta)$, using the theorem of Molteni \cite{Mol} (cf.~\cite{Li}).
Hence, we begin by stating the following result (``Ramanujan on average"):

\begin{theorem}[Molteni, {\cite[Theorem 4]{Mol}}]\label{molteni} 

Let $\pi$ be a cuspidal automorphic representation of $\GLn({\bf{A}}_{\bf{Q}})$ of conductor $N$,
with $L$-function coefficients $a(m)$ as above. Then, for any choice of $\varepsilon >0$, we have that
\begin{align*}\sum_{1 \leq m < x } \frac{ \vert a (m) \vert }{m} \ll_{\varepsilon} (N x)^{\varepsilon}.\end{align*} 
\end{theorem}
 
Let us now return to the setup of Proposition \ref{MF} above.  

\begin{lemma}\label{lower}

We have for any choice of $1 < Z < p^{\beta-1}$ and for any choice of $A \geq 1$ and $C >0$ the estimate
\begin{align*} X_{\beta,1}(\pi, \delta, Z) &= 1 + O_A(Z^{-A}) + O_{C,p} \left( (p^{\beta})^{\theta - \Re(\delta) - C} Z^C \right). \end{align*}
Here, we write $\theta \in [0, 1/2]$ to denote the best known approximation towards the generalized Ramanujan conjecture (with $\theta = 0$ conjectured).
Hence, taking $C \gg \theta - \Re(\delta)$ sufficiently large gives us the lower bound \begin{align}\label{LB} X_{\beta,1}(\pi, \delta, Z) \gg 1. \end{align}

\end{lemma}

\begin{proof} 

Let us first consider the contribution from the first coefficient $m=1$ in $X_{\beta, 1}(\pi, \delta, Z)$:
\begin{align*} a(1)V_1 \left( \frac{1}{Z} \right) = V_1 \left( \frac{1}{Z} \right) = 1 + O_A(Z^{-A}). \end{align*}
Here, we have used that $a(1) = 1$ in the first equality, and then the estimate of Lemma \ref{RD} to 
bound the contribution of $V_1(Z^{-1})$ (which lies in the region of moderate decay). 

To deal with the remaining contributions $m \geq 2$ in the expression $(\ref{X1})$, notice that $m$ must satisfy 
one of the constraints $m \equiv \pm 1 \m p^{\beta}$ or else $m \equiv \pm 1 \m p^{\beta-1}$ with $m \not\equiv \pm 1 \m p^{\beta}$.
On the other hand, observe that since we have chosen $1 < Z < p^{\beta-1}$, each of the remaining contributions $m \geq 2$
must satisfy the condition $m \geq Z$. Hence for each such $m \geq 2$, we have by the estimate of Lemma \ref{RD} that 
\begin{align*} V_1 \left( \frac{m}{Z} \right) &= O_C \left( \left(\frac{m}{Z}\right)^{-C} \right) ~~~~~~\text{for any choice of constant $C >0$.}\end{align*}
We can then bound the coefficient corresponding to each contributing term as 
\begin{align*} \frac{a(m)}{m^{\delta}}V_1 \left( \frac{m}{Z} \right) &= O_C \left( m^{\theta - \Re(\delta) -C} Z^C \right). \end{align*}
Expanding out the arithmetic progressions which define the sum of remaining contributions, we obtain
\begin{align*} \sum_{t \geq 1}\frac{a(\pm 1 + p^{\beta} t)}{(\pm 1 + p^{\beta} t)^{\delta}} V_1 \left( \frac{\pm 1 + p^{\beta} t}{Z}\right)
- \frac{1}{\varphi(p)} \sum_{ t \geq 1} \frac{a(\pm 1 + p^{\beta-1}t)}{(\pm 1 + p^{\beta-1} t)^{\delta}} V_1 \left(\frac{\pm 1 + p^{\beta-1} t}{Z} \right) 
\ll_{C, p} \sum_{t \geq 1} (p^{\beta}t)^{\theta - \Re(\delta) - C} Z^C. \end{align*}
That is, the sum of remaining contributions is bounded above in modulus by $Z^C (p^{\beta})^{\theta - \Re(\delta) - C} \sum_{t \geq 1} t^{-C}$. \end{proof}

\begin{lemma}\label{trivialX2} 

We have for any choices of $Z > 1$ and $\varepsilon > 0$ the (coarse) estimate
\begin{align*} X_{\beta, 2}(\pi, \delta, Z) &\ll_{p, \pi, \varepsilon} p^{-\frac{\beta}{2}} (N p^{\beta n})^{\frac{3}{2} + \varepsilon} 
N^{\Re(d) + \varepsilon} Z^{-(1 + \Re(\delta) + \varepsilon)}. \end{align*} \end{lemma}

\begin{proof} Put $f_{\beta} = N p^{\beta n}Z^{-1}$.
Using the classical bound $\Kl_n (c, p^{\beta}) \ll (p^{\beta})^{\frac{(n-1)}{2}}$ together with Theorem \ref{molteni} and Lemma \ref{RD} (iii), it follows that 
\begin{align*} X_{\beta, 2}(\pi, \delta, Z) \ll_{p, \pi, \varepsilon} (p^{\beta})^{-\frac{1}{2}}  ( N p^{\beta n})^{ \frac{1}{2} - \Re(\delta)} 
& (N f_{\beta})^{\Re(\delta) + \varepsilon} f_{\beta}. \end{align*} 
The stated bound follows after expanding and grouping together like terms. \end{proof} 


\section{Calculation of the twisted sum} 

We now consider the twisted sum $X_{\beta, 2}(\pi, \delta, Z)$, taking for granted the result of Lemma \ref{lower}.
That is, let us choose some unbalancing parameter $1 < Z < p^{\beta-1}$ of the form $Z = p^{u}$ with $1 < u < \beta -1$, and consider
\begin{align}\label{twisted} X_{\beta, 2}(\pi, \delta, p^u) &=  
 \frac{p}{\varphi(p) } \frac{ W(\pi) \omega(p^{\beta}) (N p^{\beta n})^{\frac{1}{2} - \delta}}{ p^{\frac{n \beta}{2}}  } 
\sum_{m \geq 1 \atop (m, p)=1} \frac{\overline{a(m)}}{m^{1-\delta} } V_{2}\left(\frac{m}{N p^{\beta n - u}}\right)
\left( \Kl_n( m \overline{N}, p^{\beta}) + \Kl_n(- m \overline{N}, p^{\beta})  \right). \end{align}
 
\subsection{Evaluation of hyper-Kloosterman sums}

Let us now suppose that $\beta \geq 4$. 

\begin{theorem}[``Sali\'e"]\label{salie}

Suppose that $p$ does not divide $n$. Assume without loss of generality that the exponent $\beta \geq 4$ is even, 
say $\beta = 2 \alpha$ for $\alpha \geq 2$. Then for any integer $c$ prime to $p^{\beta}$ (and hence prime to $p$),   
\begin{align}\label{hK} \Kl_n(c, p^{\beta}) &= p^{ \beta \left( \frac{n-1}{2} \right)  }
\sum_{w \m p^{\alpha} \atop w^n \equiv c \m p^{\alpha}} e \left( \frac{ (n-1) w + c \overline{w} }{p^{\beta}}  \right), \end{align} 
where the sum runs over all $n$-th roots of $c \m p^{\alpha}$. \end{theorem}

\begin{proof} 

The result is supposedly classical, though the main reference is \cite[Theorem C.1]{BB} (cf.~\cite[Lemma 12.2]{IK}). 
Note however that the statement of \cite[Theorem C.1]{BB} in fact depends on a 
choice of lifting of root $\m p^{\alpha}$ (i.e.~their notation $r^{1/n}$ refers to a lifting of a root of $r \m p^{\alpha}$ to $p^{2 \alpha}$). 
\end{proof}

\subsection{Reduction to twists by additive characters} 

Given a class $c \m p^{\beta}$, let $\psi_c$ denote the additive character defined by $\psi_c(m) = e \left( \frac{cm}{p^{\beta}}\right)$.
Let us also write $\psi_c(\pm m) = \psi_c(m) + \psi_c(m)$ to lighten notation. 
Given $\beta \geq 1$ an integer, let $(\frac{c}{p^{\beta}})_n$ denote the $n$-th power residue symbol. 
Hence, $(\frac{c}{p^{\beta}})_n =1$ if any only if there exists a coprime class $l \m p^{\beta}$ with $l^n \equiv c \m p^{\beta}$.
Note that by Hensel's lemma, $(\frac{c}{p^{\beta}})_n = 1$ if any only if $(\frac{c}{p})_n=1$.

\begin{proposition}\label{FA} 

Suppose that $p$ does not divide $n$. Assume again (without loss of generality) that $\beta \geq 4$ is even, 
say $\beta = 2 \alpha$ with $\alpha \geq 2$. Then, the twisted sum $X_{\beta, 2}(\pi, \delta, p^u) $ is equal to 
\begin{align*} \frac{p}{\varphi(p)} \frac{ W(\pi) \omega(p^{\beta}) (N p^{n \beta })^{\frac{1}{2} - \delta}}{ p^{\frac{3 \beta}{2}}  }  
&\sum_{x \m p^{\beta} \atop (\frac{x}{p})_n = 1} \sum_{w \m p^{\alpha} \atop w^n \equiv x \m p^{\alpha}} e \left( \frac{(n-1) w + x \overline{w}}{p^{\beta}} \right) \\
&\times \sum_{t \m p^{\beta}} \psi_{t}(-x) \sum_{m \geq 1 \atop (m,p)=1} \frac{ \overline{a(m)} \psi_t( \pm m \overline{N}) }{m^{1-\delta} } 
V_{2}\left(\frac{m}{ N p^{n \beta  - u} }\right). \end{align*} 

\end{proposition} 

\begin{proof} 

We apply Fourier inversion to the function $\K: \left( {\bf{Z}}/p^{\beta} {\bf{Z}} \right) \longrightarrow {\bf{C}}$ defined by 
\begin{align*} \K(c) &= \begin{cases} \sum\limits_{w \m p^{\alpha} \atop w^n \equiv c \m p^{\alpha}} e \left( \frac{(n-1) w + c \overline{w}}{p^{\beta}} \right)
&\text{ if $(\frac{x}{p})_n =1$}\\ 0 &\text{ otherwise}. \end{cases} \end{align*} 
Hence,  
\begin{align*} \K(c) &= p^{-\frac{\beta}{2}} \sum_{t \m p^{\beta}} \widehat{\K}(t) e \left( \frac{tc}{p^{\beta}}\right), \end{align*} 
where $\widehat{\K}(t)$ denotes the Fourier transform at the additive character determined by the class $t \m p^{\beta}$:
\begin{align*} \widehat{\K}(t) &= p^{-\frac{\beta}{2}} \sum_{x \m p^{\beta}} \K(x) e \left(- \frac{tx}{p^{\beta} } \right).\end{align*}
Using this relation, we find that for any integer $c$ prime to $p^{\beta}$, 
\begin{align*} \K(c) &= p^{-\beta} \sum_{t \m p^{\beta}} \sum_{x \m p^{\beta}} 
\sum_{w \m p^{\alpha} \atop w^n \equiv x \m p^{\alpha}} e \left( \frac{ (n-1)w  + x \overline{w}}{p^{\beta}}\right) 
e \left( \frac{ct - xt} {p^{\beta}}\right) \end{align*}
and hence  
\begin{align*} \K(c) + \K(-c) &= p^{-\beta} \sum_{t \m p^{\beta}} \sum_{x \m p^{\beta}} 
\sum_{w \m p^{\alpha} \atop w^n \equiv x \m p^{\alpha}} e \left( \frac{ (n-1) w  + x \overline{w}}{p^{\beta}}\right) 
\left( e \left( \frac{ct - xt} {p^{\beta}}\right) + e \left(\frac{-ct -xt}{p^{\beta}} \right)\right). \end{align*}
Using Proposition \ref{salie}, it follows that 
\begin{align*} \Kl_n(c, p^{\beta}) + \Kl_n(c, p^{\beta}) &= \left( p^{\beta} \right)^{\frac{n-1}{2}} \left( \K(c) + \K(-c) \right) \\ 
&= (p^{\beta})^{\frac{n-3}{2}} \sum_{t \m p^{\beta}} \sum_{x \m p^{\beta}} 
\sum_{w \m p^{\alpha} \atop w^n \equiv x \m p^{\alpha}} e \left( \frac{ (n-1) w  + x \overline{w}}{p^{\beta}}\right) 
\left( e \left( \frac{ct - xt} {p^{\beta}}\right) + e \left(\frac{-ct -xt}{p^{\beta}} \right)\right).\end{align*}
Substituting this back into $(\ref{X2})$, and switching the order of summation, we derive
\begin{align*} X_{\beta, 2}(\pi, \delta, p^u) 
= &\frac{p}{\varphi(p)} \frac{ W(\pi) \omega(p^{\beta}) (N p^{\beta n})^{\frac{1}{2} - \delta}}{ p^{\frac{n \beta}{2}}  }  (p^{\beta})^{\frac{n-3}{2}} 
 \sum_{m \geq 1 \atop (m, p)=1}  \frac{\overline{a(m)}}{m^{1-\delta} } V_{2}\left(\frac{m}{ N p^{\beta n-u} }\right)  \\ 
&\times  \sum_{t \m p^{\beta}} \sum_{x \m p^{\beta}} \sum_{w \m p^{\alpha} \atop w^n \equiv x \m p^{\alpha}} e \left( \frac{ (n-1)w  + x \overline{w}}{p^{\beta}}\right) 
\left( e \left( \frac{ t m \overline{N} -tx}{p^{\beta}}\right) + e \left( \frac{ -t m \overline{N} - tx}{p^{\beta}}\right) \right). \end{align*} 
which after re-arranging terms is equal to the stated formula. \end{proof} 

\subsection{Voronoi summation for additive twists} 

We now derive special Voronoi summation formulae (with polar terms) for the twisted sum $X_{\beta, 2}(\pi, \delta, p^u)$ via Proposition \ref{FA}, 
using nothing more than the functional equation for $L(s, \pi \otimes \chi)$. Recall that this functional equation is given explicitly by  
\begin{align}\label{fFE} L(s, \pi \otimes \chi) &= W(\pi) \omega(p^{\beta}) \chi(N) N^{\frac{1}{2}-s} p^{-\beta n s} \tau(\chi)^n  
\left( \pi^{\frac{n}{2} - ns} \frac{ \prod_{j=1}^n\Gamma \left( \frac{1 - s- \overline{\mu}_j}{2}\right)}{ \prod_{j=1}^n \Gamma \left(  \frac{s - \mu_j}{2}\right)} \right)
L(1-s, \widetilde{\pi} \otimes \chi^{-1}). \end{align}
Again (as in $(\ref{F})$ above), we shall write $F(s)$ to denote the quotient of archimedean factors appearing in $(\ref{fFE})$.

\subsubsection{Functional identities for additive twists} We begin with the following Corollary to Lemma \ref{SOGS} above:

\begin{corollary}\label{lcAC}

Let $m$ be any integer prime to $p$. Given $\beta \geq 2$ an integer, we have that
\begin{align*} e \left(\frac{m}{p^{\beta}} \right) + e \left(-\frac{m}{p^{\beta}} \right)  
&= \frac{2}{\varphi(p^{\beta})} \seven \overline{\chi}(m) \tau(\chi), \end{align*}
and in the case of $\beta =1$ corresponding to prime modulus $p$ that 
\begin{align*} e \left(\frac{m}{p} \right) + e \left(-\frac{m}{p} \right) 
&= \frac{2}{p-3} \left( \sevenp \overline{\chi}(m) \tau(\chi) - (-1)^n \right). \end{align*}
\end{corollary}

\begin{proof} Specialize Lemma \ref{SOGS} to $n=1$, then isolate the sums of additive characters in each case. \end{proof}

Given $\beta \geq 1$ any integer, and $h$ any coprime class modulo $p^{\beta}$,
let us now consider the Dirichlet series defined on $s \in {\bf{C}}$ (first with $\Re(s) >1$) by
\begin{align*} D(\pi, h, p^{\beta}, s) &= \sum_{m \geq 1 \atop (m,p)=1} \frac{a(m)}{m^s} 
\left(e \left(\frac{mh}{p^{\beta}} \right) + e \left(-\frac{mh}{p^{\beta}} \right) \right) .\end{align*}
We now show that $D(\pi, h, p^{\beta}, s)$ has an analytic continuation to $s \in {\bf{C}}$ via the following functional identities. 
Let us again (for any $n\geq 1$ and $\beta \geq 1$) write $\Kl_n( \pm c, p^{\beta}) = \Kl_n(c, p^{\beta}) + \Kl_n(-c, p^{\beta})$ to simplify expressions.

\begin{proposition}\label{AFI} We have the following additive functional identities for the Dirichlet series $D(\pi, h, p^{\beta}, s)$. \\

(i) If $\beta \geq 2$, then we have for any coprime class $h \m p^{\beta}$ the additive functional identity
\begin{align*} D(\pi, h, p^{\beta}, s) 
&=W(\pi) \omega( p^{\beta}) N^{\frac{1}{2} - s} p^{\beta(1 - ns)} F(s)
\sum_{m \geq 1 \atop (m, p)=1} \frac{\overline{a(m)}}{m^{1-s}}  \Kl_{n-1}( \pm m\overline{N h}, p^{\beta}). \end{align*}

(ii) In the case of $\beta =1$ corresponding to prime modulus $p$, we also have the additive functional identity  
\begin{align*} D(\pi, h, p, s) &= W(\pi) \omega(p) N^{\frac{1}{2} - s} p^{1 - ns} F(s) \sum_{m \geq 1 \atop (m, p)=1} \frac{\overline{a(m)}}{m^{1-s}}
\left(  \Kl_{n-1}( \pm m \overline{hN}, p) + (-1)^n \left( \frac{2}{p-3} \right) 
\left[ 1 - \frac{\epsilon_p(s) \overline{\epsilon}_p(1-s) }{p^{1-ns}} \right]  \right),\end{align*}
where $\epsilon_p(s)^{-1}$ denotes the Euler factor at $p$ of $L(s, \pi)$, and $\overline{\epsilon}_p(s)^{-1}$ that of $L(s, \widetilde{\pi})$. 

\end{proposition}

\begin{proof} 

Let us start with (i). Hence for $\Re(s) > 1$, we open up the sum and use Corollary \ref{lcAC} (i) to obtain
\begin{align}\label{r0} D(\pi, h, p^{\beta}, s) 
&= \frac{2}{\varphi(p^{\beta})}\seven \overline{\chi}(mh) \tau(\chi) \sum_{m \geq 1 \atop (m, p)=1} \frac{a(m)}{m^s}
= \frac{2}{\varphi(p^{\beta})} \seven \overline{\chi}(h) \tau(\chi) L(s, \pi \otimes \overline{\chi}). \end{align} 
Applying the functional equation $(\ref{fFE})$ to the inner Dirichlet series $L(s, \pi \otimes \chi)$, we then obtain  
\begin{align*} D(\pi, h, p^{\beta}, s) &=  \frac{2}{\varphi(p^{\beta})} W(\pi)\omega(p^{\beta}) N^{\frac{1}{2}-s} p^{-\beta n s} F(s) 
\seven \overline{\chi}(N h) \vert \tau(\overline{\chi}) \vert^2 \tau(\overline{\chi})^{n-1} L(1-s, \widetilde{\pi} \otimes \chi), \end{align*}
which after using that $\tau(\overline{\chi}) = \overline{\tau(\chi)}$ (and hence that $\tau(\chi) \tau(\overline{\chi}) = \vert \tau(\chi) \vert^2 = p^{\beta}$) gives us the identity 
\begin{align}\label{r1} D(\pi, h, p^{\beta}, s) &= \frac{2}{\varphi(p^{\beta})} W(\pi) \omega(p^{\beta}) N^{\frac{1}{2}-s} p^{\beta(1 - ns)} F(s) 
\seven \overline{\chi}(N h) \tau(\overline{\chi})^{n-1} L(1-s, \widetilde{\pi} \otimes \chi) \end{align} after analytic continuation. 
Let us now suppose that $\Re(s) <0$, in which case we can open up the Dirichlet series on the right of $(\ref{r1})$ and interchange summation to obtain 
\begin{align*} \frac{2}{\varphi(p^{\beta})} W(\pi) \omega(p^{\beta}) N^{\frac{1}{2}-s} p^{\beta -\beta n s} F(s) \sum_{m \geq 1 \atop (m, p)=1} 
\frac{\overline{a(m)}}{m^{1-s}} \seven \overline{\chi}(\overline{h N} m) \tau(\chi)^{n-1}. \end{align*} 
Using Corollary \ref{lcAC} (i) to evaluate the inner sum, we then obtain (after analytic continuation) the identity 
\begin{align*} D(\pi, h, p^{\beta}, s) &= W(\pi) \omega(p^{\beta}) N^{\frac{1}{2}-s} p^{\beta -\beta n s} F(s) 
\sum_{m \geq 1 \atop (m, p)=1} \frac{\overline{a(m)}}{m^{1-s}} \Kl_{n-1}( \pm \overline{N h} m, p^{\beta}). \end{align*} 

Let us now consider (ii). Hence for $\Re(s) >1$, we open up the sum and use Corollary \ref{lcAC} (ii) to obtain 
\begin{align*} D(\pi, h, p, s) 
&=  \frac{2}{p-3} \left( \sevenp \overline{\chi}(h)\tau(\chi) L(s, \pi \otimes \overline{\chi}) - (-1)^n \epsilon_p(s)L(s, \pi) \right). \end{align*}  
Applying the functional equation $(\ref{fFE})$ to each of the inner Dirichlet series, we then obtain
\begin{align*} &\frac{2}{p-3} W(\pi) \omega(p) N^{\frac{1}{2} -s} F(s) \left(
p^{-ns} \sevenp \overline{\chi}(h N) \tau(\chi) \tau (\overline{\chi})^n L(1-s, \widetilde{\pi} \otimes \chi)  - (-1)^n \epsilon_p(s) L(1-s, \widetilde{\pi}) \right), \end{align*}
which after using again that $\tau(\overline{\chi}) = \overline{\tau(\chi)}$ gives us (after analytic continuation) the expression 
\begin{align}\label{r1p} D(\pi, h, p, s) &= \frac{2}{p-3} W(\pi) \omega(p) N^{\frac{1}{2} -s} F(s) 
\left( p^{-ns +1}\sevenp \overline{\chi} (h N ) \tau (\overline{\chi})^{n-1}  L(1-s, \widetilde{\pi} \otimes \overline{\chi})  - \epsilon_p(s) L(1-s, \widetilde{\pi}) \right). \end{align}
Let us now suppose that $\Re(s) <0$. We can then expand the Dirichlet series on the right of $(\ref{r1p})$ to obtain 
\begin{align*} &\frac{2}{p-3} W(\pi) \omega(p) N^{ \frac{1}{2} -s} F(s) \left( p^{-ns +1} \sevenp \overline{\chi} (hN) \tau (\overline{\chi})^{n-1}  
\sum_{m \geq 1 \atop (m, p)=1} \frac{ \overline{a(m)} \chi (m) }{ m^{1-s} } - \epsilon_p(s) \sum_{m \geq 1} \frac{\overline{a(m)}}{m^{1 - s}}  \right) \\ 
&= \frac{2}{p-3} W(\pi) \omega(p) N^{ \frac{1}{2} -s} F(s) \left( p^{-ns +1} \sum_{m \geq 1 \atop (m, p)=1} \frac{ \overline{a(m)}}{ m^{1-s} } 
\sevenp \overline{\chi} (\overline{hN} m) \tau (\chi)^{n-1}  
- \epsilon_p(s) \overline{\epsilon}_p(1-s) \sum_{m \geq 1 \atop (m, p)=1} \frac{\overline{a(m)}}{m^{1 - s}}  \right). \end{align*} 
Now, observe that we may use Lemma \ref{SOGS} to evaluate the inner sum in this latter expression as 
\begin{align*} \sevenp \overline{\chi}(\overline{h N} m) \tau(\chi)^{n-1}  
&= \frac{p-3}{2} \Kl_{n-1}( \pm m \overline{h N}, p ) + (-1)^{n},  \end{align*} 
which gives us 
\begin{align*} &W(\pi) \omega(p) N^{\frac{1}{2} - s} F(s) \left( p^{1 - ns} \sum_{m \geq 1 \atop (m, p)=1} \frac{ \overline{a(m)} }{ m^{1 - s} } \Kl_{n-1}(\pm m \overline{h N}, p) 
+(-1)^n [p^{1 - ns} - \epsilon_p(s) \overline{\epsilon}_p(1-s)] \frac{2}{p-3} \sum_{m \geq 1 \atop (m, p)=1} \frac{ \overline{a(m)} }{ m^{1-s} } \right), \end{align*}
or equivalently 
\begin{align*} &W(\pi) \omega(p) N^{\frac{1}{2} - s} p^{1 - ns} F(s) \sum_{m \geq 1 \atop (m, p)=1} \frac{\overline{a(m)}}{m^{1-s}}
\left(  \Kl_{n-1}(\pm m \overline{hN}, p) + (-1)^n \left( \frac{2}{p-3} \right) \left[ 1 - \frac{\epsilon_p(s) \overline{\epsilon}_p(1-s)}{p^{1-ns}} \right]  \right). \end{align*}
Hence (after analytic continuation), we derive the stated functional identity for $D(\pi, h, p, s)$. \end{proof}

Let us also consider the following hyper-Kloosterman Dirichlet series. Let $\beta \geq 2$ be an integer. 
Here, we consider the Dirichlet series defined for a coprime residue class $h \m p^{\beta}$ and $s \in {\bf{C}}$ (first with $\Re(s) > 1$) by 
\begin{align}\label{hKD} \mathfrak{K}_n(\pi, h, p^{\beta}, s) &= 
\sum_{m \geq 1 \atop (m, p)=1}  \frac{a(m)}{m^s} \Kl_n(\pm mh, p^{\beta}) =
\sum_{m \geq 1 \atop (m, p)=1}  \frac{a(m)}{m^s} \left(\Kl_n(mh, p^{\beta}) + \Kl_n(-mh, p^{\beta})\right). \end{align}

\begin{proposition}\label{AFIhK} 

Assume that $\beta \geq 2$. The Dirichlet series $\mathfrak{K}(\pi, h, p^{\beta}, s)$ satisfies the functional identity 
\begin{align}\label{kfeat} \mathfrak{K}_n(\pi, h, p^{\beta}, s) &= W(\pi) \omega(p^{\beta}) N^{\frac{1}{2} - s} p^{n \beta(1-s)} F(s) 
\left( \frac{\varphi(p)}{p} \sum_{m \geq 1 \atop m \equiv \pm h N \m p^{\beta}} \frac{\overline{a(m)}}{m^{1- s}}
 - \frac{1}{p} \sum_{ {m \geq 1 \atop m \equiv \pm h N \m p^{\beta-1}} \atop m \not\equiv \pm h N \m p^{\beta}} \frac{\overline{a(m)}}{m^{1-s}} \right)
\end{align} for $\Re(s) < 0$ (after analytic continuation). \end{proposition}

\begin{proof} 

Observe that Lemma \ref{SOGS} gives us for $\Re(s) >1$ the relation
\begin{align*} \K_n(\pi, h, p^{\beta}, s)
&= \frac{2}{\varphi(p^{\beta})} \sum_{m \geq 1 \atop (m, p)=1}  \frac{a(m)}{m^s} \seven \chi(m h) \tau(\overline{\chi})^n 
=  \frac{2}{\varphi(p^{\beta})} \seven \chi(h) \tau(\overline{\chi})^n L(s, \pi \otimes \chi). \end{align*} 
Applying the functional equation $(\ref{fFE})$ to each $L(s, \pi \otimes \chi)$, we then obtain (after analytic continuation)
\begin{align*} &\frac{2}{\varphi(p^{\beta})} \seven \chi(h) \tau(\overline{\chi})^n
\left( W(\pi) \omega(p^{\beta}) \chi(N) N^{\frac{1}{2} - s} p^{-\beta n s}  \tau(\chi)^n F(s) L(1-s, \widetilde{\pi} \otimes \overline{\chi})\right) \\ 
&= \frac{2}{\varphi(p^{\beta})}  W(\pi) \omega(p^{\beta}) N^{\frac{1}{2} - s} p^{\beta n (1 - s)} F(s) 
\seven \chi(h N) L(1-s, \widetilde{\pi} \otimes \overline{\chi}) .\end{align*} 
Note that in the last step, we use that $\tau(\overline{\chi}) \tau(\chi) = \overline{\tau(\chi)} \tau(\chi) = \vert \tau(\chi) \vert^2 = p^{\beta}$. 
Hence, we derive the expression
\begin{align}\label{resk} \mathfrak{K}_n(\pi, h, p^{\beta}, s) &= \frac{2}{\varphi(p^{\beta})} W(\pi) \omega(p^{\beta}) 
N^{\frac{1}{2} - s} p^{\beta n (1 - s)} F(s) \seven \chi(h N) L(1-s, \widetilde{\pi} \otimes \overline{\chi}) \end{align}
after analytic continuation. Let us now suppose that $\Re(s) <0$, 
so that we can expand the absolutely convergent Dirichlet series on the right hand side of this latter expression as 
\begin{align*} \seven \chi(h N ) L(1-s, \widetilde{\pi} \otimes \overline{\chi}) 
&= \sum_{m \geq 1 \atop (m, p)=1} \frac{\overline{a(m)}}{m^{1-s}} \seven \chi(h N \overline{m}). \end{align*}
Applying the quasi-orthogonality relations of Proposition \ref{QO} to the inner sum, this latter expression equals
\begin{align*} \frac{\varphi^{\star}(p^{\beta})}{2} \sum_{m \geq 1 \atop m \equiv \pm h N \m p^{\beta}} \frac{\overline{ a(m)}}{m^{1-s}} 
- \frac{\varphi(p^{\beta-1})}{2} \sum_{ {m \geq 1 \atop m \equiv \pm h N \m p^{\beta-1}} \atop m \not\equiv \pm h N \m p^{\beta}} \frac{\overline{a(m)}}{m^{1-s}}. \end{align*}
Substituting this back into the previous expression, we see that $\mathfrak{K}_n(f, h, p^{\beta}, s)$ can be expressed for $\Re(s) <0$ (after analytic continuation) as
\begin{align*} \frac{2}{\varphi(p^{\beta})} W(\pi) \omega(p^{\beta}) N^{\frac{1}{2} - s} p^{n\beta(1 - s)} F(s)  
\left( \frac{\varphi^{\star}(p^{\beta})}{2} \sum_{m \geq 1 \atop m \equiv \pm h N \m p^{\beta}} \frac{\overline{ a(m)}}{m^{1-s}} - \frac{\varphi(p^{\beta-1})}{2} 
\sum_{ {m \geq 1 \atop m \equiv \pm hN \m p^{\beta-1}} \atop m \not\equiv \pm hN \m p^{\beta}} \frac{\overline{a(m)}}{m^{1-s}}\right). \end{align*}
Simplifying the scalar terms, using that $\varphi^{\star}(p^{\beta}) = (p-1)^2 p^{\beta-2}$ for $\beta \geq 2$, we derive the stated result. \end{proof}

\subsubsection{Derivation of formulae} 

Let $\phi$ be any continuous or piecewise continuous function on ${\bf{R}}_{>0}$ which decays rapidly as $0$ and $\infty$,
and let $\phi^*(s) = \int_0^{\infty}\phi(x) x^s \frac{dx}{x}$ denote its Mellin transform (when defined). 
Note that the only property we shall require of this of this function $\phi$ is that its Mellin transform be defined, 
and that it can be recovered from its Mellin transform by the inversion formula 
$\phi(x) = \int_{(\sigma)} \phi^*(s) x^{-s} \frac{ds}{2 \pi i}$ for a suitable choice of $\sigma \in {\bf{R}}_{>0}$
so that $\phi^*(s)$ is analytic and the integral absolutely convergent for $\Re(s) = \sigma$. 

\begin{theorem}[Voronoi summation formula]\label{VSF} 

Let $\pi = \otimes_v \pi_v$ be a cuspidal automorphic representation of $\GLn({\bf{A}}_{\bf{Q}})$ for $n \geq 2$, 
with $L$-function coefficients $a(m)$ and conductor $N$. Let $p$ be a prime which does not divide $N$.
Let $\phi$ be a smooth on ${\bf{R}}_{>0}$ which decays rapidly at $0$ and $\infty$, and let 
$\Phi$ denote the function defined on $y \in {\R}_{>0}$ for suitable choice of real number $\sigma \in {\R}_{>1}$ by the integral transform
\begin{align*} \Phi(y) &= \int_{(-\sigma)} \phi^*(s) \left( \pi^{-\frac{n}{2} + ns} 
\frac{\prod_{j=1}^n \Gamma \left( \frac{1 - s - \overline{\mu}_j}{2} \right)}{ \prod_{j=1}^n \Gamma \left( \frac{s-\mu_j}{2} \right)} \right) 
y^{s} \frac{ds}{2 \pi i} . \end{align*} 

(i) Given an integer $\beta \geq 2$, we have for each coprime class $h \m p^{\beta}$ the summation formula 

\begin{align*} \sum_{m \geq 1 \atop (m,p)=1} a(m) \Kl_1(\pm m h, p^{\beta}) \phi(m) 
&= W(\pi) \omega(p^{\beta}) N^{ \frac{1}{2} } p^{\beta} \sum_{m \geq 1 \atop (m, p)=1} \frac{\overline{a(m)}}{m} 
\Kl_{n-1}(\pm m \overline{N h}, p^{\beta}) \Phi\left( \frac{m }{N p^{\beta n}} \right). \\  \end{align*}

(ii) In the case of $\beta = 1$ corresponding to prime modulus $p$, we also have the summation formula 

\begin{align*} &\sum_{m \geq 1 \atop (m,p)=1} a(m) \Kl_{1}(\pm m h, p) \phi(m) \\
&=W(\pi) \omega(p) N^{\frac{1}{2}} p \left( \sum_{m \geq 1 \atop (m, p)=1} \frac{\overline{a(m)}}{m} 
\left( \Kl_{n-1}(\pm m \overline{N h}, p) + (-1)^n \frac{2}{p-3} \right) \Phi \left( \frac{m}{N p^n} \right) 
- (-1)^n \frac{2}{p-3} \cdot \frac{1}{p} \sum_{m \geq 1}  \frac{ \overline{a(m)}}{m} \widetilde{\Phi} \left( \frac{m }{N} \right) \right).\end{align*} 
Here, $\widetilde{\Phi}$ denotes the modified function defined on $y \in {\bf{R}}_{>0}$ by the integral transform
\begin{align*} \widetilde{\Phi}(y) &= \int_{(-\sigma)} \phi^*(s) \left( \pi^{-\frac{n}{2} + ns}
\frac{\prod_{j=1}^n \Gamma \left( \frac{1 - s - \overline{\mu}_j}{2} \right)}{ \prod_{j=1}^n \Gamma \left( \frac{s-\mu_j}{2} \right)} \right) 
\epsilon_p(s) y^{s} \frac{ds}{2 \pi i} , \end{align*} 
where $\epsilon_p(s)$ denotes the multiplicative inverse of the Euler factor at $p$ of $L(s, \pi)$.

\end{theorem}

\begin{proof} In either case, we use the Mellin inversion theorem $\phi(x) = \int_{(\sigma)} \phi^*(s) x^{-s} \frac{ds}{2 \pi i}$ to express  the sum as
\begin{align*} \sum_{m \geq 1 \atop (m,p)=1} a(m) \Kl_1(\pm m h, p^{\beta}) \phi(m) 
= \sum_{m \geq 1 \atop (m,p)=1} a(m) \left( e\left( \frac{mh}{p^{\beta}} \right) + e \left(-\frac{mh}{p^{\beta}} \right) \right) \phi(m) 
&= \int_{(\sigma)} \phi^*(s) D(\pi, h, p^{\beta}, s) \frac{ds}{2 \pi i}.\end{align*} 
Switching the range of integration to $\Re(s) = - \sigma$, then applying the corresponding additive functional identity of Proposition \ref{AFI}  
to the Dirichlet series in remaining integral, the stated formula (in each case) follows. \end{proof}

Let us now consider the corresponding Voronoi summation formulae we obtain after replacing the generic choice of well-behaved weight function
$\phi$ with the function $\phi_{\infty}$ appearing in Proposition \ref{mellin} above. More specifically, let us now consider what happens when we 
take as the weight function in Theorem \ref{VSF} the function defined on $y \in {\bf{R}}_{>0}$ by $\phi_{\infty}(y) := y^{-(1-\delta)}V_2(f_{\beta}^{-1}y)$, 
where $V_2$ is the cutoff function of rapid decay defined in $(\ref{V2})$ above, and $f_{\beta} := N p^{n \beta -u} = Np^{n \beta} Z^{-1}$ is now taken 
to be the length of its region of moderate decay (according to our choice of unbalancing parameter $Z = p^u$). Recall that in the definition $(\ref{V2})$ 
of the cutoff function $V_2(x)$, we introduced a holomorphic test function $k(s) := G^*(s)/ (\prod_{j=1}^n \overline{\mu}_j)$ from Lemma \ref{test}, 
and that this function satisfies the convenient properties $k(0)=1$ and $k(\overline{\mu}_1) = \cdots = k(\overline{\mu}_n) = 0$. 

\begin{theorem}[Voronoi summation with the weight function $\phi_{\infty}$]\label{VSF2} 

Let $\pi = \otimes_v \pi_v$ be a cuspidal automorphic representation of $\GLn({\bf{A}}_{\bf{Q}})$ for $n \geq 2$, 
with $L$-function coefficients $a(m)$. Fix $\delta \in {\bf{C}}$ with $0 < \Re(\delta) < 1$. 
Let $\phi_{\infty}$ denote the function defined on $y \in {\bf{R}}_{>0}$ by $\phi_{\infty}(y) = y^{-(1 - \delta)} V_2(f_{\beta}^{-1}y)$,
where $f _{\beta} = N p^{n \beta -u}$ for some fixed real parameter $0 < u < \beta -1$ is the length of the region of moderate 
decay for the cutoff function $V_2(y)$. Let us for this choice of $u$ write $\Phi_u$ to denote the function on $y \in {\bf{R}}_{>0}$ 
defined for any choice of real number $1 <  \sigma < 3 - \Re(\delta) $ by the integral transform
\begin{align*} \Phi_u(y) &= \int_{(-\sigma)} \frac{k(-s + (1 - \delta))}{s - (1 - \delta)} \left(\frac{y}{p^u} \right)^{s} \frac{ds}{2 \pi i}. \\ \end{align*}

(i) Given an integer $\beta \geq 2$, we have for each integer $h$ prime to $p$ the summation formula 
\begin{align*} \sum_{m \geq 1 \atop (m,p)=1} \overline{a(m)} \Kl_1(\pm m \overline{N h}, p^{\beta})\phi_{\infty}(m) 
= \frac{2}{\varphi(p^{\beta})} &W(\widetilde{\pi}) \overline{\omega}(p^{\beta}) N^{\delta - \frac{1}{2}} p^{\beta(1 - n(1 - \delta))} 
\seven \overline{\chi}(Nh) \tau(\overline{\chi})^{n-1}L(\delta, \pi \otimes \chi) \\ 
&+  \frac{  W(\widetilde{\pi}) \overline{\omega}(p^{\beta}) N^{\frac{1}{2}} p^{\beta}}{ (N p^{n \beta -u})^{1 - \delta}} 
\sum_{m \geq 1 \atop (m, p)=1} \frac{a(m)}{m}  \Kl_{n-1} ( \pm m \overline{N h}, p^{\beta}) \Phi_u(m). \\ \end{align*} 

(ii) In the case of $\beta =1$ corresponding to prime modulus $p$, we also have the summation formula 
\begin{align*} &\sum_{m \geq 1 \atop (m,p)=1} \overline{a(m)} \Kl_1(\pm mh, p) \phi_{\infty}(m) \\ 
&= \frac{2}{p-3} \cdot W(\widetilde{\pi}) \overline{\omega}(p) N^{\delta - \frac{1}{2}} \left( p^{1 - n(1 - \delta)} 
\sevenp \overline{\chi}(Nh) \tau(\overline{\chi})^{n-1} L(\delta, \pi \otimes \chi) - \epsilon_p(1-\delta) L(\delta, \pi) \right) \\ 
&+ \frac{ W(\widetilde{\pi}) \overline{\omega}(p) N^{\frac{1}{2}} p }{(N p^{n -u})^{1 - \delta}}  \left( \sum_{m \geq 1} \frac{a(m)}{m} 
\left( \Kl_{n-1}( \pm m \overline{N h}, p)  +  (-1)^n \frac{2}{p -3}  \right)  \Phi_u( m ) 
- (-1)^n \frac{2}{p-3} \cdot \frac{1}{p} \sum_{m \geq 1} \frac{a(m)}{m} \widetilde{\Phi}_u(p^n m) \right). \\ \end{align*}
Here, $\widetilde{\Phi}_u$ denotes the function defined on $y \in {\bf{R}}_{>0}$ by the modified integral transform
\begin{align*} \widetilde{\Phi}_u(y) &= \int_{(-\sigma)} \frac{ k(-s + (1 - \delta))}{s - (1 - \delta)} 
\left( \frac{y}{p^u} \right)^s \overline{\epsilon}_p(s) \frac{ds}{2 \pi i}, \end{align*} 
where $\overline{\epsilon}_p(s)$ again denotes the multiplicative inverse of the Euler factor at $p$ of $L(s, \pi)$. \\ 

\end{theorem}

\begin{proof} 

We proceed in the same way as for Theorem \ref{VSF} (but spelling out all details), 
viewing Proposition \ref{mellin} above as an explicit form of the Mellin inversion theorem. 
Hence, fix any real number $\sigma$ in the interval $1 < \sigma < 3 - \Re(\delta)$. 
Then for any $\beta \geq 1$, Proposition \ref{mellin} (with $f_{\beta} = Np^{n \beta -u}$) gives us the expression
\begin{align}\label{IP} \sum_{m \geq 1 \atop (m, p)=1} \overline{a(m)} \Kl_1(\pm m h, p^{\beta}) \phi_{\infty}(m) 
&= \int_{(\sigma)} D(\widetilde{\pi}, h, p^{\beta}, s) \phi_{\infty}^*(s)  \frac{ds}{2 \pi i}, \end{align} where  
\begin{align}\label{exact} \phi_{\infty}^*(s) &= f_{\beta}^{s - (1 - \delta)} \frac{k(-s + (1 - \delta))}{s - (1 - \delta)} F(-s + 1) 
= f_{\beta}^{s - (1 - \delta)} \frac{k(-s + (1 - \delta))}{s - (1 - \delta)} \pi^{-\frac{n}{2} + n(-s +1)} 
\frac{\prod_{j=1}^n \Gamma \left( \frac{s - \overline{\mu}_j}{2}\right) }{\prod_{j=1}^n \Gamma \left( \frac{-s +1 - \mu_j}{2} \right)} \end{align}
denotes the Mellin transform of $\phi_{\infty}(s)$ in this region $1 < \sigma < 3 - \Re(\delta)$. 

Suppose first that $\beta \geq 2$. We shift the range of integration in $(\ref{IP})$ to $\Re(s) = -\sigma$, crossing poles at 
$s = \overline{\mu}_j$ for each $1 \leq j \leq n$ of vanishing residues, i.e.~since $k(\overline{\mu}_1) = \ldots = k(\overline{\mu}_n) =0$
thanks to the construction of $k(s)$ in Lemma \ref{test} above. We also cross a simple pole at $s = 1 - \delta$ of residue
\begin{align*} &\Res_{s  = 1 - \delta} \left( D(\widetilde{\pi}, h, p^{\beta}, s) \phi_{\infty}^*(s) \right) 
=\Res_{s = 1 - \delta} \left( D(\widetilde{\pi}, h, p^{\beta}, s) f_{\beta}^{s - (1 - \delta)} \frac{k(-s + (1 - \delta))}{s - (1 - \delta)} F(s) \right) \\
&= D(\widetilde{\pi}, h, p^{\beta}, 1-\delta) F(1 - \delta) = D(\widetilde{\pi}, h, p^{\beta}, 1 - \delta) \pi^{-\frac{n}{2} + n(1 - \delta)} 
\frac{\prod_{j=1}^n \Gamma \left( \frac{\delta - \overline{\mu}_j}{2}\right)}{\prod_{j=1}^n \Gamma \left( \frac{1 - \delta - \mu_j}{2} \right)}. \end{align*}
Recall that we can calculate the value $D(\widetilde{\pi}, h, p^{\beta}, 1-\delta)$ using analytic continuation as in $(\ref{r1})$ above. 
To be precise, let us write $\overline{F}(s)$ to denote the corresponding quotient of contragredient archimedean components
\begin{align*} \overline{F}(s) &= \frac{L(1-s, \pi_{\infty})}{L(s, \widetilde{\pi}_{\infty})} = \pi^{-\frac{n}{2} + ns} 
\frac{\prod_{j=1}^n \Gamma \left( \frac{1 - s - \mu_j}{2} \right)}{\prod_{j=1}^n \Gamma \left(  \frac{s - \overline{\mu}_j}{2}\right)}. \end{align*}
Using the calculation $(\ref{r1})$, we then have the formula 
\begin{align*} D(\widetilde{\pi}, h, p^{\beta}, 1-\delta) 
&= \frac{2}{\varphi(p^{\beta})} \cdot W(\widetilde{\pi}) \overline{\omega}(p^{\beta}) N^{\delta - \frac{1}{2}} p^{\beta (1 - n(1 - \delta))} \overline{F}(\delta)
\seven \overline{\chi}(N h) \tau(\overline{\chi})^{n-1} L(\delta, \pi \otimes \chi) , \end{align*}
from which it follows that 
\begin{align*} &D(\widetilde{\pi}, h, p^{\beta}, 1-\delta) F(1 - \delta) 
&= \frac{2}{\varphi(p^{\beta})} \cdot W(\widetilde{\pi}) \overline{\omega}(p^{\beta}) N^{\delta - \frac{1}{2}} p^{\beta (1 - n(1 - \delta))} 
\seven \overline{\chi}(N h) \tau(\overline{\chi})^{n-1} L(\delta, \pi \otimes \chi). \end{align*}
To be clear, we have used the fact that the quotients of archimedean factors $\overline{F}(\delta) F(1 - \delta)$ cancel out:
\begin{align}\label{cancel} \overline{F}(s) F(-s + 1) 
&= \frac{L(1-s, \pi_{\infty})}{L(s, \widetilde{\pi}_{\infty})} \cdot \frac{L(1 - s + 1, \widetilde{\pi}_{\infty})}{L(-s + 1, \pi_{\infty})} 
= \pi^{ - \frac{n}{2} + ns - \frac{n}{2} + n(1-s)} 
\frac{\prod_{j=1}^n \Gamma \left( \frac{1 - s - \mu_j}{2} \right) \Gamma \left( \frac{s - \overline{\mu}_j}{2}\right) }
{\prod_{j=1}^n \Gamma \left( \frac{s - \overline{\mu}_j}{2}\right) \Gamma \left( \frac{1 - s - \mu_j}{2}\right)} = 1. \end{align}
Let us now consider the remaining integral (first with shorthand notations introduced above)
\begin{align*} \int_{(-\sigma)} D(\widetilde{\pi}, h, p^{\beta}, s) \phi_{\infty}^*(s)  \frac{ds}{2 \pi i}. \end{align*}
Since we are now in the range of absolute convergence for the Dirichlet series $D(\widetilde{\pi}, h, p^{\beta}, s)$,
we may invoke the functional identity of Proposition \ref{AFI} (i) above to obtain the expression
\begin{align*} \int_{(-\sigma)} &\phi_{\infty}^*(s) \left[ W(\widetilde{\pi}) \overline{\omega}(p^{\beta}) 
N^{\frac{1}{2} - s} p^{\beta(1 - ns)} \overline{F}(s)
\sum_{m \geq 1 \atop (m, p)=1} \frac{a(m)}{m^{1-s}} \Kl_{n-1}( \pm m\overline{N h}, p^{\beta}) \right] \frac{ds}{2 \pi i } \\ 
&= W(\widetilde{\pi}) \overline{\omega}(p^{\beta}) N^{\frac{1}{2}} p^{\beta} \sum_{m \geq 1 \atop (m,p)=1} \frac{a(m)}{m} \Kl_{n-1}(\pm m \overline{N h}, p^{\beta})
\int_{(-\sigma)} \overline{F}(s) \left(  \frac{m}{N p^{n \beta}} \right)^s \phi_{\infty}^*(s) \frac{ds}{2 \pi i}. \end{align*}
Opening up the definition $(\ref{exactmellin})$ of $\phi_{\infty}^*(s)$, this expression is then seen to be given more precisely by 
\begin{align*} 
&\frac{  W(\widetilde{\pi}) \overline{\omega}(p^{\beta}) N^{\frac{1}{2}} p^{\beta}}{f_{\beta}^{1 - \delta}} 
\sum_{m \geq 1 \atop (m, p)=1} \frac{a(m)}{m} \Kl_{n-1}(\pm m \overline{Nh}, p^{\beta})
\int_{(-\sigma)} \left( \frac{m f_{\beta}}{N p^{n \beta}} \right)^s \frac{k(-s + (1 - \delta))}{s - (1 - \delta)} \overline{F}(s)F(-s + 1) \frac{ds}{2 \pi},\end{align*}
where the product of quotients of archimedean factors $\overline{F}(s) F(-s + 1)$ cancels out identically as in $(\ref{cancel})$ above. 
Now, using that $f _{\beta} = Np^{n \beta -u}$, we obtain the even more precise expression 
\begin{align*} \frac{  W(\widetilde{\pi}) \overline{\omega}(p^{\beta}) N^{\frac{1}{2}} p^{\beta}}{ (N p^{n \beta -u})^{1 - \delta}} 
\sum_{m \geq 1 \atop (m, p)=1} \frac{a(m)}{m} \Kl_{n-1}(\pm m \overline{N h}, p^{\beta}) 
\int_{(-\sigma)} \left( \frac{m}{p^u} \right)^s \frac{k(-s + (1 - \delta))}{s - (1 - \delta)} \frac{ds}{2 \pi i}. \end{align*}
Putting this together with the residue term, we then derive the stated formula (i).

Let us now consider (ii), starting with the integral presentation $(\ref{IP})$. Shifting the range of integration to $\Re(s) = -\sigma$,
we cross poles at $s = \overline{\mu}_j$ for each $1 \leq j \leq n$ of vanishing residues thanks to the fact that $k(\overline{\mu}_j) = 0$
for each $1 \leq j \leq n$ by Lemma \ref{test} above. We also cross a simple pole at $s = 1 - \delta$ of residue 
\begin{align*} &\Res_{s  = 1 - \delta} \left( D(\widetilde{\pi}, h, p, s) \phi_{\infty}^*(s) \right) 
=\Res_{s = 1 - \delta} \left( D(\widetilde{\pi}, h, p, s) f_{\beta}^{s - (1 - \delta)} \frac{k(-s + (1 - \delta))}{s - (1 - \delta)} F(s) \right) \\
&= D(\widetilde{\pi}, h, p, 1-\delta) F(1 - \delta) = D(\widetilde{\pi}, h, p, 1 - \delta) \pi^{-\frac{n}{2} + n(1 - \delta)} 
\frac{\prod_{j=1}^n \Gamma \left( \frac{\delta - \overline{\mu}_j}{2}\right)}{\prod_{j=1}^n \Gamma \left( \frac{1 - \delta - \mu_j}{2} \right)}, \end{align*}
which we can calculate thanks to analytic continuation as in $(\ref{r1p})$ above as 
\begin{align*} \frac{2}{p-3} W(\widetilde{\pi}) \overline{\omega}(p) N^{\delta - \frac{1}{2}} 
\left( p^{1 - n(1 - \delta)} \sevenp \overline{\chi}(Nh) \tau(\overline{\chi})^{n-1} L(\delta, \pi \otimes \chi)  - \epsilon_p(1-\delta) L(\delta, \pi) \right). \end{align*} 
Here again, in the last equality, we use that $\overline{F}(1 - \delta ) F(\delta) =1$. To evaluate the remaining integral
\begin{align*} \int_{(-\sigma)} D(\widetilde{\pi}, h, p, s) \phi_{\infty}^*(s) \frac{ds}{2 \pi i}, \end{align*} 
we apply the functional identity of Proposition \ref{AFI} (ii) to the Dirichlet series $D(\widetilde{\pi}, h, p, s)$ to obtain 
\begin{align*} &\int_{(-\sigma)}  W(\widetilde{\pi}) \overline{\omega}(p) N^{\frac{1}{2} - s} p^{1 - ns} \overline{F}(s) 
\sum_{m \geq 1 \atop (m, p)=1} \frac{a(m)}{m^{1-s}} \left(  \Kl_{n-1}(\pm m \overline{hN}, p) + (-1)^n \left( \frac{2}{p-3} \right) 
\left[ 1 - \frac{ \overline{\epsilon}_p(s) \epsilon_p(1-s)}{p^{1-ns}} \right]  \right) \phi_{\infty}^*(s) \frac{ds}{2 \pi i } \\ 
&= W(\widetilde{\pi}) \overline{\omega}(p) N^{\frac{1}{2}} p 
\sum_{m \geq 1 \atop (m,p)=1} \frac{a(m)}{m} \left( \Kl_{n-1}(\pm m \overline{Nh}, p) +(-1)^n \frac{2}{p-3} \right) \int_{(-\sigma)} 
\left(  \frac{m}{N p^{n \beta}} \right)^s \overline{F}(s) \phi_{\infty}^*(s) \frac{ds}{2 \pi i} \\ 
&- \frac{1}{p} (-1)^n \frac{2}{p-3} W(\widetilde{\pi}) \overline{\omega}(p) N^{\frac{1}{2}} p \sum_{m \geq 1 \atop (m, p)=1} \frac{a(m)}{m} 
\int_{(\sigma)} \left( \frac{m p^{n} }{N p^{n}} \right)^s \overline{F}(s) \overline{\epsilon}_p(s) \phi_{\infty}^*(s) \frac{ds}{2 \pi i}, \end{align*}
which after using the definition $(\ref{exactmellin})$ of the Mellin transform $\phi_{\infty}^*(s)$ is given more precisely by 
\begin{align*} &\frac{ W(\widetilde{\pi}) \overline{\omega}(p) N^{\frac{1}{2}} p } {f_{\beta}^{1-\delta} }
\sum_{m \geq 1 \atop (m,p)=1} \frac{a(m)}{m} \left( \Kl_{n-1}(\pm m \overline{Nh}, p) +(-1)^n \frac{2}{p-3} \right) \int_{(-\sigma)} 
\left(  \frac{m f_{\beta}} {N p^{n \beta}} \right)^s \overline{F}(s) \frac{k(-s + (1 - \delta))}{s - (1 - \delta)} F(-s+1) \frac{ds}{2 \pi i} \\ 
&- \frac{1}{p} (-1)^n \frac{2}{p-3} \frac{ W(\widetilde{\pi}) \overline{\omega}(p) N^{\frac{1}{2}} p}{f_{\beta}^{1 - \delta}} 
\sum_{m \geq 1 \atop (m, p)=1} \frac{a(m)}{m} 
\int_{(\sigma)} \left( \frac{m f_{\beta}}{N} \right)^s \overline{F}(s) \overline{\epsilon}_p(s) \frac{k(-s + (1 - \delta))}{s - (1 - \delta)} F(-s +1) \frac{ds}{2 \pi i}. \end{align*}
Using again that $\overline{F}(s) F(-s+ 1) =1$, as spellt out in $(\ref{cancel})$ above, this latter expression is the same as 
\begin{align*} &\frac{ W(\widetilde{\pi}) \overline{\omega}(p) N^{\frac{1}{2}} p } {f_{\beta}^{1-\delta} }
\sum_{m \geq 1 \atop (m,p)=1} \frac{a(m)}{m} \left( \Kl_{n-1}(\pm m \overline{Nh}, p) +(-1)^n \frac{2}{p-3} \right) \int_{(-\sigma)} 
\left(  \frac{m f_{\beta}} {N p^{n \beta}} \right)^s  \frac{k(-s + (1 - \delta))}{s - (1 - \delta)} \frac{ds}{2 \pi i} \\ 
&- \frac{1}{p} (-1)^n \frac{2}{p-3} \frac{ W(\widetilde{\pi}) \overline{\omega}(p) N^{\frac{1}{2}} p}{f_{\beta}^{1 - \delta}} 
\sum_{m \geq 1 \atop (m, p)=1} \frac{a(m)}{m} 
\int_{(\sigma)} \left( \frac{m f_{\beta}}{N} \right)^s \overline{\epsilon}_p(s) \frac{k(-s + (1 - \delta))}{s - (1 - \delta)} \frac{ds}{2 \pi i}. \end{align*}
Now, using that $f_{\beta} = Np^{n \beta - u}$, this latter expression simplifies to give the stated formula. \end{proof}

We can now derive a Voronoi summation formula to describe the sum $X_{\beta, 2}(\pi, \delta, p^u)$ defined in $(\ref{X2})$ above. 

\begin{theorem}[Voronoi summation formula for the twisted sum $X_{\beta, 2}(\pi, \delta, p^u)$]\label{VSF3} 

Suppose that $\beta \geq 4$ is even, say $\beta = 2 \alpha$ for $\alpha \geq 2$.
Fixing a real parameter $0 < u < \beta -1$ as above, let us again write $\Phi_u$ to denote
the function on $y \in {\bf{R}}_{>0}$ defined for any choice of  real number $1 < \sigma < 3 - \Re(\delta) $ by the integral transform
\begin{align*} \Phi_u(y) &= \int_{(-\sigma)}  \frac{k(-s + (1 - \delta))}{s - (1 - \delta)}  \left( \frac{y}{p^u} \right)^{s} \frac{ds}{2 \pi i}. \\ \end{align*}
The twisted sum $X_{\beta, 2}(\pi, \delta, p^u)$ defined in $(\ref{X2})$ above can be described equivalently by the formula 
\begin{align*}  X_{\beta, 2}(\pi, \delta, p^u) &=
\frac{p^{\beta(1 - \frac{n}{2})}}{p^{\frac{3\beta}{2}}} \sum_{x \m p^{\beta} \atop (\frac{x}{p})_n=1} 
\sum_{w \m p^{\alpha} \atop w^n \equiv x \m p^{\alpha}} e \left( \frac{ (n-1)w  + x \overline{w}}{p^{\beta}}\right) %
\left\lbrace \frac{2}{\varphi^{\star}(p^{\beta})} \seven \chi(-x) \tau(\overline{\chi})^n L(\delta, \pi \otimes \chi) \right. \\ 
&\left. + \sum_{1 \leq y \leq \beta-2} \frac{   \omega(p^y) \psi_{p^y}(-x)}{p^y} 
p^{ny(1 - \delta)} \frac{2}{\varphi^{\star}(p^{\beta-y})} \seveny \tau(\overline{\chi})^{n-1} L(\delta, \pi \otimes \chi) \right. \\ 
&\left. + \frac{\omega(p^{\beta-1}) \psi_{p^{\beta-1}}(-x)}{p^{\beta-1}} \frac{p}{\varphi(p)} \frac{2}{p-3} 
\left( p^{n \delta} \sevenp \tau(\overline{\chi})^{n-1} L(\delta, \pi \otimes \chi) - p^{n-1} \overline{\epsilon}_p(1- \delta) L(\delta, \pi) \right) \right. \\
&\left. + p^{u(1 - \delta)} \left( \mathfrak{S}_{1, x} + \mathfrak{S}_{2, x} + \mathfrak{S}_{3, x} \right) \right\rbrace, \end{align*} 
where 
\begin{align*} \mathfrak{S}_{1, x} &= \frac{p}{\varphi(p)} \sum_{m \geq 1 \atop (m, p)=1} \frac{ a(m)}{m} \Kl_{n}(\pm m x, p^{\beta})\Phi_u(m), \end{align*}
\begin{align*} \mathfrak{S}_{2, x} &= \frac{p}{\varphi(p)} \sum_{1 \leq y \leq \beta-2} \frac{\omega(p^{y}) \psi_{-p^y}(x)  }{p^y}   
\sum_{m \geq 1 \atop (m, p)=1} \frac{a(m)}{m}  \Kl_{n-1}( \pm m, p^{\beta -y}) \Phi_u(p^{ny} m) , \end{align*}
and
\begin{align*} \mathfrak{S}_{3, x} &=  \frac{p}{\varphi(p)} \frac{ \omega(p^{\beta-1}) \psi_{ p^{\beta-1}} (-x)}{p^{\beta-1}} \\
&\times \left[ \sum_{m \geq 1 \atop (m,p)=1} \frac{a(m)}{m} \left( \Kl_{n-1}(\pm m, p) + (-1)^n \frac{2}{p-3} \right) \Phi_u(p^{n(\beta-1)} m) 
- \frac{1}{p} (-1)^n \frac{2}{p-3}  \sum_{m \geq 1 \atop (m, p)=1} \frac{a(m)}{m} \widetilde{\Phi}_u (p^{n \beta}m) \right].  \\ \end{align*}
\end{theorem}

\begin{proof} 

Let us keep all of the setup of Proposition \ref{FA} and Theorem \ref{VSF2}. Hence, we start with the formula
\begin{align*} X_{\beta, 2}(\pi, \delta, p^u) &= \frac{p}{\varphi(p)} 
\frac{W(\pi) \omega(p^{\beta})(N p^{n \beta})^{\frac{1}{2} - \delta} }{ p^{\frac{ 3 \beta}{2} } } \\
&\sum_{x \m p^{\beta} \atop (\frac{x}{p})_n = 1} \sum_{w \m p^{\alpha} \atop w^n \equiv x \m p^{\alpha}} e \left( \frac{ (n-1)w  + x \overline{w}}{p^{\beta}}\right) 
\sum_{t \m p^{\beta}} \psi_{-t}(x) \sum_{m \geq 1 \atop (m, p)=1} \overline{a(m)} \psi_t (\pm m \overline{N}) \phi_{\infty}(m).\end{align*}
Let us first divide the $t$-sum into classes which are coprime to $p$, 
plus a sum over multiples of $p$ as follows:
\begin{align}\label{S1} S_1 &= \frac{p}{\varphi(p)} \frac{W(\pi) \omega(p^{\beta})(N p^{n \beta})^{\frac{1}{2} - \delta} }{p^{\frac{3 \beta}{2}} } 
\sum_{x \m p^{\beta} \atop (\frac{x}{p})_n = 1} \sum_{w \m p^{\alpha} \atop w^n \equiv x \m p^{\alpha}} e \left( \frac{ (n-1)w  + x \overline{w}}{p^{\beta}}\right) 
\sum_{h \m p^{\beta} \atop (h, p^{\beta})=1} \psi_{-h}(x) \sum_{m \geq 1 \atop (m, p)=1} \overline{a(m)} \psi_h(\pm m \overline{N}) \phi_{\infty}(m) \end{align} 
and
\begin{align}\label{S2} S_2 &= \frac{p}{\varphi(p)} \frac{W(\pi) \omega(p^{\beta})(N p^{n \beta})^{\frac{1}{2} - \delta} }{ p^{\frac{3 \beta}{2}} } 
\sum_{x \m p^{\beta} \atop (\frac{x}{p})_n = 1} \sum_{w \m p^{\alpha} \atop w^n \equiv x \m p^{\alpha}} e \left( \frac{ (n-1)w  + x \overline{w}}{p^{\beta}}\right) 
\sum_{1 \leq y \leq \beta-1}  \psi_{- p^y}(x)  \sum_{m \geq 1 \atop (m, p)=1} \overline{a(m)} \psi_{p^y}(\pm m \overline{N}) \phi_{\infty}(m).\end{align}

We us start with the sum $S_1$ over coprime classes $(\ref{S1})$. It is easy to see from Theorem \ref{VSF2} that 
\begin{align*} S_1 &= \frac{p}{\varphi(p)} \frac{W(\pi) \omega(p^{\beta})(N p^{n \beta})^{\frac{1}{2} - \delta} }{ p^{\frac{3 \beta}{2}} }
\cdot \frac{2}{\varphi(p^{\beta})} W(\widetilde{\pi}) \overline{\omega}(p^{\beta}) N^{\delta - \frac{1}{2}} p^{\beta(1 - n(1 - \delta))} \\
&\times \sum_{x \m p^{\beta} \atop (\frac{x}{p})_n = 1} \sum_{w \m p^{\alpha} \atop w^n \equiv x \m p^{\alpha}} e \left( \frac{ (n-1)w  + x \overline{w}}{p^{\beta}}\right) 
\sum_{h \m p^{\beta} \atop (h, p^{\beta})=1} \psi_{-h}(x)
\seven \overline{\chi}(h\overline{N}N) \tau(\overline{\chi})^{n-1} L(\delta, \pi \otimes \chi) \\
&+ \frac{p}{\varphi(p)} \frac{W(\pi) \omega(p^{\beta})(N p^{n \beta})^{\frac{1}{2} - \delta} }{ p^{\frac{3 \beta}{2}} } 
\cdot \frac{W(\widetilde{\pi}) \overline{\omega}(p^{\beta}) N^{\frac{1}{2}} p^{\beta}}{(N p^{n \beta -u})^{1 - \delta}} \\
&\times \sum_{x \m p^{\beta} \atop (\frac{x}{p})_n = 1} \sum_{w \m p^{\alpha} \atop w^n \equiv x \m p^{\alpha}} e \left( \frac{ (n-1)w  + x \overline{w}}{p^{\beta}}\right) 
\sum_{h \m p^{\beta} \atop (h, p^{\beta})=1} \psi_{-h}(x) \sum_{m \geq 1} \frac{ a(m)}{m} \Kl_{n-1}(\pm m \overline{h}N\overline{N}, p^{\beta}) \Phi_u(m), \end{align*} 
which after grouping together and cancelling out like scalar terms (using the basic identity $(\ref{factors})$) equals 
\begin{align*} S_1 =  \frac{1}{p^{\frac{3 \beta}{2}}} 
\sum_{x \m p^{\beta} \atop (\frac{x}{p})_n = 1} &\sum_{w \m p^{\alpha} \atop w^n \equiv x \m p^{\alpha}} e \left( \frac{ (n-1)w  + x \overline{w}}{p^{\beta}}\right)  
\left( p^{\beta(1 - \frac{n}{2})} \sum_{h \m p^{\beta} \atop (h, p^{\beta})=1} \psi_{-h}(x)
\frac{2}{\varphi^{\star}(p^{\beta})} \seven \overline{\chi}(h) \tau(\overline{\chi})^{n-1} L(\delta, \pi \otimes \chi) \right. \\
&\left. + \frac{p}{\varphi(p)} p^{\beta(1 - \frac{n}{2})} p^{u(1 - \delta)}   \sum_{h \m p^{\beta} \atop (h, p^{\beta})=1} \psi_{-h}(x)
\sum_{m \geq 1 \atop (m, p)=1} \frac{a(m)}{m} \Kl_{n-1}(\pm m \overline{h}, p^{\beta}) \Phi_u(m) \right), \end{align*}
and which after switching the order of summation (in each of the two sums) is the same as
\begin{align*} S_1 =  \frac{p^{\beta(1 - \frac{n}{2})}}{p^{\frac{3 \beta}{2}}} 
\sum_{x \m p^{\beta} \atop (\frac{x}{p})_n = 1} &\sum_{w \m p^{\alpha} \atop w^n \equiv x \m p^{\alpha}} e \left( \frac{ (n-1)w  + x \overline{w}}{p^{\beta}}\right) 
\left( \frac{2}{\varphi^{\star}(p^{\beta})} \seven \tau(\overline{\chi})^{n-1} L(\delta, \pi \otimes \chi) 
\sum_{h \m p^{\beta} \atop (h, p^{\beta})=1}  \overline{\chi}(h) \psi_{-h}(x) \right. \\
&\left. + p^{u(1 - \delta)}  \frac{p}{\varphi(p)}  \sum_{m \geq 1 \atop (m, p)=1} \frac{a(m)}{m} \Phi_u(m)  
\sum_{h \m p^{\beta} \atop (h, p^{\beta})=1} \psi_{-h}(x)  \Kl_{n-1}(\pm m \overline{h}, p^{\beta}) \right). \end{align*}
Let us now consider the inner sums over coprime residue classes $h \m p^{\beta}$ appearing in this expression: 
\begin{align}\label{innerh0} \sum_{h \m p^{\beta} \atop (h, p^{\beta})= 1} \overline{\chi}(h) \psi_{-h}(x) 
= \sum_{h \m p^{\beta} \atop (h, p^{\beta})=1} \overline{\chi}(h) e \left(- \frac{xh}{p^{\beta}} \right) \end{align}
and
\begin{align}\label{innerh} \sum_{h \m p^{\beta} \atop (h, p^{\beta})= 1} \psi_{-h}(x) \Kl_{n-1}(\pm m \overline{h}, p^{\beta}) &= 
\sum_{h \m p^{\beta} \atop (h, p^{\beta})= 1} e \left( - \frac{xh}{p^{\beta}} \right) 
\sum_{x_1, \ldots, x_{n-1} \m p^{\beta} \atop x_1 \cdots x_{n-1} \equiv \pm m \overline{h} \m p^{\beta}} e \left( \frac{x_1 + \cdots + x_{n-1}}{p^{\beta}} \right). \end{align}
We argue that the first sum $(\ref{innerh0})$ can be evaluated by taking the Fourier transform of the additive character:
\begin{align}\label{innerh0ev} 
\sum_{h \m p^{\beta} \atop (h, p^{\beta})=1} \overline{\chi}(x) e \left(-\frac{xh}{p^{\beta}} \right) &= \chi(-x) \tau(\overline{\chi}). \end{align}
This formula is in fact classical (see e.g.~\cite[(3.12)]{IK}). Using this identity $(\ref{innerh0ev})$, we may then compute using $(\ref{innerh})$ as follows.
Notice that we may use Lemma \ref{SOGS} to evaluate 
\begin{align*} \sum_{h \m p^{\beta} \atop (h, p^{\beta})=1} e \left( - \frac{xh}{p^{\beta}} \right) \Kl_{n-1}(\pm m \overline{h}, p^{\beta}) 
&= \frac{2}{\varphi(p^{\beta})} \sum_{h \m p^{\beta} \atop (h, p^{\beta})=1} 
e \left( - \frac{xh}{p^{\beta}} \right) \seven \overline{\chi}(m \overline{h}) \tau(\chi)^{n-1}, \end{align*}
which after switching the order of summation is the same as 
\begin{align*} \frac{2}{\varphi(p^{\beta})} \seven \overline{\chi}(m) \tau(\chi)^{n-1} 
\sum_{h \m p^{\beta} \atop (h, p^{\beta})=1} \chi(h) e \left( - \frac{xh}{p^{\beta}} \right). \end{align*}
Using that 
\begin{align*} \sum_{h \m p^{\beta} \atop (h, p^{\beta})=1} \chi(h) e \left( - \frac{xh}{p^{\beta}} \right) &= \overline{\chi}(-x) \tau(\chi), \end{align*}
this latter expression is then evaluated as 
\begin{align*} \frac{2}{\varphi(p^{\beta})} \seven \overline{\chi}(-xm) \tau(\chi)^n. \end{align*}
Applying Lemma \ref{SOGS} again to evaluate this latter expression, we then obtain the identity 
\begin{align*} 
\sum_{h \m p^{\beta} \atop (h, p^{\beta})= 1} e \left(- \frac{xh}{p^{\beta}} \right) \Kl_{n-1}(\pm m \overline{h}, p^{\beta}) &= \Kl_n(\pm mx, p^{\beta}) \end{align*}
for the inner sum $(\ref{innerh})$. Using these identities for $(\ref{innerh0})$ and $(\ref{innerh})$, we then obtain the expression 
\begin{align*} S_1 =  \frac{p^{\beta(1 - \frac{n}{2})}}{p^{\frac{3 \beta}{2}}} 
\sum_{x \m p^{\beta} \atop (\frac{x}{p})_n = 1} &\sum_{w \m p^{\alpha} \atop w^n \equiv x \m p^{\alpha}} e \left( \frac{ (n-1)w  + x \overline{w}}{p^{\beta}}\right) 
\left( \frac{2}{\varphi^{\star}(p^{\beta})} \seven \chi(-x) \tau(\overline{\chi})^{n} L(\delta, \pi \otimes \chi) \right. \\
&\left. + \frac{p}{\varphi(p)} p^{u(1 - \delta)}  \sum_{m \geq 1 \atop (m, p)=1} \frac{a(m)}{m} \Kl_n(\pm mx, p^{\beta})\Phi_u(m)  \right). \end{align*}

Let us now consider the sum $S_2$ over classes given by powers of $p$ $(\ref{S2})$. We decompose this sum as 
\begin{align*} S_2 = \sum_{1 \leq y \leq \beta-y} S_{2, y}, \end{align*} where each sum $S_{2, y}$ is defined by 
\begin{align*} S_{2, y} &= \frac{p}{\varphi(p)} \frac{W(\pi) \omega(p^{\beta}) (N p^{n \beta})^{\frac{1}{2} - \delta} }{ p^{\frac{3 \beta}{2}}} 
\sum_{x \m p^{\beta} \atop (\frac{x}{p})_n=1} \sum_{w \m p^{\alpha} \atop w^n \equiv x \m p^{\alpha}} e \left( \frac{ (n-1)w  + x \overline{w}}{p^{\beta}}\right) 
\psi_{-p^y}(x) \sum_{m \geq 1 \atop (m, p)=1} \overline{a(m)} \psi_{p^y}(\pm m \overline{N}) \phi_{\infty}(m). \end{align*} 
We first evaluate the sums $S_{2, y}$ in the range $1 \leq y \leq \beta-2$ using the argument of Theorem \ref{VSF2} (i) above. 
Hence, let us consider the inner sum $S_{2, y}^{\star}$ defined by 
\begin{align*} S_{2, y}^{\star} &= \sum_{m \geq 1 \atop (m, p)=1} \overline{a(m)} \psi_{p^y}(\pm m \overline{N}) \phi_{\infty}(m) = 
\sum_{m \geq 1 \atop (m, p)=1} \overline{a(m)} \Kl_1(\pm m \overline{N}, p^{\beta-y}) \phi_{\infty}(m), \end{align*} 
where (recall) $\phi_{\infty}(y) = y^{-(1 - \delta)}V_2(f_{\beta}^{-1}y)$ for $f_{\beta} = Np^{n \beta -u}$ as above 
(with $\beta \geq 4$ and $0 < u < \beta-1$ fixed). Fixing a real number $\sigma$ in the interval $1 < \sigma < 3 - \Re(\delta)$,
we can use the integral presentation of $\phi_{\infty}(y)$ given in Proposition \ref{mellin} above to describe this sum $S_{2, y}^{\star}$ as
\begin{align*} S_{2, y}^{\star} &= \int_{(\sigma)} D(\widetilde{\pi}, \overline{N}, p^{\beta-y}, s) \phi_{\infty}^*(s) \frac{ds}{2 \pi i}, \end{align*}
where the Mellin transform $\phi_{\infty}^*(s)$ is given explicitly as in $(\ref{exactmellin})$ above as 
\begin{align*} \phi_{\infty}^*(s) &= f_{\beta}^{ s - (1 - \delta)} \frac{k(-s + (1 - \delta))}{s - (1 - \delta)} \cdot F(-s + 1) = 
f_{\beta}^{s - (1 - \delta)} \frac{k(-s + (1 - \delta))}{s -(1 - \delta)} \cdot \pi^{-\frac{n}{2} + ns} 
\frac{\prod_{j=1}^n \Gamma \left( \frac{s - \overline{\mu}_j}{2} \right) }{\prod_{j=1}^n \Gamma \left( \frac{1 - s - \mu_j}{2}\right)}.\end{align*}
Shifting the range of integration to $\Re(s) = -\sigma$, we cross poles at $s = \overline{\mu}_j$ for each $1 \leq j \leq n$ of vanishing residues
(thanks to Lemma \ref{test}). We also cross a simple pole at $s = 1-\delta$ of residue 
\begin{align*} \Res_{s = 1 - \delta} \left( D(\widetilde{\pi}, \overline{N}, p^{\beta-y}, s) \phi_{\infty}^*(s) \right) 
&= D(\widetilde{\pi}, \overline{N}, p^{\beta -y}, 1 - \delta) F(\delta). \end{align*} 
Now, we can calculate the residue via analytic continuation as in $(\ref{r1})$ above: 
\begin{align*} &D(\widetilde{\pi}, \overline{N}, p^{\beta -y}, 1 - \delta) F(\delta) \\  &= 
\left[ \frac{2}{\varphi(p^{\beta - y})} W(\widetilde{\pi}) \overline{\omega}(p^{\beta -y}) N^{\delta - \frac{1}{2}} p^{(\beta - y)(1 - n (1 - \delta))} \overline{F}(1 - \delta) 
\seveny \overline{\chi}(N \overline{N}) \tau(\overline{\chi})^{n-1} L(\delta, \pi \otimes \chi)\right] F(\delta) \\ &= 
\frac{2}{\varphi(p^{\beta - y})} W(\widetilde{\pi}) \overline{\omega}(p^{\beta -y}) N^{\delta - \frac{1}{2}} p^{(\beta - y)(1 - n (1 - \delta))}
\seveny \tau(\overline{\chi})^{n-1} L(\delta, \pi \otimes \chi), \end{align*}
using again that that $\overline{F}(1 - \delta) F(\delta)= 1$. To evaluate the remaining integral 
\begin{align*} \int_{(-\sigma)} D(\widetilde{\pi}, \overline{N}, p^{\beta -y}, s) \phi_{\infty}^*(s) \frac{ds}{2 \pi i}, \end{align*}
we use that $- \sigma < 0$ allows us to apply the functional identity of Proposition \ref{AFI} (i) to $D(\widetilde{\pi}, \overline{N}, p^{\beta-y}, s)$:
\begin{align*} D(\widetilde{\pi}, \overline{N}, p^{\beta - y}, s) 
&=W(\widetilde{\pi}) \overline{\omega}( p^{\beta - y}) N^{\frac{1}{2} - s} p^{(\beta - y)(1 - ns)} \overline{F}(s)
\sum_{m \geq 1 \atop (m, p)=1} \frac{a(m)}{m^{1-s}} \Kl_{n-1}( \pm m \overline{N} N , p^{\beta - y}).\end{align*}
This gives us the expression 
\begin{align*} &\int_{(-\sigma)} \phi_{\infty}^*(s) \left[ W(\widetilde{\pi}) \overline{\omega}( p^{\beta - y}) N^{\frac{1}{2} - s} p^{(\beta - y)(1 - ns)} \overline{F}(s)
\sum_{m \geq 1 \atop (m, p)=1} \frac{\overline{a(m)}}{m^{1-s}} \Kl_{n-1}( \pm m, p^{\beta - y}) \right] \frac{ds}{2 \pi i} \\
&= W(\widetilde{\pi}) \overline{\omega}(p^{\beta -y}) N^{\frac{1}{2}} p^{\beta - y} \sum_{m \geq 1 \atop (m, p)=1} \frac{a(m)}{m} 
\Kl_{n-1}(\pm m, p^{\beta -y}) \int_{(-\sigma)} \left( \frac{m}{N p^{n(\beta -y)}} \right)^s \overline{F}(s) \phi_{\infty}^*(s) \frac{ds}{2 \pi i}, \end{align*}
which after expanding the definition of the Mellin transform $\phi_{\infty}^*(s)$ is given more explicitly by 
\begin{align*} &\frac{W(\widetilde{\pi}) \overline{\omega}(p^{\beta -y}) N^{\frac{1}{2}} p^{\beta - y}}{f_{\beta}^{1 - \delta}} 
\sum_{m \geq 1 \atop (m,p)=1} \frac{a(m)}{m} \Kl_{n-1}(\pm m, p^{\beta -y}) 
\int_{(-\sigma)} \left( \frac{m f_{\beta}}{N p^{n(\beta -y)}} \right)^s \frac{k(-s + (1 - \delta))}{s - (1 - \delta)} F(-s +1) \overline{F}(s) \frac{ds}{2 \pi i} \\
&= \frac{W(\widetilde{\pi}) \overline{\omega}(p^{\beta -y}) N^{\frac{1}{2}} p^{\beta - y}}{f_{\beta}^{1 - \delta}} 
\sum_{m \geq 1 \atop (m,p)=1} \frac{a(m)}{m} \Kl_{n-1}(\pm m, p^{\beta -y}) 
\int_{(-\sigma)} \left( \frac{m f_{\beta}}{N p^{n(\beta -y)}} \right)^s \frac{k(-s + (1 - \delta))}{s - (1 - \delta)} \frac{ds}{2 \pi i}. \end{align*}
Here again, we use that $F(-s +1) \overline{F}(s) = 1$. Since $f_{\beta} = Np^{n \beta -u}$, the latter integral expression equals
\begin{align*} &\frac{W(\widetilde{\pi}) \overline{\omega}(p^{\beta -y}) N^{\frac{1}{2}} p^{\beta - y}}{(N p^{n \beta -u})^{1 - \delta}} 
\sum_{m \geq 1 \atop (m,p)=1} \frac{a(m)}{m} \Kl_{n-1}(\pm m, p^{\beta -y}) 
\int_{(-\sigma)} \left( \frac{m N p^{n \beta -u}}{N p^{n(\beta -y)}} \right)^s \frac{k(-s + (1 - \delta))}{s - (1 - \delta)} \frac{ds}{2 \pi i} \\ 
&= \frac{W(\widetilde{\pi}) \overline{\omega}(p^{\beta -y}) N^{\frac{1}{2}} p^{\beta - y} p^{u(1 - \delta)}}{(N p^{n \beta})^{1 - \delta}} 
\sum_{m \geq 1 \atop (m,p)=1} \frac{a(m)}{m} \Kl_{n-1}(\pm m, p^{\beta -y}) \Phi_u(p^{ny} m). \end{align*}
Hence, putting this latter expression together with the residue term, we have shown (for $1 \leq y \leq \beta-2$) that 
\begin{align*} S_{2, y}^{\star} 
= \frac{2}{\varphi(p^{\beta - y})} &W(\widetilde{\pi}) \overline{\omega}(p^{\beta -y}) N^{\delta - \frac{1}{2}} p^{(\beta - y)(1 - n (1 - \delta))}
\seveny \tau(\overline{\chi})^{n-1} L(\delta, \pi \otimes \chi) \\ 
&+ \frac{W(\widetilde{\pi}) \overline{\omega}(p^{\beta -y}) N^{\frac{1}{2}} p^{\beta - y} p^{u(1 - \delta)}}{(N p^{n \beta})^{1 - \delta}} 
\sum_{m \geq 1 \atop (m,p)=1} \frac{a(m)}{m} \Kl_{n-1}(\pm m, p^{\beta -y}) \Phi_u(p^{ny} m).\end{align*}
It then follows (from the definition) that 
\begin{align*} S_{2, y} &= \frac{p}{\varphi(p)} \frac{W(\pi) \omega(p^{\beta}) (N p^{n \beta})^{\frac{1}{2} - \delta}}{p^{\frac{3 \beta}{2}}} 
\cdot \frac{2}{\varphi(p^{\beta - y})} W(\widetilde{\pi}) \overline{\omega}(p^{\beta -y}) N^{\delta - \frac{1}{2}} p^{(\beta - y)(1 - n (1 - \delta))}
\sum_{x \m p^{\beta} \atop (\frac{x}{p})_n = 1} \sum_{w \m p^{\alpha} \atop w^n \equiv x \m p^{\alpha}} e \left( \frac{ (n-1)w  + x \overline{w}}{p^{\beta}}\right)  \\ 
&\times \psi_{p^y}(-x) \seveny \tau(\overline{\chi})^{n-1} L(\delta, \pi \otimes \chi) \\ 
&+ \frac{p}{\varphi(p)} \frac{W(\pi) \omega(p^{\beta}) (N p^{n \beta})^{\frac{1}{2} - \delta}}{p^{\frac{3 \beta}{2}}} \cdot 
\frac{W(\widetilde{\pi}) \overline{\omega}(p^{\beta -y}) N^{\frac{1}{2}} p^{\beta - y} p^{u(1 - \delta)}}{(N p^{n \beta})^{1 - \delta}} 
\sum_{x \m p^{\beta} \atop (\frac{x}{p})_n = 1} \sum_{w \m p^{\alpha} \atop w^n \equiv x \m p^{\alpha}} e \left( \frac{ (n-1)w  + x \overline{w}}{p^{\beta}}\right)  \\ 
&\times \psi_{p^y}(-x) \sum_{m \geq 1 \atop (m,p)=1} \frac{a(m)}{m} \Kl_{n-1}(\pm m, p^{\beta -y}) \Phi_u(p^{ny} m). \end{align*} 
Now, we can simplify this latter expression by grouping together (and cancelling out) like scalar terms, 
using that $W(\widetilde{\pi}) = \overline{W(\pi)}$ (so that $W(\pi) W(\widetilde{\pi}) = \vert W(\pi) \vert^2 = 1$), 
that $\omega(p^{\beta}) \overline{\omega}(p^{\beta - y}) = \omega(p^{\beta}) \overline{\omega}(p^{\beta}) \overline{\omega}(\overline{p^{y}}) = \omega(p^{y})$, 
and that the remaining scalar terms can be simplified as in $(\ref{factors})$ above (since $\beta - y \geq 2$). Hence, we obtain
\begin{align*} S_{2, y} &= \frac{p^{\beta(1 - \frac{n}{2})}}{p^{\frac{3\beta}{2}}} \sum_{x \m p^{\beta} \atop (\frac{x}{p})_n = 1} 
\sum_{w \m p^{\alpha} \atop w^n \equiv x \m p^{\alpha}} e \left( \frac{ (n-1)w  + x \overline{w}}{p^{\beta}}\right) 
\frac{\omega(p^y)  \psi_{p^y}(-x)}{p^y} \\ 
&\times \left( p^{ny (1 - \delta)} \cdot \frac{2}{\varphi^{\star}(p^{\beta - y})}  \seveny \tau(\overline{\chi})^{n-1} L(\delta, \pi \otimes \chi)
+ p^{u(1 - \delta)} \cdot \frac{p}{\varphi(p)} \cdot \sum_{m \geq 1 \atop (m,p)=1} \frac{a(m)}{m} \Kl_{n-1}(\pm m, p^{\beta -y}) \Phi_u(p^{ny} m) \right). \end{align*}
Let us now consider the case of $y = \beta -1$ (corresponding to the case of prime modulus), starting with 
\begin{align*} S_{2, \beta-1}^{\star} 
&= \sum_{m \geq 1 \atop (m,p)=1} \overline{a(m)} \psi_{p^{\beta-1}}(\pm m \overline{N}) \phi_{\infty}(m) 
= \sum_{m \geq 1 \atop (m,p)=1} \overline{a(m)} \Kl_1(\pm m \overline{N}, p) \phi_{\infty}(m). \end{align*}
Once again, we use the result of Proposition \ref{mellin}, which for any choice of real number $1 < \sigma < 3 - \Re(\delta)$ gives 
\begin{align*} S_{2, \beta-1}^{\star} &= \int_{(\sigma)} D(\widetilde{\pi}, \overline{N}, p, s) \phi_{\infty}^*(s) \frac{ds}{2 \pi i}. \end{align*}
Shifting the range of integration to $\Re(s) = -\sigma$, we cross poles at $s = \overline{\mu}_j$ for each $1 \leq j \leq n$ of vanishing
residues (thanks to Lemma \ref{test}), as well as a simple pole at $s= 1 - \delta$ of residue 
\begin{align*} \Res_{s=1-\delta} \left( D(\widetilde{\pi}, \overline{N}, p, s) \phi_{\infty}^*(s) \right) 
&= D(\widetilde{\pi}, \overline{N}, p, 1 - \delta) F(\delta). \end{align*} 
Again, we can compute this residue term via analytic continuation as in $(\ref{r1p})$ above to obtain
\begin{align*} &D(\widetilde{\pi}, \overline{N}, p, 1 - \delta) F(\delta) \\ 
&= \left[ \frac{2}{p-3} W(\widetilde{\pi}) \overline{\omega}(p) N^{\delta - \frac{1}{2}} \overline{F}(1 - \delta) 
\left( p^{1 - n(1 - \delta)} \sevenp \tau(\overline{\chi})^{n-1} L(\delta, \pi \otimes \chi) - \overline{\epsilon}_p(1 - \delta)L(\delta, \pi) \right) \right] F(\delta) \\ 
&= \frac{2}{p-3} W(\widetilde{\pi}) \overline{\omega}(p) N^{\delta - \frac{1}{2}} 
\left( p^{1 - n(1 - \delta)} \sevenp \tau(\overline{\chi})^{n-1} L(\delta, \pi \otimes \chi) - \overline{\epsilon}_p(1-\delta)L(\delta, \pi) \right), \end{align*}
where we use the cancellation of archimedean factors $F(\delta) \overline{F}(1 - \delta) = 1$. To evaluate the remaining integral
\begin{align*} \int_{(-\sigma)} D(\widetilde{\pi}, \overline{N}, p, s) \phi_{\infty}^*(s) \frac{ds}{2 \pi i}, \end{align*}
we apply the additive functional identity of Proposition \ref{AFI} (ii) to $D(\widetilde{\pi}, \overline{N}, p, s)$ to obtain
\begin{align*} &\int_{(- \sigma)}  W(\widetilde{\pi}) \overline{\omega}(p) N^{\frac{1}{2} - s} \overline{F}(s) 
\sum_{m \geq 1 \atop (m,p)=1}\frac{a(m)}{m} \left( \Kl_{n-1}(\pm m, p) + (-1)^n \frac{2}{p-3} 
\left[ 1 - \frac{\overline{\epsilon}_p(s) \epsilon_p(1 -s)}{p^{1 - ns}} \right] \right) \phi_{\infty}^*(s) \frac{ds}{2 \pi i } \\
&= W(\widetilde{\pi}) \overline{\omega}(p)N^{\frac{1}{2}} p \sum_{m \geq 1 \atop (m, p)=1} \frac{a(m)}{m} 
\left( \Kl_{n-1}(\pm m, p) +(-1)^n \frac{2}{p-3} \right) \int_{(-\sigma)} \left( \frac{m}{N p^n} \right)^s \overline{F}(s) \phi_{\infty}^*(s) \frac{ds}{2 \pi i} \\
&-  W(\widetilde{\pi}) \overline{\omega}(p)N^{\frac{1}{2}} p \cdot \frac{1}{p} (-1)^n \frac{2}{p-3} \sum_{m \geq 1 \atop (m,p)=1} \frac{a(m)}{m} 
\int_{(-\sigma)} \left( \frac{m p^n}{N p^n} \right)^s \overline{\epsilon}_p(s) \overline{F}(s) \phi_{\infty}^*(s) \frac{ds}{2 \pi i}. \end{align*}
Expanding out the definition of $\phi_{\infty}^*(s)$, this latter expression is given more explicitly by 
\begin{align*}
& \frac{ W(\widetilde{\pi}) \overline{\omega}(p)N^{\frac{1}{2}} p}{f_{\beta}^{1 - \delta}} \sum_{m \geq 1 \atop (m, p)=1} \frac{a(m)}{m} 
\left( \Kl_{n-1}(\pm m, p) +(-1)^n \frac{2}{p-3} \right) \int_{(-\sigma)} \left( \frac{m f_{\beta}}{N p^n } \right)^s \overline{F}(s) 
\frac{k(-s + (1 - \delta))}{s - (1 - \delta)} F(-s +1) \frac{ds}{2 \pi i} \\
&-  \frac{W(\widetilde{\pi}) \overline{\omega}(p)N^{\frac{1}{2}} p}{f_{\beta}^{1-\delta}} \cdot \frac{1}{p} (-1)^n \frac{2}{p-3} \sum_{m \geq 1 \atop (m,p)=1} \frac{a(m)}{m} 
\int_{(-\sigma)} \left( \frac{m f_{\beta}}{N} \right)^s \overline{F}(s) \overline{\epsilon}_p(s) \frac{k(-s + (1 - \delta))}{s - (1 - \delta)} F(-s +1) \frac{ds}{2 \pi i}, \end{align*} 
which after using (again) that $\overline{F}(s) F(1 - s) = 1$ is the same as 
\begin{align*}& \frac{ W(\widetilde{\pi}) \overline{\omega}(p)N^{\frac{1}{2}} p}{f_{\beta}^{1 - \delta}} \sum_{m \geq 1 \atop (m, p)=1} \frac{a(m)}{m} 
\left( \Kl_{n-1}(\pm m, p) +(-1)^n \frac{2}{p-3} \right) \int_{(-\sigma)} \left( \frac{m f_{\beta}}{N p^n} \right)^s 
\frac{k(-s + (1 - \delta))}{s - (1 - \delta)} \frac{ds}{2 \pi i} \\
&-  \frac{W(\widetilde{\pi}) \overline{\omega}(p)N^{\frac{1}{2}} p}{f_{\beta}^{1-\delta}} \cdot \frac{1}{p} (-1)^n \frac{2}{p-3} \sum_{m \geq 1 \atop (m,p)=1} \frac{a(m)}{m} 
\int_{(-\sigma)} \left( \frac{m f_{\beta}}{N} \right)^s \overline{\epsilon}_p(s) \frac{k(-s + (1 - \delta))}{s - (1 - \delta)} \frac{ds}{2 \pi i}. \end{align*} 
Expanding out the scalar contribution $f_{\beta} = Np^{n \beta - u}$ then gives us the even more explicit expression
\begin{align*}& \frac{ W(\widetilde{\pi}) \overline{\omega}(p)N^{\frac{1}{2}} p}{(N p^{n \beta -u})^{1 - \delta}} 
\sum_{m \geq 1 \atop (m, p)=1} \frac{a(m)}{m} 
\left( \Kl_{n-1}(\pm m, p) +(-1)^n \frac{2}{p-3} \right) \int_{(-\sigma)} \left( \frac{m Np^{n \beta -u} }{N p^n} \right)^s 
\frac{k(-s + (1 - \delta))}{s - (1 - \delta)} \frac{ds}{2 \pi i} \\
&-  \frac{ W(\widetilde{\pi} ) \overline{\omega}(p) N^{\frac{1}{2}} p}{(N p^{n \beta -u})^{1-\delta}} \cdot \frac{1}{p} (-1)^n \frac{2}{p-3} 
\sum_{m \geq 1 \atop (m, p)=1} \frac{a(m)}{m} 
\int_{(-\sigma)} \left( \frac{m N p^{n \beta -u} }{N} \right)^s \frac{k(-s + (1 - \delta))}{s - (1 - \delta)} \frac{ds}{2 \pi i}, \end{align*} 
from which we derive that
\begin{align*} &\int_{(-\sigma)} D(\widetilde{\pi}, \overline{N}, p, s) \phi_{\infty}^*(s) \frac{ds}{2 \pi i} \\
&=  \frac{ W(\widetilde{\pi}) \overline{\omega}(p)N^{\frac{1}{2}} p}{(N p^{n \beta -u})^{1 - \delta}} \left[
\sum_{m \geq 1 \atop (m, p) =1} \frac{a(m)}{m} \left( \Kl_{n-1}(\pm m, p) + (-1)^n \frac{2}{p-3} \right) \Phi_u( p^{n(\beta-1)}m)
- \frac{1}{p} (-1)^n \frac{2}{p-3} \sum_{m \geq 1 \atop (m, p)=1} \frac{a(m)}{m} \widetilde{\Phi}_u(p^{n \beta} m) \right].\end{align*}
Putting this together with the residue term then gives the formula 
\begin{align*} &S_{2, \beta-1}^{\star} = \frac{2}{p-3} W(\widetilde{\pi}) \overline{\omega}(p) N^{\delta - \frac{1}{2}} 
\left( p^{1 - n(1 - \delta)} \sevenp \tau(\overline{\chi})^{n-1} L(\delta, \pi \otimes \chi) - \overline{\epsilon}_p(1-\delta)L(\delta, \pi) \right) \\
&+ \frac{ W(\widetilde{\pi}) \overline{\omega}(p)N^{\frac{1}{2}} p}{(N p^{n \beta -u})^{1 - \delta}} \left[
\sum_{m \geq 1 \atop (m, p) =1} \frac{a(m)}{m} \left( \Kl_{n-1}(\pm m, p) + (-1)^n \frac{2}{p-3} \right) \Phi_u( p^{n(\beta-1)}m)
- \frac{1}{p} (-1)^n \frac{2}{p-3} \sum_{m \geq 1 \atop (m, p)=1} \frac{a(m)}{m} \widetilde{\Phi}_u(p^{n \beta} m) \right], \end{align*}
from which we derive
\begin{align*} &S_{2, \beta-1} := \frac{p}{\varphi(p)} \frac{W(\pi) \omega(p^{\beta}) (N p^{n \beta})^{\frac{1}{2} - \delta}}{p^{\frac{3 \beta}{2}}}
\sum_{x \m p^{\beta} \atop (\frac{x}{p})_n =1} \sum_{w \m p^{\alpha} \atop w^n \equiv x \m p^{\alpha}} e \left( \frac{ (n-1)w  + x \overline{w}}{p^{\beta}}\right) 
\psi_{p^{\beta-1}}(-x) S_{2, \beta-1}^{\star} \\
&= \frac{p}{\varphi(p)} \frac{W(\pi) \omega(p^{\beta}) (N p^{n \beta})^{\frac{1}{2} - \delta}}{ p^{\frac{3 \beta}{2}} } 
\sum_{x \m p^{\beta} \atop (\frac{x}{p})_n =1} \sum_{w \m p^{\alpha} \atop w^n \equiv x \m p^{\alpha}} e \left( \frac{ (n-1)w  + x \overline{w}}{p^{\beta}}\right)   
\psi_{p^{\beta-1}}(-x) \\
&\times \frac{2}{p-3} W(\widetilde{\pi}) \overline{\omega}(p) N^{\delta - \frac{1}{2}} 
\left( p^{1 - n(1 - \delta)} \sevenp \tau(\overline{\chi})^{n-1} L(\delta, \pi \otimes \chi) - \overline{\epsilon}_p(1-\delta)L(\delta, \pi) \right)  \\
&+ \frac{p}{\varphi(p)} \frac{W(\pi) \omega(p^{\beta}) (N p^{n \beta})^{\frac{1}{2} - \delta}}{ p^{\frac{3 \beta}{2}} } 
\sum_{x \m p^{\beta} \atop (\frac{x}{p})_n =1} \sum_{w \m p^{\alpha} \atop w^n \equiv x \m p^{\alpha}} e \left( \frac{ (n-1)w  + x \overline{w}}{p^{\beta}}\right) 
\psi_{p^{\beta-1}}(-x) \\ 
&\times  \frac{ W(\widetilde{\pi}) \overline{\omega}(p)N^{\frac{1}{2}} p}{(N p^{n \beta -u})^{1 - \delta}} \left[
\sum_{m \geq 1 \atop (m, p) =1} \frac{a(m)}{m} \left( \Kl_{n-1}(\pm m, p) + (-1)^n \frac{2}{p-3} \right) \Phi_u( p^{n(\beta-1)}m) 
- \frac{1}{p} (-1)^n \frac{2}{p-3} \sum_{m \geq 1 \atop (m, p)=1} \frac{a(m)}{m} \widetilde{\Phi}_u(p^{n \beta} m) \right], \end{align*}
which after grouping together and cancelling out like scalar terms is the same as 
\begin{align*} &S_{2, \beta-1} = \frac{p}{\varphi(p)} \frac{1}{p^{\frac{3 \beta}{2}}} 
\sum_{x \m p^{\beta} \atop (\frac{x}{p})_n =1} \sum_{w \m p^{\alpha} \atop w^n \equiv x \m p^{\alpha}} 
e \left( \frac{ (n-1) w  + x \overline{w}}{p^{\beta}}\right) \cdot \psi_{p^{\beta-1}}(x)\omega(p^{\beta-1}) \\
& \times p^{n \beta (\frac{1}{2} - \delta)} \left( \frac{2}{p-3} 
\left( p^{1 - n(1 - \delta)} \sevenp \tau(\overline{\chi})^{n-1} L(\delta, \pi \otimes \chi) - \overline{\epsilon}_p(1-\delta)L(\delta, \pi) \right) \right. \\
&\left. + \frac{N^{\frac{1}{2}} p}{(Np^{n\beta})^{ \frac{1}{2} }} p^{u(1 - \delta)} 
\left[ \sum_{m \geq 1 \atop (m, p) =1} \frac{a(m)}{m} \left( \Kl_{n-1}(\pm m, p) + (-1)^n \frac{2}{p-3} \right) \Phi_u( p^{n(\beta-1)}m)
- \frac{1}{p} (-1)^n \frac{2}{p-3} \sum_{m \geq 1 \atop (m, p)=1} \frac{a(m)}{m} \widetilde{\Phi}_u(p^{n \beta} m) \right] \right) \\ 
&= \frac{p^{\beta(1 - \frac{n}{2})}}{p^{\frac{3 \beta}{2}}} \sum_{x \m p^{\beta} \atop (\frac{x}{p})_n =1}
\sum_{w \m p^{\alpha} \atop w^n \equiv x \m p^{\alpha}} e \left( \frac{ (n-1)w  + x \overline{w}}{p^{\beta}}\right) 
\cdot \frac{\psi_{p^{\beta-1}}(x)\omega(p^{\beta-1})}{p^{\beta-1}} \\
& \times \frac{p}{\varphi(p)} \left( \frac{2}{p-3}
\left( p^{n \delta} \sevenp \tau(\overline{\chi})^{n-1} L(\delta, \pi \otimes \chi) - p^{n - 1} \overline{\epsilon}_p(1-\delta)L(\delta, \pi) \right) \right. \\
&\left. + p^{u(1 - \delta)} \left[ \sum_{m \geq 1 \atop (m, p) =1} \frac{a(m)}{m} \left( \Kl_{n-1}(\pm m, p) + (-1)^n \frac{2}{p-3} \right) \Phi_u( p^{n(\beta-1)}m)
- \frac{1}{p} (-1)^n \frac{2}{p-3} \sum_{m \geq 1 \atop (m, p)=1} \frac{a(m)}{m} \widetilde{\Phi}_u(p^{n \beta} m) \right] \right). \end{align*}

Putting together all of the pieces (separating out residues), we derive the stated formula. \end{proof}

\begin{corollary}\label{VSF4} 

Keep the hypotheses of Theorem \ref{VSF3} above. We also have the summation formula 
\begin{align*}  X_{\beta, 2}(\pi, \delta, p^u) &=
\frac{1}{p^{\beta n}} \sum_{x \m p^{\beta} \atop (\frac{x}{p})_n=1} \Kl_n(x, p^{\beta}) 
\left\lbrace \frac{2}{\varphi^{\star}(p^{\beta})} \seven \chi(-x) \tau(\overline{\chi})^n L(\delta, \pi \otimes \chi) \right. \\ 
&\left. + \sum_{1 \leq y \leq \beta-2} \frac{   \omega(p^y) \psi_{p^y}(-x)}{p^y} 
p^{ny(1 - \delta)} \frac{2}{\varphi^{\star}(p^{\beta-y})} \seveny \tau(\overline{\chi})^{n-1} L(\delta, \pi \otimes \chi) \right. \\ 
&\left. + \frac{\omega(p^{\beta-1}) \psi_{p^{\beta-1}}(-x)}{p^{\beta-1}} \frac{p}{\varphi(p)} \frac{2}{p-3} 
\left( p^{n \delta} \sevenp \tau(\overline{\chi})^{n-1} L(\delta, \pi \otimes \chi) - p^{n-1} \overline{\epsilon}_p(1- \delta) L(\delta, \pi) \right) \right. \\
&\left. + p^{u(1 - \delta)} \left( \mathfrak{S}_{1, x} + \mathfrak{S}_{2, x} + \mathfrak{S}_{3, x} \right) \right\rbrace. \end{align*}  \end{corollary} 

\begin{proof} We see a direct substitution of the formula of Proposition \ref{salie} above to derive the stated formula. \end{proof}

Using this latter summation formula, we can now derive the following simplification.

\begin{lemma}\label{hKsum} We have the following identity for any exponent $\beta \geq 4$ and any integer $n \geq 2$:

\begin{align*} \sum_{x \m p^{\beta} \atop (\frac{x}{p})_n = 1} \Kl_n(x, p^{\beta}) \Kl_n(\pm mx, p^{\beta}) 
&= p^{\beta n} \frac{2}{\varphi(p^{\beta})} \seven \overline{\chi}(m) \end{align*} \end{lemma}

\begin{proof} Since $\beta \geq 4$, we argue that the $x$-sum is the same as the sum over all coprime classes $x \m p^{\beta}$, 
i.e.~as the sum is supported only classes $x \m p^{\beta}$
such that $(\frac{x}{p})_n=1$ (by Proposition \ref{salie}). Thus, we have
\begin{align*}  \sum_{x \m p^{\beta} \atop (\frac{x}{p})_n = 1} \Kl_n(x, p^{\beta}) \Kl_n(\pm mx, p^{\beta}) &=
\sum_{x \m p^{\beta} \atop (x, p^{\beta})=1} \Kl_n(x, p^{\beta}) \Kl_n(\pm mx, p^{\beta}), \end{align*}
which after applying Lemma \ref{SOGS} to describe each of the sums $\Kl_n(\pm mx, p^{\beta})$ is the same as 
\begin{align*} \frac{2}{\varphi(p^{\beta})} \sum_{x \m p^{\beta} \atop (x, p^{\beta})=1} 
\Kl_n(x, p^{\beta}) \seven \overline{\chi}(mx) \tau(\chi)^n. \end{align*}
Switching the order of summation, and opening up each of the sums $\Kl_n(x, p^{\beta})$, we obtain
\begin{align*} 
\frac{2}{\varphi(p^{\beta})} \seven \overline{\chi}(m) \tau(\chi)^n \sum_{y_1, \cdots, y_{n-1} \m p^{\beta}} e \left( \frac{y_1 + \cdots y_{n-1}}{p^{\beta}} \right)
\sum_{x \m p^{\beta} \atop (x, p^{\beta})=1} \overline{\chi}(x) e \left( \frac{x \overline{y_1 \cdots y_{n-1}}}{p^{\beta}} \right). \end{align*}
Changing variables to evaluate the inner $x$-sum as 
\begin{align*} \sum_{x \m p^{\beta} \atop (x, p^{\beta})=1} \overline{\chi}(x) e \left( \frac{x \overline{y_1 \cdots y_{n-1}}}{p^{\beta}} \right) &= 
\overline{\chi}(y_1 \cdots y_{n-1}) \tau(\overline{\chi}) = \chi(y_1 \cdots y_{n-1}) \tau(\overline{\chi}),\end{align*}
we then obtain 
\begin{align*} 
\frac{2}{\varphi(p^{\beta})} \sum_{y_1, \cdots, y_{n-1} \m p^{\beta}} e \left( \frac{y_1 + \cdots y_{n-1}}{p^{\beta}} \right) 
\seven \overline{\chi}(m y_1 \cdots y_{n-1}) \tau(\chi)^n \tau(\overline{\chi}), \end{align*}
which after using that $\tau(\chi)^n \tau(\overline{\chi}) = \tau(\chi)^{n-1} \vert \tau(\chi)\vert^2 = \tau(\chi)^{n-1} p^{\beta} $ is the same as 
\begin{align*}  p^{\beta} \sum_{y_1, \cdots, y_{n-1} \m p^{\beta}} e \left( \frac{y_1 + \cdots y_{n-1}}{p^{\beta}} \right) 
\frac{2}{\varphi(p^{\beta})} \seven \overline{\chi}(m y_1 \cdots y_{n-1}) \tau(\chi)^{n-1}.\end{align*}
Switching the order of summation in this latter expression, we then compute 
\begin{align*} 
p^{\beta} \left( \frac{2}{\varphi(p^{\beta})} \right) \seven \overline{\chi}(m) \tau(\chi)^{n-1}  \tau(\overline{\chi})^{n-1}, \end{align*}
which after using that $\tau(\chi)^{n-1} \tau(\overline{\chi})^{n-1} = (\vert \tau(\chi) \vert^2)^{n-1} = p^{\beta(n-1)}$ gives the stated formula. \end{proof}

\begin{corollary}\label{VSFts} Corollary \ref{VSF4} gives us the following expression for the twisted sum $X_{\beta, 2}(\pi, \delta)$:
\begin{align*} \frac{2}{\varphi^{\star}(p^{\beta})} \seven L(\delta, \pi \otimes \chi) 
+ p^{u(1 - \delta)} \left( \sum_{m \geq 1 \atop m \equiv \pm 1 \m p^{\beta}} \frac{a(m)}{m} \Phi_u(m) 
- \frac{1}{\varphi(p)} \sum_{ {m \geq 1 \atop m \equiv \pm 1 \m p^{\beta-1}} \atop m \not\equiv \pm 1 \m p^{\beta}}  \frac{a(m)}{m} \Phi_u(m) \right). \end{align*}
Equivalently, we have for any exponent $\beta \geq 4$ and for any real parameter $u >0$ the average formula
\begin{align*} X_{\beta}(\pi, \delta) &= - p^{u(1 - \delta)} \left( \sum_{m \geq 1 \atop m \equiv \pm 1 \m p^{\beta}} \frac{a(m)}{m} \Phi_u(m) 
- \frac{1}{\varphi(p)} \sum_{ {m \geq 1 \atop m \equiv \pm 1 \m p^{\beta-1}} \atop m \not\equiv \pm 1 \m p^{\beta}}  \frac{a(m)}{m} \Phi_u(m) \right)
+ X_{\beta, 2}(\pi, \delta, p^u). \end{align*} \end{corollary}

\begin{proof} 
It is easy (and classical) to show that 
\begin{align}\label{hK2} \sum_{x \m p^{\beta} \atop (\frac{x}{p})_n = 1} \chi(x) \Kl_n(x, p^{\beta}) = \tau(\chi)^n.\end{align}
Using this identity $(\ref{hK2})$, it is then easy to see that 
\begin{align*} \frac{1}{p^{\beta n}} &\sum_{x \m p^{\beta} \atop (\frac{x}{p})_n = 1} \Kl_n(x, p^{\beta}) 
\frac{2}{\varphi^{\star}(p^{\beta})} \seven \chi(x) \tau(\overline{\chi})^n L(\delta, \pi \otimes \chi) \\ 
&= \frac{1}{p^{\beta n}} \frac{2}{\varphi^{\star}(p^{\beta})} \seven \tau(\chi)^n \tau(\overline{\chi})^n L(\delta, \pi \otimes \chi) 
= \frac{2}{\varphi^{\star}(p^{\beta})} \seven L(\delta, \pi \otimes \chi). \end{align*}
Here, in the last step, we use that $\tau(\chi)^n \tau(\overline{\chi})^n = (\vert \tau(\chi) \vert^2)^n = p^{\beta n}$. 
This gives the stated residue term for the formula. To evaluate each of the remaining terms in the expression 
of Corollary \ref{VSF4} after switching the order of summation in this way, we argue using the orthogonality 
of additive characters that each of the remaining terms except for the sums $\mathfrak{S}_{1, x}$ must vanish. 
To evaluate the sum over $\mathfrak{S}_{1,x}$, we apply Lemma \ref{hKsum}:
\begin{align*} \frac{1}{p^{\beta n}} \sum_{x \m p^{\beta} \atop (\frac{x}{p})_n = 1} \Kl_n(x, p^{\beta}) p^{u(1 - \delta)} \mathfrak{S}_{1, x} 
&= \frac{1}{p^{\beta n}}  \sum_{x \m p^{\beta} \atop (\frac{x}{p})_n = 1} \Kl_n(x, p^{\beta}) p^{u(1 - \delta)} \frac{p}{\varphi(p)} 
\sum_{m \geq \atop (m, p)=1} \frac{a(m)}{m} \Kl_n(\pm mx, p^{\beta}) \Phi_u(m) \\ 
&= \frac{1}{p^{\beta n}} p^{u(1 - \delta)} \frac{p}{\varphi(p)} \sum_{m \geq 1 \atop (m, p)=1} \frac{a(m)}{m} 
\left(  \sum_{x \m p^{\beta} \atop (\frac{x}{p})_n = 1} \Kl_n(x, p^{\beta}) \Kl_n(\pm mx, p^{\beta}) \right) \Phi_u(m) \\
&= p^{u(1 - \delta)} \frac{2}{\varphi^{\star}(p^{\beta})}  \sum_{m \geq 1 \atop (m, p)=1} \frac{a(m)}{m} \left( \seven \overline{\chi}(m) \right) \Phi_u(m) \\
&= p^{u(1-\delta)}  \left( \sum_{m \geq 1 \atop (m, p)=1} \frac{a(m)}{m} \Phi_u(m) 
- \frac{1}{\varphi(p)} \sum_{ {m \geq 1 \atop m \equiv \pm 1 (p^{\beta-1})} \atop m \not\equiv \pm 1 \m p^{\beta}}  \frac{a(m)}{m} \Phi_u(m) \right). \end{align*}
Here, in the last step, we use $(\ref{QO})$ (as well as $(\ref{factors})$). This proves the stated formula for the twisted sum. \end{proof}

\subsection{Some estimates} 

We now determine the rate of decay of the dual function corresponding to the weight function $\phi_{\infty}$ defined on 
$y \in {\bf{R}}_{>0}$ by $\phi_{\infty}(y) := y^{-(1-\delta)}V_2(f_{\beta}^{-1}y)$ appearing in Proposition \ref{mellin}, 
where $f_{\beta} = N p^{\beta n -u}$ denotes the length of the region of non-negligible 
summation of $X_{\beta, 2}(\pi, \delta, p^u)$ as defined $(\ref{X2})$ above. 
Let us write $d  = \Re(\delta)$ and $\delta_0 = \max(\Re(\overline{\mu}_1), \Re(\overline{\mu}_2))$ to lighten notations. 

\begin{lemma}\label{archdual} 

Fixing a real parameter $u \in {\bf{R}}$ as above, let $\Phi_u$ denote the dual weight functions appearing in 
Theorems \ref{VSF2} and \ref{VSF3}. Hence, we let $\Phi_u$ denote the function defined on $y \in {\bf{R}}_{>0}$ by the integral transform
\begin{align*} \Phi_u(y) &= \int_{\Re(s) = - \sigma }  \frac{ k(-s + (1 - \delta) )}{s -(1 - \delta)} \left( \frac{y}{p^u} \right)^{s} \frac{ds}{2 \pi i} \end{align*} 
for $ 1 < \sigma < 2 + (1 - \delta)$. 
We have for any choice of constants $C > 0$ and $B \geq 1$ the bounds  
\begin{align*}\Phi_{u}(y) 
&= \begin{cases} O_{C}\left(  \left( \frac{ y }{ p^{u}} \right)^{-C} \right) &\text{ if $ y \geq {p^u}$, i.e.~as $\frac{ y }{p^{u}} \rightarrow \infty$} \\
- \left(\frac{y}{p^u} \right)^{1 - \delta} +  O_{B} \left( \left( \frac{y}{p^u}\right)^B \right) 
&\text{ if $ y \leq p^u$, i.e.~as $\frac{ y }{ p^{u}} \rightarrow 0$}. \end{cases} \end{align*} 
The modified weight functions $\widetilde{\Phi}_u(y)$ are estimated in a completely analogous way.
\end{lemma}

\begin{proof} 

We estimate the integral by a variation of the standard contour argument used to derive Lemma \ref{RD} above.  
Let us simplify the discussion by writing $x = y p^{-u}$. Hence, the task is to estimate the integral
\begin{align*} \int_{(-\sigma)} \frac{ k(-s + (1 - \delta) )}{s -(1 - \delta)} x ^{s} \frac{ds}{2 \pi i}. \end{align*}
To estimate the behaviour as $x \rightarrow \infty$, we move the line of integration to the left to derive the bound  
\begin{align*} \Phi_u(y) &= O_C(x^{-C}) = O_C \left( (y p^{-u})^{-C} \right) \text{~~~~~ for any choice of $C >0$}. \end{align*} 

To estimate the behaviour as $ x \rightarrow 0$, we move to the right, crossing a simple pole at $s = 1 - \delta$ of residue
\begin{align*} -\Res_{s=1 - \delta} \left( \frac{k(-s + (1 - \delta))}{s - (1 - \delta)} x^s \right) &= -x^{1 - \delta}. \end{align*} 
The remaining integral is then seen easily to be bounded as $O_B(x^B)$ for any choice of constant $B \geq 1$ to derive 
the stated estimate in this region. \end{proof}


We now at last return to the issue of bounding the twisted sum $X_{\beta, 2}(\pi, \delta, p^u)$, with notations and conventions
as above (so that $0 < u < \beta - 1$ is our fixed real parameter). 

\begin{lemma}\label{twistedbound} 

Taking any choice of real parameter $0 < u < \beta - 1$, we have for any choice constant $C>0$ 
\begin{align*} X_{\beta, 2}(\pi, \delta, p^u) 
&= - 1 + \frac{2}{\varphi^{\star}(p^{\beta})} \seven L(\delta, \pi \otimes \chi) + O_{C} \left( p^{u(1 - d +C) } p^{\beta(\theta - (1 - \Re(\delta)) - C)} ) \right), \end{align*} 
where $0 \leq \theta \leq 1/2 $ denotes the best known approximation towards the generalized Ramanujan conjecture for $\GLn({\bf{A}}_{\bf{Q}})$-automorphic forms. 
Equivalently, we have the estimate 
\begin{align*} X_{\beta}(\pi, \delta) &= 1 + X_{\beta, 2}(\pi, \delta, p^u) + O_{C} \left( p^{u(1 - d +C) } p^{\beta(\theta - (1 - \Re(\delta)) - C)} ) \right). \end{align*} \end{lemma}

\begin{proof} Using Corollary \ref{VSFts} above (derived from Theorem \ref{VSF3} and Corollary \ref{VSF4}), it will suffice to estimate
\begin{align*} p^{u(1 - \delta)}  \left( \sum_{m \geq 1 \atop (m, p)=1} \frac{a(m)}{m} \Phi_u(m) 
- \frac{1}{\varphi(p)} \sum_{ {m \geq 1 \atop m \equiv \pm 1 \m p^{\beta-1}} \atop m \not\equiv \pm 1 \m p^{\beta}}  \frac{a(m)}{m} \Phi_u(m) \right). \end{align*}
Since $0 < u < \beta -1$, the description of the decay of the weight function $\Phi_u$
in Lemma \ref{archdual} implies that the only contribution in the region of moderate decay comes from $m=1$, this being
\begin{align*} p^{u(1 - \delta)} \Phi_u(1) &= p^{u(1 - \delta)} \left(  -\left( \frac{1}{p^u} \right)^{1 - \delta}  + O_B \left( p^{-uB} \right) \right)
= -1 + O_B\left( p^{u(1 - d -B) } \right) \end{align*} for any choice of $B \geq 1$.
Using a variation of the argument given for Lemma \ref{lower} above, with Lemma \ref{archdual} in place of Lemma \ref{RD}, 
we see that each of the remaining contributions $m \equiv \pm 1 \m p^{\beta}$ is bounded above by 
\begin{align*} p^{u(1 - \delta)}m^{\theta-1 - C}p^{uC} &= O_{C, \theta} \left(p^{ u(1 - d + C) } p^{\beta(\theta - (1 - \Re(\delta)) -C)} \right) \end{align*}
for any choice of constant $C >0$. Since the sum over contributions will be 
dominated by least $m \geq 2$ such that $m \equiv \pm 1 \m p^{\beta}$, 
we obtain the stated bound after taking $B \geq 1 - \delta$ to be sufficiently large. \end{proof}




\subsection{Some remarks on hyper-Kloosterman Dirichlet series}\label{direct}

Let us now explain how we could have worked directly with the hyper-Kloosterman Dirichlet series 
$\mathfrak{K}_n(\pi, h, p^{\beta}, s)$ to establish a relevant Voronoi summation formula via the additive functional
identity $\ref{AFIhK}$ to describe the twisted sum $X_{\beta, 2}(\pi, \delta, p^u)$. 

\begin{theorem}\label{VSFK} 

Let $\phi_{\infty}$ denote the function defined on $y \in {\bf{R}}_{>0}$ by $\phi_{\infty}(y) = y^{-(1 - \delta)} V_2(f_{\beta}^{-1}y)$ as above
(where $f_{\beta} = Np^{n \beta -u}$),
and let $\Phi_u$ denote the integral transform defined in Theorem \ref{VSF2} (cf.~Lemma \ref{archdual}).
We have for any coprime residue class $h \m p^{\beta}$ the Voronoi summation formula 
\begin{align*}  \sum_{m \geq 1 \atop (m, p^{\beta})=1} &\overline{a(m)} \Kl_n(\pm mh, p^{\beta}) \phi_{\infty}(m) 
=W(\widetilde{\pi}) \overline{\omega}(p^{\beta}) N^{\delta - \frac{1}{2} } p^{\beta n \delta} \cdot  \frac{2}{\varphi(p^{\beta})} 
\seven \chi(h N) L(\delta, \pi \otimes \overline{\chi}) \\ 
&+  W(\widetilde{\pi}) \overline{\omega}(p^{\beta}) N^{\delta - \frac{1}{2}} p^{n \beta \delta} \cdot p^{u(1 - \delta)}
\left( \frac{\varphi(p)}{p}  \sum_{m \geq 1 \atop m \equiv \pm h N \m p^{\beta}} \frac{ a(m)}{m} \Phi_u(m)
-  \frac{1}{p} \sum_{ {m \geq 1 \atop m \equiv \pm hN \m p^{\beta-1}} \atop m \not\equiv \pm h N \m p^{\beta}} \frac{a(m)}{m} \Phi_u(m) \right).
\end{align*}
\end{theorem}

\begin{proof} 

Using Proposition \ref{mellin} above, we have for any choice of $1 < \sigma < 3 - \Re(\delta)$ the integral presentation 
\begin{align*} \sum_{m \geq 1 \atop (m, p)=1} \overline{a(m)} \Kl_n( \pm mh, p^{\beta}) \phi_{\infty}(m) 
&= \int_{(\sigma)} \mathfrak{K}_n(\widetilde{\pi}, h, p^{\beta}, s) \phi_{\infty}^*(s) \frac{ds}{2 \pi i}, \end{align*} where
\begin{align*} \phi_{\infty}^*(s) &= f_{\beta}^{s - (1 - \delta)} \frac{k(-s + (1 - \delta))}{s - (1 - \delta)} F(-s + 1) = 
f_{\beta}^{s - (1 - \delta)} \frac{k(-s + (1 - \delta))}{s - (1 - \delta)} \cdot \pi^{-\frac{n}{2} + n(-s +1)} 
\frac{\prod_{j=1}^n \Gamma \left(\frac{s - \overline{\mu}_j}{2} \right)}{\prod_{j=1}^n \Gamma \left( \frac{1 - s - \mu_j}{2} \right)}.\end{align*}
Shifting the line of integration to $\Re(s) = - \sigma$, we cross poles of vanishing residues at $s = \overline{\mu}_j $ for each $j = 1, \ldots, n$
(thanks to the construction of $k(s)$ in Lemma \ref{test} above), as well as a simple pole of residue 
\begin{align*} \Res_{s=(1 - \delta)} \left( \mathfrak{K}_n(\widetilde{\pi}, h, p^{\beta}, s) \phi_{\infty}^*(s) \right) 
&= \mathfrak{K}_n(\widetilde{\pi}, h, p^{\beta}, 1 - \delta) F(\delta). \end{align*}
Again, we can evaluate this residue via analytic continuation as in $(\ref{resk})$ (with $\overline{F}(1 - \delta) F(\delta) =1$) to derive  
\begin{align*} \mathfrak{K}_n(\widetilde{\pi}, h, p^{\beta}, 1 - \delta) F(\delta) 
&= \frac{2}{\varphi(p^{\beta})} W(\widetilde{\pi}) \overline{\omega}(p^{\beta}) N^{\delta - \frac{1}{2} } p^{\beta n \delta} 
\seven \chi(h N) L(\delta, \pi \otimes \overline{\chi}).\end{align*}
To evaluate the remaining integral
\begin{align*} \int_{(-\sigma)} \mathfrak{K}_n(\widetilde{\pi}, h, p^{\beta}, s) \phi_{\infty}^*(s) \frac{ds}{2 \pi i}, \end{align*}
we apply the additive functional identity 
\begin{align*} \mathfrak{K}_n(\widetilde{\pi}, h, p^{\beta}, s) &= W(\widetilde{\pi}) \overline{\omega}(p^{\beta}) N^{\frac{1}{2} - s} p^{n \beta(1-s)} \overline{F}(s) 
\left( \frac{\varphi(p)}{p} \sum_{m \geq 1 \atop m \equiv \pm h N \m p^{\beta}} \frac{a(m)}{m^{1- s}}
 - \frac{1}{p} \sum_{ {m \geq 1 \atop m \equiv \pm h N \m p^{\beta-1}} \atop m \not\equiv \pm h N \m p^{\beta}} \frac{a(m)}{m^{1-s}} \right) \end{align*}
of Proposition \ref{AFIhK} to obtain 
\begin{align*} &W(\widetilde{\pi}) \overline{\omega}(p^{\beta}) N^{\frac{1}{2}} p^{n \beta} \times \\ 
&\left( \frac{\varphi(p)}{p}  \sum_{m \geq 1 \atop m \equiv \pm h N \m p^{\beta}} \frac{a(m)}{m} 
\int_{(-\sigma)} \phi_{\infty}^*(s) \overline{F}(s) \left(\frac{m}{N p^{n\beta}} \right)^s\frac{ds}{2 \pi i}  
-  \frac{1}{p} \sum_{ {m \geq 1 \atop m \equiv \pm h N \m p^{\beta-1}} \atop m \not\equiv \pm h N \m p^{\beta}} \frac{a(m)}{m} 
\int_{(-\sigma)} \phi_{\infty}^*(s) \overline{F}(s) \left( \frac{m}{N p^{n \beta}} \right)^s \frac{ds}{2 \pi i} \right), \end{align*} 
which after expanding out the explicit definition of the Mellin transform $\phi_{\infty}^*(s)$ (as above) 
and using that $\overline{F}(s) F(-1 +s) = 1$ and that $f_{\beta} = Np^{n \beta -u}$, is the same as 
\begin{align*} \int_{(-\sigma)} \mathfrak{K}_n(\widetilde{\pi}, h, p^{\beta}, s) \phi_{\infty}^*(s) \frac{ds}{2 \pi i} &= 
\frac{W(\widetilde{\pi}) \overline{\omega}(p^{\beta}) N^{\frac{1}{2}} p^{n \beta}}{(N p^{n \beta - u})^{1 - \delta}} 
\left( \frac{\varphi(p)}{p}  \sum_{m \geq 1 \atop m \equiv \pm h N \m p^{\beta}} \frac{ a(m)}{m} \Phi_u(m)
-  \frac{1}{p} \sum_{ {m \geq 1 \atop m \equiv \pm hN \m p^{\beta-1}} \atop m \not\equiv \pm h N \m p^{\beta}} \frac{a(m)}{m} \Phi_u(m) \right).\end{align*}
Simplifying scalar terms, and putting this together with the residue, we obtain the stated formula. \end{proof}

Hence, we derive the same recursive formula for the average:

\begin{corollary} 

Assume that $\beta \geq 2$. The twisted sum $X_{\beta, 2}(f, \delta, p^u)$ defined in $(\ref{X2})$ 
above can be described equivalently for any choice of real parameter $u >0$ by
\begin{align*} \frac{2}{\varphi^{\star}(p^{\beta})} \seven L(\delta, \pi \otimes \overline{\chi}) 
+ p^{u(1 - \delta)} \left( \sum_{m \geq 1 \atop m \equiv \pm1 \m p^{\beta}} \frac{ a(m)}{m} \Phi_u(m)
-  \frac{1}{\varphi(p)} \sum_{ {m \geq 1 \atop m \equiv \pm 1 \m p^{\beta-1}} \atop m \not\equiv \pm 1 \m p^{\beta}} 
\frac{a(m)}{m} \Phi_u(m) \right).\end{align*} \end{corollary}

\begin{proof} The result is immediate after grouping together like scalar terms. \end{proof}

\section{Hyper-Kloosterman Dirichlet series at large}\label{hKL} 

We can now give the proofs of Theorems \ref{DAFI} and \ref{D} for the hyper-Kloosterman Dirichlet series $(\ref{HKDS})$:

\begin{proof}[Proof of Theorem \ref{DAFI} (A)] 

The first claim (i) appears in Proposition \ref{AFIhK}. For (ii), fix $s \in {\bf{C}}$ with $\Re(s) >1$. 
Expanding the absolutely convergent Dirichlet series and applying Lemma \ref{SOGS} (ii), we obtain 
\begin{align*} \K_n(\pi, h, p, s) &= \sum_{m \geq 1 \atop (m,p)=1} \frac{a(m)}{m^s} \Kl_n(\pm mh, p) 
= \frac{2}{p-3} \sum_{m \geq 1 \atop (m,p)=1} \frac{a(m)}{m^s} \left(  \sevenp \chi(mh) \tau(\overline{\chi})^n + (-1)^n \right), \end{align*}
which after switching the order of summation is the same as 
\begin{align*} \K_n(\pi, h, p, s) &= \frac{2}{p-3} \left( \sevenp \chi(h) \tau(\overline{\chi})^n L(s, \pi \otimes \chi) + (-1)^n L(s, \pi) \right). \end{align*}
Applying the functional equation $(\ref{fFE})$ to each of the $L$-functions $L(s, \pi \otimes \chi)$ and $L(s, \pi)$ then gives us  
\begin{align*} \K_n(\pi, h, p, s) &= \frac{2}{p-3} W(\pi) N^{\frac{1}{2} - s} F(s)\left( p^{n(1 -s)} \omega(p) \sevenp
\chi(hN) L(1 -s, \widetilde{\pi} \otimes \overline{\chi}) + (-1)^nL(1-s, \widetilde{\pi}) \right),\end{align*}
which (by the analytic continuation of $L(s, \pi \otimes \chi)$ and $L(s, \pi)$) is valid for any $s \in {\bf{C}}$. 
Let us now assume that $\Re(s) < 0$, in which case we can open up the absolutely convergent Dirichlet series 
\begin{align*} \sevenp \chi(hN) L(1-s, \widetilde{\pi} \otimes \overline{\chi}) 
&= \sevenp \chi(hN) \sum_{m \geq 1 \atop (m, p)=1} \frac{\overline{a(m)} \chi(\overline{m})}{m^{1-s}} 
= \sum_{m \geq 1 \atop (m, p)=1} \frac{\overline{a(m)}}{m^{1-s}} \sevenp \chi(h N \overline{m}) \end{align*}
in the latter expression. Evaluating the inner sum via the relation of Proposition \ref{QO} then gives us 
\begin{align*} \sum_{m \geq 1 \atop (m, p)=1} \frac{\overline{a(m)}}{m^{1-s}} \sevenp \chi(h N \overline{m}) 
&= \frac{p-3}{2} \sum_{m \geq 1 \atop m \equiv \pm hN \m p} \frac{\overline{a(m)}}{m^{1-s}} 
- \sum_{m \geq 1 \atop m \not\equiv \pm hN \m p} \frac{\overline{a(m)}}{m^{1-s}} . \end{align*}
Using this relation in the previous expression for $\K_n(\pi, h, p, s)$ then gives the stated functional identity. \end{proof}

\begin{proof}[Proof of Theorem \ref{DAFI} (B)] 

The proof in either case follows from Theorem \ref{DAFI} (A) via Mellin inversion, as in Theorem \ref{VSF}.
Hence for (i), choosing $\sigma \in {\bf{R}}_{>1}$ suitably so that $\phi(y) = \int_{(\sigma)} \phi^*(s) y^{-s} \frac{ds}{2 \pi i}$, we have that
\begin{align*} \sum_{m \geq 1 \atop (m,p)=1} a(m) \Kl_n(\pm mh, p^{\beta}) \phi(m) = 
\int_{(\sigma)} \phi^*(s) \sum_{m \geq 1 \atop (m, p)=1} \frac{a(m)}{m^s} \Kl_n(\pm hm, p^{\beta}, s) \frac{ds}{2 \pi i}
= \int_{(\sigma)} \phi^*(s) \mathfrak{K}_n(\pi, h, p^{\beta}, s) \frac{ds}{2 \pi i}. \end{align*}
Shifting the range of integration to $\Re(s) = -\sigma$, we then apply the additive functional 
identity of Theorem \ref{DAFI} (A) (i) to derive the stated formula. The proof of (ii) follow in the same way for Theorem \ref{DAFI} (B) (ii). \end{proof}

\begin{proof}[Proof of Theorem \ref{D} (A)]

Let us first consider (i), hence with $\beta \geq 2$. 
Taking $s \in {\bf{C}}$ with $\Re(s) >1$, we open up the absolutely convergent Dirichlet series and apply Lemma \ref{SOGS} to obtain the identification 
\begin{align*} \K_n^0(\xi, h, p^{\beta}) &= \sum_{m \geq 1} \frac{\xi(m)}{m^s} \Kl_n(\pm mh, p^{\beta}) 
= \frac{2}{\varphi(p^{\beta})} \sum_{m \geq 1 \atop (m, p)=1} \frac{\xi(m)}{m^s} \seven \chi(mh) \tau(\overline{\chi})^n. \end{align*}
Switching the order of summation, we then obtain
\begin{align*} \K_n^0(\xi, h, p^{\beta}, s) &= \frac{2}{\varphi(p^{\beta})} \seven \chi(h) \tau(\overline{\chi})^n L(s, \xi\chi). \end{align*}
Applying the classical functional equation 
\begin{align*} L(s, \xi\chi) &= (q p^{\beta})^{-s} \tau(\xi\chi) 
\left( \pi^{s - \frac{1}{2}} \frac{\Gamma \left( \frac{1-s}{2} \right)}{\Gamma \left( \frac{s}{2}\right)} \right) L(1-s, \overline{\xi \chi}) 
= (qp^{\beta})^{-s} \xi(p^{\beta}) \chi(q) \tau(\xi) \tau(\chi) 
\left( \pi^{s - \frac{1}{2}} \frac{\Gamma \left( \frac{1-s}{2} \right)}{\Gamma \left( \frac{s}{2}\right)} \right) L(1-s, \overline{\xi \chi}) \end{align*}
to this latter expression, we then obtain the identification
\begin{align*} \K_n^0(\xi, h, p^{\beta}, s) &= q^{-s} p^{\beta(1 -s)} \xi(p^{\beta})\chi(q) \tau(\xi)
\left( \pi^{s - \frac{1}{2}} \frac{\Gamma \left( \frac{1-s}{2} \right)}{\Gamma \left( \frac{s}{2}\right)} \right)
\cdot \frac{2}{\varphi(p^{\beta})} \seven \chi(hq) \tau(\overline{\chi})^{n-1} L(1-s, \overline{\xi \chi}),\end{align*}
which is valid for any $s \in {\bf{C}}$ (thanks to the analytic continuation of the Dirichlet series $L(s, \xi \chi)$).
Let us now consider this latter expression at a complex variable $s$ with $\Re(s) < 0$, where we can expand out as 
\begin{align*} \frac{2}{\varphi(p^{\beta})} \seven \chi(hq) \tau(\overline{\chi})^{n-1} L(1-s, \overline{\xi \chi})
&= \frac{2}{\varphi(p^{\beta})} \seven \chi(hq) \tau(\overline{\chi})^{n-1} \sum_{m \geq 1 \atop (m,p)=1} \frac{\overline{\xi \chi}(m)}{m^{1-s}} \\
&=\sum_{m \geq 1 \atop (m, p)=1} \frac{\overline{\xi}(m)}{m^s} \cdot  \frac{2}{\varphi(p^{\beta})}  \seven \chi(hq \overline{m}) \tau(\overline{\chi})^{n-1}. \end{align*}
Applying Lemma \ref{SOGS} (or Proposition \ref{QO} if $n=1$) to evaluate the inner sum, we then find that 
\begin{align*} \sum_{m \geq 1 \atop (m, p)=1} \frac{\overline{\xi }(m)}{m^s} \cdot  \frac{2}{\varphi(p^{\beta})} 
\seven \chi(hq \overline{m}) \tau(\overline{\chi})^{n-1}
&= \sum_{m \geq 1 \atop (m, p)=1} \frac{\overline{\xi}(m)}{m^{1-s}} \Kl_{n-1}(\pm m \overline{hq}, p^{\beta}) \end{align*}
if $n \geq 2$, and 
\begin{align*}  \sum_{m \geq 1 \atop (m, p)=1} \frac{\overline{\xi }(m)}{m^s} \cdot  \frac{2}{\varphi(p^{\beta})} \seven \chi(hq \overline{m}) 
&= \sum_{m \geq 1 \atop m \equiv \pm hq \m p^{\beta}} \frac{\overline{\xi}(m)}{m^{1-s}} - \frac{2}{\varphi(p^{\beta})} \cdot \frac{\varphi(p^{\beta-1})}{2} 
\sum_{ {m \geq 1 \atop m \equiv \pm hq \m p^{\beta-1}} \atop m \not\equiv \pm hq \m p^{\beta}} \frac{\overline{\xi}(m)}{m^{1-s}} \end{align*}
if $n=1$. Substituting these expressions back into the previous (analytic continuation) formula for $\K_n^0(\xi, h, p^{\beta}, s)$, we then 
obtain for $\Re(s)<0$ (after analytic continuation) the stated additive functional identity 
\begin{align*} \K_n^0(\xi, h, p^{\beta}, s) &= q^{- s} p^{\beta(1 -s)} \xi(p^{\beta}) \tau(\xi) 
\left( \pi^{s - \frac{1}{2}} \frac{\Gamma \left( \frac{1-s}{2} \right)}{\Gamma \left( \frac{s}{2}\right)} \right)
\K_{n-1}^0(\overline{\xi}, \overline{hq}, p^{\beta}, 1-s). \end{align*}

Let us now show (ii), hence with $\beta =1$. Again we start with $s \in {\bf{C}}$ having $\Re(s) >1$, 
opening up the absolutely convergent Dirichlet series and applying Lemma \ref{SOGS} to obtain 
\begin{align*} \K_n^0(\xi, h, p, s) &= \sum_{m \geq 1 \atop (m, p)=1} \frac{\xi(m)}{m^s} \Kl_n(\pm mh, p) &= 
\frac{2}{p-3} \sum_{m \geq 1 \atop (m,p)=1} \frac{\xi(m)}{m^s} \left( \sevenp \chi(mh) \tau(\overline{\chi})^n +(-1)^n \right), \end{align*} 
which after switching the order of summation is the same as 
\begin{align*} \K_n^0(\xi, h, p, s) &= \frac{2}{p-3} \left( \sevenp \chi(h) \tau(\overline{\chi})^n L(s, \xi\chi) + (-1)^n \epsilon_p(s, \xi)L(s, \xi) \right). \end{align*}
Again, we write $\epsilon_p(s, \xi)^{-1}$ to denote the Euler factor at $p$ of $L(s, \xi)$, so that 
$\epsilon_p(s, \xi) L(s, \xi) = L^{(p)}(s, \xi)$ denotes the Dirichlet series with the Euler factor at $p$ removed. 
Applying the functional equations
\begin{align*} L(s, \xi \chi) &= (q p^{\beta})^{-s} \xi(p^{\beta}) \chi(q) \tau(\xi) \tau(\chi) 
\left( \pi^{s - \frac{1}{2}} \frac{\Gamma \left( \frac{1-s}{2} \right)}{\Gamma \left( \frac{s}{2}\right)} \right) L(1-s, \overline{\xi \chi}) \\
L(s, \xi) &= q^{-s} \tau(\xi) \left( \pi^{s - \frac{1}{2}} \frac{\Gamma \left( \frac{1-s}{2} \right)}{\Gamma \left( \frac{s}{2}\right)} \right) L(1-s, \overline{\xi}) \end{align*}
to this latter expression, we then obtain the identification 
\begin{align*} \K_n^0(\xi, h, p, s) &= q^{- s} \tau(\xi) 
\left( \pi^{s - \frac{1}{2}} \frac{\Gamma \left( \frac{1-s}{2} \right)}{\Gamma \left( \frac{s}{2}\right)} \right) \\ &\times \frac{2}{p-3} 
\left( p^{1 - s} \xi(p^{\beta}) \sevenp \chi(h q) \tau(\overline{\chi})^{n-1} L(1-s, \overline{\xi \chi}) + (-1)^n \epsilon_p(s, \xi) L(1-s, \overline{\xi})\right), \end{align*}
which is valid for all $s \in {\bf{C}}$ (again by the analytic continuation of the Dirichlet series $L(s, \xi\chi)$ and $L(s, \xi)$). 
Let us now assume that $\Re(s) < 0$. Hence, we can expand out the absolutely convergent Dirichlet series in this latter expression,
switching the order of summation to derive 
\begin{align*} \frac{2}{p-3} \sevenp \chi(hq) \tau(\overline{\chi})^{n-1} L(1-s, \overline{\xi \chi})
&= \frac{2}{p-3} \sum_{m \geq 1 \atop (m, p)=1} \frac{\overline{\chi}(m)}{m^{1-s}} \sevenp \chi(h q \overline{m}) \tau(\overline{\chi})^{n-1}. \end{align*}
If $n \geq 2$, then we can apply Lemma \ref{SOGS} to evaluate the inner sum so that 
\begin{align*} \frac{2}{p-3} \sevenp \chi(hq) \tau(\overline{\chi})^{n-1} L(1-s, \overline{\xi \chi}) &=
\sum_{m \geq 1 \atop (m, p)=1} \frac{\overline{\chi}(m)}{m^{1-s}} \left( \Kl_{n-1}(\pm mh, p) + (-1)^n \right). \end{align*}
If $n =1$, then we simply apply Proposition \ref{QO} to evaluate 
\begin{align*} \frac{2}{p-3} \sevenp \chi(hq) \tau(\overline{\chi})^{n-1} L(1-s, \overline{\xi \chi}) &= 
\sum_{m \geq 1 \atop m \equiv \pm hq \m p} \frac{\overline{\xi}(m)}{m^{1-s}} 
- \frac{2}{p-3} \sum_{m \geq 1 \atop m \not\equiv \pm hq \m p} \frac{\overline{\xi}(m)}{m^{1-s}}. \end{align*}
Substituting these expressions back into the previous formula for $\K_n^0(\xi, h, p, s)$ then proves the claim. \end{proof}

\begin{proof}[Proof of Theorem \ref{D} (B)] 

In either case, we expand for a suitable choice of real number $\sigma >1$, shifting the range of integration to $\Re(s) = -\sigma$:
\begin{align*} \sum_{m \geq 1 \atop (m,p)=1} \xi(m) \Kl_{n}(\pm mh, p^{\beta}) \phi(m) &= \int_{(\sigma)} \phi^*(s) \K_n^0(\xi, h, p^{\beta}, s) \frac{ds}{2 \pi i} 
=  \int_{(-\sigma)} \phi^*(s) \K_n^0(\xi, h, p^{\beta}, s) \frac{ds}{2 \pi i}. \end{align*}
Suppose first that $\beta \geq 2$. Applying the functional identity of Theorem \ref{D} (A) (i) to $\K_n(\xi, h, p^{\beta}, s)$ gives  
\begin{align*} \int_{(-\sigma)} \phi^*(s) \K_n^0(\xi, h, p^{\beta}, s) \frac{ds}{2 \pi i} 
&= \tau(\xi) \xi(p^{\beta}) p^{\beta} \int_{(-\sigma)} \phi^*(s) (q p^{\beta})^{-s} 
\left( \pi^{s-\frac{1}{2}} \frac{ \Gamma \left( \frac{1-s}{2} \right) }{ \Gamma \left( \frac{s}{2} \right) } \right) 
\K_{n-1}^0(\overline{\xi}, \overline{h q}, p^{\beta}, 1-s) \frac{ds}{2 \pi i}, \end{align*}
which after expanding the absolutely convergent Dirichlet series $\K_{n-1}^0(\overline{\chi}, \overline{hq}, p^{\beta}, 1-s)$ equals
\begin{align*} \tau(\xi) \xi(p^{\beta}) p^{\beta} \sum_{m \geq 1 \atop (m,p)=1} \frac{\overline{\xi}(m)}{m} \Kl_{n-1}(\pm m\overline{hq}, p^{\beta})
\int_{(-\sigma)} \phi^*(s) \left( \pi^{s-\frac{1}{2}} \frac{ \Gamma \left( \frac{1-s}{2} \right) }{ \Gamma \left( \frac{s}{2} \right) } \right) 
\left( \frac{m}{q p^{\beta}}\right)^s \frac{ds}{2 \pi i}. \end{align*}
This shows (i). For $\beta =1$, we apply Theorem \ref{D} (A) (i) to $\K_n^0(\xi, h, p, s)$ to find
\begin{align*}  \int_{(-\sigma)} \phi^*(s) \K_n^0(\xi, h, p, s) \frac{ds}{2 \pi i} 
&= \tau(\xi) \xi(p) p \int_{(-\sigma)} \phi^*(s) \left( \pi^{s - \frac{1}{2}} \frac{\Gamma \left( \frac{1-s}{2} \right)}{\Gamma \left( \frac{s}{2} \right)} \right) 
(qp)^{-s} \K_{n-1}^0(\overline{\xi}, \overline{hq}, p, 1-s) \frac{ds}{2 \pi i} \\
&+(-1)^n \tau(\xi) \int_{(-\sigma)} \phi^*(s) \left( \pi^{s - \frac{1}{2}} \frac{\Gamma \left( \frac{1-s}{2} \right)}{\Gamma \left( \frac{s}{2} \right)} \right) 
q^{-s} L^{(p)}(1-s, \overline{\xi}) \frac{ds}{2 \pi i} \\
&+(-1)^n \tau(\xi) \frac{2}{p-3} \int_{(-\sigma)} \phi^*(s) \left( \pi^{s - \frac{1}{2}} \frac{\Gamma \left( \frac{1-s}{2} \right)}{\Gamma \left( \frac{s}{2} \right)} \right) 
q^{-s} \epsilon_p(s, \overline{\xi}) L^{(p)}(1-s, \overline{\xi}) \frac{ds}{2 \pi i}, \end{align*} 
which after expanding out the absolutely convergent Dirichlet series is the same as 
\begin{align*}  \tau(\xi) \xi(p) p &\sum_{m \geq 1 \atop (m,p)=1)} \frac{\overline{\xi}(m)}{m} \Kl_{n-1}(\pm m \overline{hq}, p)
\int_{(-\sigma)} \phi^*(s) \left( \pi^{s - \frac{1}{2}} \frac{\Gamma \left( \frac{1-s}{2} \right)}{\Gamma \left( \frac{s}{2} \right)} \right) 
\left(  \frac{m}{qp}\right)^s \frac{ds}{2 \pi i} \\
&+(-1)^n \tau(\xi) \sum_{m \geq 1 \atop (m,p)=1} \frac{\overline{\xi}(m)}{m}
\int_{(-\sigma)} \phi^*(s) \left( \pi^{s - \frac{1}{2}} \frac{\Gamma \left( \frac{1-s}{2} \right)}{\Gamma \left( \frac{s}{2} \right)} \right) 
\left( \frac{m}{q} \right)^s \frac{ds}{2 \pi i} \\
&+(-1)^n \tau(\xi) \frac{2}{p-3} \sum_{m \geq 1 \atop (m,p)=1} \frac{\overline{\xi}(m)}{m}
\int_{(-\sigma)} \phi^*(s) \left( \pi^{s - \frac{1}{2}} \frac{\Gamma \left( \frac{1-s}{2} \right)}{\Gamma \left( \frac{s}{2} \right)} \right) 
\epsilon_p(s, \overline{\xi}) \left(\frac{m}{q} \right)^s \frac{ds}{2 \pi i}. \end{align*} \end{proof}

\end{document}